\documentclass[12pt]{article}
\title{The many faces of the stochastic zeta function}
\usepackage{palatino}
\date{}
\author{Benedek Valk\'o and B\'alint Vir\'ag}

\usepackage{amsmath, amsthm, amssymb}
\usepackage{graphicx}
\usepackage{amsmath, enumerate}
\usepackage{color, url}

\usepackage{hyperref}
    \newtheorem{theorem}{Theorem}
    \newtheorem{lemma}[theorem]{Lemma}
    \newtheorem{proposition}[theorem]{Proposition}
    \newtheorem{corollary}[theorem]{Corollary}

        \newtheorem{conjecture}[theorem]{Conjecture}

\theoremstyle{definition} 
    \newtheorem{definition}[theorem]{Definition}
    \newtheorem{remark}[theorem]{Remark}
    \newtheorem{example}[theorem]{Example}
    \newtheorem{assumption}{Assumption}
    
\newcommand{\eps}{\varepsilon}
\newcommand{\Z}{{\mathbb Z}}
\newcommand{\ZZ}{{\mathbb Z}}

\newcommand{\R}{{\mathbb R}}
\newcommand{\CC}{{\mathbb C}}

\newcommand{\HH}{{\mathbb H} }
\newcommand{\ev}{{\rm   E}}

\newcommand{\lstar}{{\raise-0.15ex\hbox{$\scriptstyle \ast$}}}

\theoremstyle{remark} 
\newcommand{\Sineb}{\operatorname{Sine}_{\beta}}
\newcommand{\Sine}{\operatorname{Sine}}

\newcommand{\Sineop}{\mathtt{Sine}_{\beta}}
\newcommand{\Circop}{\mathtt{Circ}_{\beta,n}}

\newcommand{\Dirop}{\mathtt{Dir}}

\definecolor{violet}{rgb}{0.8,0,0.2}
\newcommand{\cd}{\stackrel{d}{\longrightarrow}}
\newcommand{\ed}{\stackrel{d}{=}}

\newcommand{\cB}{{\mathcal B}}

\newcommand{\cA}{{\mathcal A}}

\newcommand{\cE}{{\mathcal E}}
\newcommand{\cH}{{\mathcal H}}
\newcommand{\cL}{{\mathcal L}}
\newcommand{\cG}{{\mathcal G}}
\newcommand{\btau}{{\boldsymbol{\tau}}}

\newcommand{\Tr}{{\operatorname{Tr\,}}}
\newcommand{\mat}[4]{\left( \hspace{-0.4em} \begin{array}{cc}
#1 & #2  \\
#3 & #4  \\
\end{array} \hspace{-0.4em}\right)}
\newcommand{\bin}[2]{\binom{#1}{#2}}

\newcommand{\ac}{{\text{\sc ac}}}
\newcommand{\dom}{\operatorname{dom}}

\newcommand{\tr}{\operatorname{tr}}
\newcommand{\ind}{\mathbf 1}

\newcommand{\benedek}[1]{\textcolor{red}{\tt{#1}}}

\newcommand{\szeta}{\zeta_\beta}

\newcommand{\res}{{\mathtt{r}\,}}

\newcommand{\q}{q}

\newcommand{\bside}{\noindent\textbf{\benedek{Begin side computation.}}
\begin{footnotesize}}

\newcommand{\eside}{\end{footnotesize}
\noindent \textbf{\benedek{End side computation.}}}

\newcommand{\sech}{\operatorname{sech}}

\begin{document}
\maketitle
\begin{abstract}
We introduce a framework to study the random entire function $\zeta_\beta$ whose zeros are given by the Sine$_\beta$ process, the bulk limit of beta ensembles. We present several equivalent characterizations, including an explicit power series representation built from Brownian motion. 

We study related distributions using stochastic differential equations. Our function is a uniform limit of characteristic polynomials in the circular beta ensemble; we give upper bounds on the rate of convergence.  Most of our results are new even for classical values of $\beta$.

We provide explicit moment formulas for $\zeta$ and its variants, and we show that the Borodin-Strahov moment formulas hold for all $\beta$ both in the limit and for circular beta ensembles. We show a uniqueness theorem for $\zeta$ in the Cartwright class, and deduce some product identities between conjugate values of $\beta$. The proofs rely on the structure of the $\Sineop$ operator to express $\zeta$ in terms of a regularized determinant. 
\end{abstract}

\tableofcontents

\section{Introduction}

The goal of this paper is to study a central object in random matrix theory: the large $n$ limit of characteristic polynomials. These limits are expected to be universal; here we study the universality class of general beta ensembles in the bulk. 

Relatively little is known about this limit, while the limit of the eigenvalues, the $\Sineb$ process, has been more extensively studied.  For $\beta=2$, Chhaibi, Najnudel and Nikhekgbali \cite{CNN} 
introduced the random entire function
$$
\zeta_2(z)=\lim_{r\to \infty} \prod_{\lambda\in \Sine_2,\; |\lambda|<r }(1-z/\lambda),
$$
and studied its properties using the special determinantal structure of this case. 
Our goal is to deepen the understanding of this function and its general $\beta$ versions by introducing a set of new tools. Among others, we give an explicit Taylor series expansion given in terms of Brownian motion. 

The  {\bf Cartwright class}  of entire functions is defined by the growth conditions $|f(z)|\le c^{1+ |z|}$ and  $\int_{\mathbb R}\log_+|f(x)|/(1+x^2) \,dx<\infty.$ It can be thought of as an infinite-dimensional version of  the class of polynomials. 

Let $b_1,b_2$ be independent copies of two-sided standard Brownian motion,  let $y_u=e^{b_2-u/2}$, let $q$ be an independent Cauchy random variable with density $1/(\pi(1+x^2))$.  Let $\cA_0\equiv 1, \cB_0 \equiv 0$, and define, recursively,
\begin{equation}\label{eq:AnBn_intro}
\begin{aligned}
\cB_{n,u}&=y_u\int_{-\infty}^u \tfrac{\beta}{8}\,e^{\beta s/4 }\cA_{n-1,s}/y_s\,ds,\\
\cA_{n,u}&=\int_{-\infty}^u \tfrac{\beta}{8}\, e^{\beta s/4}\,\cB_{{n-1},s}\,ds- \int_{-\infty}^u\cB_{n,s} \, db_1.
\end{aligned}
\end{equation}

\begin{theorem}[The stochastic zeta function]\label{t:main}
There exists a unique probability measure  on the Cartwright class of  entire functions $f$ with $f(0)=1$, $f(\R)\subset \R$, so that the law of zeros is given by the $\Sineb$ process. 

The corresponding random entire function, {\bf the stochastic zeta function} $\zeta_\beta$ has several explicit representations. As a  power series with infinite radius of convergence, 
\begin{equation}\label{e:zeta-intro}
\zeta_\beta(z)=\sum_{n=0}^\infty(\cA_{n,0}-q \cB_{n,0})z^n.
\end{equation}
As the principal value infinite product
$$
\zeta_\beta(z)=\lim_{r\to \infty} \prod_{\lambda\in \Sineb,\; |\lambda|<r }(1-z/\lambda).
$$
As $\zeta_\beta=[1,-q]\cH_0$, where $\cH_u(z)$ is the unique analytic  solution of the system of stochastic differential equations 
\begin{align}\label{HSDE-intro}
d \cH= \mat{0}{-db_1}{0}{db_2} \cH-z\tfrac{\beta}{8}e^{\beta u/4} J \cH \, du, \qquad u\in \mathbb R, \qquad \lim_{u\to -\infty} \sup_{z<1}\Big|\cH_u(z)-\binom{1}{0}\Big|=0.
\end{align}
\end{theorem}
The claims of Theorem \ref{t:main} are proved in Propositions \ref{prop:zetaunique}, \ref{prop:Taylor_1}, \ref{prop:HSDE} and in Section \ref{subs:C_class}, see equation \eqref{eq:zeta_prod}. Most of the underlying theory is developed in Sections 2-5  of this paper.  

The name, stochastic zeta function, is motivated by  \cite{CNN}, and   their analogue of the Montgomery conjecture about the Riemann zeta function $\boldsymbol{\zeta}$.
\begin{conjecture}[Chhaibi, Najnudel and Nikhegbali \cite{CNN}]
For $\omega$ uniform on $[0,1]$ as $\nu\to \infty$,  the following distributional convergence of random analytic functions holds:
$$
\frac{\boldsymbol{\zeta}\left(\tfrac12+ie^\nu \omega  + i z/\nu\right)}{
\boldsymbol{\zeta}\left(\tfrac12+ie^\nu\omega  \right)}\;\cd \;\zeta_2(z).
$$
\end{conjecture}
Properties of $\boldsymbol{\zeta}$ have been conjectured based on $\zeta_2$, see Sodin \cite{sodin2018critical}.
The function $\zeta_2$ is the limit of characteristic polynomials, see \cite{CNN} and Chhaibi, Hovhannisyan, Najnudel, Nikeghbali,   and Rodgers \cite{chhaibi2019limiting}. We show that $\zeta_\beta$ is the limit in the case of circular beta ensembles, Section \ref{subs:char_circ_beta}, with an explicit rate of convergence. 

\begin{theorem}\label{thm:intro-limit} There is a coupling of the characteristic polynomials $$p_{n}(z)=\det(I-zU_\beta^{-1})=\prod_{i=1}^n(1-z/\lambda_i)$$ 
of the circular beta ensemble, $\zeta_\beta$, and a random $C$ so that for all $z\in \mathbb C$, $n>1$ we have
\begin{align}
|p_n(e^{i z/n})e^{-i z/2}-\zeta_\beta(z)| \le  \Big(e^{|z| \frac{\log^3 n}{\sqrt{n}}}-1\Big)
C^{|z|^2+1}.\label{circ_cp_bnd-intro}
\end{align}
\end{theorem}

The starting point of our analysis is the general framework of Dirac differential operators $\btau$, see Section \ref{s:secular}. In  this setting, we introduce the {\bf secular function}, the analogue of the characteristic polynomial. We apply this theory to the $\Sineop$ operator introduced in \cite{BVBV_op} and its conjugate $\btau_\beta$, see Section \ref{s:sineop}. The eigenvalues of these operators are given by the $\Sineb$ process.

For trace class $\btau^{-1}$, a natural  choice for the secular function would be $\det(I-\btau^{-1}z)$. However, the  eigenvalues of $\btau_\beta$ are given by the stationary process $\Sineb$, so their inverses are not summable.   In this case the  regularized determinant $\det_2$ can be used. For trace class operators, the ordinary and regularized determinant are related by 
$$
\det(I-\btau^{-1}z)={\det}_2(I-\btau^{-1}z)e^{z\tr \btau^{-1}},
$$
which motivates the following theorem. We consider the integral operator $\res\btau_\beta$, a conjugate of the inverse of the $\Sineop$ operator. Interestingly enough, the integral trace  ${\mathfrak t}\,\res\btau_\beta$ can be defined as the diagonal integral of the kernel, see \eqref{eq:int_tr}, and it agrees with the principal value trace \eqref{eq:traceid} defined as
$$\lim_{r\to \infty} \sum_{|\lambda_k|\le r} \frac{1}{\lambda_k}.
$$ The principal value  trace has Cauchy distribution, a general phenomenon for translation-invariant processes, see Theorem \ref{t:AW} and Aizenman and Warzel \cite{AW_2014}.
\begin{theorem}
\begin{equation}\label{eq:zeta-det-intro}
\zeta_\beta(z) = {\det}_2(I-z\,\res \btau_\beta)e^{z{\mathfrak t}\, \res \btau_\beta}.
\end{equation}
\end{theorem}
The equivalence to \eqref{e:zeta-intro} is proved in Proposition \ref{prop:Taylor_1}. 
In fact, we will use \eqref{eq:zeta-det-intro} as the definition of the {\bf secular function} $\zeta$ for general Dirac operators, see Definition \ref{def:zetab}. Characterizations similar to  Theorem \ref{t:main} hold in this setting, see Section \ref{s:secular}.

The $\Sineop$ operator is built from  hyperbolic Brownian motion, which allows us to provide additional characterizations of $\zeta_\beta$ via a system of stochastic differential equations \eqref{HSDE-intro}. We use this characterization to compute  expectations of quantities related to  $\zeta_\beta$. As an application, we verify a conjecture of Borodin and Strahov \cite{BorodinStrahov} in this setting, see Section \ref{subs:BS_zeta}.
\begin{theorem}\label{thm:BS_conj-intro} For all $\beta>0$ we have
\begin{align} 
E \prod_{j=1}^k \frac{\zeta_\beta(z_j)}{\zeta_\beta(w_j)} =\begin{cases}
e^{i \sum_{j=1}^k \frac{z_j-w_j}{2}}&\qquad \text{ if } \Im w_j<0 \text{ for all } j,\\
e^{-i \sum_{j=1}^k \frac{z_j-w_j}{2}}&\qquad \text{ if } \Im w_j>0  \text{ for all } j.
\end{cases}
\end{align}
\end{theorem}

In Theorem \ref{thm:finite-BStr} we show that the corresponding moment formulas also hold for circular beta ensembles.  In fact, Theorem \ref{thm:finite-BStr} applies in greater generality: for all models with independent, rotationally invariant Verblunsky coefficients. This includes well-studied random Schr\"odinger models for which few exact formulas are known, see Simon \cite{OPUC2}, Section 12.6. 

Because of the Cauchy variable in the normalization \eqref{e:zeta-intro},  $\zeta_\beta$ itself does not to have a first moment. A better choice for moments is the function 
$\hat \zeta_{\beta}(z)=(1+q^2)^{-1/2}\zeta_\beta(z)$, with $q$ as in \eqref{e:zeta-intro}. We have $$
E\hat \zeta_\beta(z)=\tfrac{2}{\pi}\cos(z/2), \qquad \beta>0, z\in \mathbb R,
$$
see \eqref{e:hatzeta-first}. We show that the moments of $\hat \zeta_\beta$  on the real line exist below  $1+\tfrac{\beta}{2}$, which is likely optimal. We have the following explicit formula for the second moment in terms of the generalized cosine integral:
$$
E\hat \zeta_\beta(z)^2=    \tfrac{1}{2}+\tfrac{2}{\beta}\int_0^1 t^{-4/\beta-1}(1-\cos(z t))dt, \qquad \beta>2, \;z\in \R,
$$
see \eqref{e:hatzeta-second}. More generally, the moments satisfy an explicit system of ordinary differential equations, see Section \ref{s:moments}. The equations are similar to those obtained by Killip and Ryckman \cite{KillipRyckman} for the circular beta ensemble. Their scaling is different from the one appearing in the distributional convergence of characteristic polynomials.  

The law of the function $\zeta_\beta$ can be represented as the stationary distribution of a system of diffusions, Proposition \ref{prop:H_stat}.  Surprisingly, the first degree Taylor coefficient $\cB_{1,0}$ 
satisfies the stationary stochastic differential equation appearing in Dufresne's identity. Its distribution can be computed explicitly:
\[
\cB_{1,0}\ed \frac{\beta}{4\, G},
\]
where $G$ has Gamma distribution with rate 1 and shape parameter $1+\frac{\beta}{2}$, see Remark \ref{rem:Dufresne}. 
Similarly, the higher order coefficients $\cB_{n,0}, \cA_{n,0}$ satisfy closed systems of stochastic differential equations. They can be converted to partial differential equations for the joint densities. For given $n$, the joint density $p(x,y)$ of $\cB_{1,0}, \cA_{1,0} \dots, \cB_{n,0}, \cA_{n,0}$ satisfies the following partial differential equation:
\begin{align}\notag 
     \sum_{k=1}^n \;\tfrac{1}{2}(\partial^2_{x_k}+\partial^2_{y_k})(x_k^2 p)-\tfrac{\beta}{8} \partial_{x_k}\left((y_{k-1}-2k x_k)p\right)-
    \tfrac{\beta}{8} \partial_{y_k}\left((x_{k-1}-2k y_k)p\right)=0,
\end{align}
where $x_0=0, y_0=1$, see Proposition \ref{prop:StationaryTaylor}.

Especially interesting is the function 
\begin{equation}\label{e:B-intro}\notag
\cB(z)=\sum_{n=0}^\infty\cB_{n,0}z^n.
\end{equation}
built from the coefficients $\cB_{n,0}$.
In \cite{BVBV_future} we show that the zero distribution of $\cB$ is the Palm measure of the $\Sineb$ process:  the process conditioned to have a point at zero. In particular, the intensity of zeros of $\cB$ is given by the two-point correlations of $\Sineb$. 

The uniqueness of $\zeta_\beta$ within the Cartwright class and results of Forrester \cite{ForresterDec} imply simple product identities connecting various $\beta$-s. 
Let $k\ge 2$ an integer. There exists a coupling of $k$ copies $\zeta_{2k,1},\ldots, \zeta_{2k,k}$ of $\zeta_{2k}$
so that
$$
\zeta_{2/k}(z)\ed \prod_{j=1}^k\zeta_{2k,j}(z/k),
$$
see Corollary \ref{c:forrester}. The way $\szeta$ is the secular function of the $\Sineb$ operator has analogues for more classical differential operators. The sine function and versions of the Bessel functions appear as secular functions. We illustrate the development of $\szeta$ with a series of concrete, classical examples, see Examples \ref{ex:bulk}, \ref{ex:edge}, \ref{ex:bulk_2}, \ref{ex:edge_2}, \ref{ex:bulk_3}, \ref{ex:edge_3}, and \ref{ex:part_4}. 

Many of our results rely on understanding the so called de Branges {\bf structure function}
$$
\cE(z)=\sum_{n=0}^\infty(\cA_{n,0}-i \cB_{n,0})z^n  = [1,-i]\cdot \cH_0(z)
$$
an analogue of orthogonal polynomials on the unit circle in the infinite-dimensional setting. For Dirac operators based on discrete circular ensembles, $\cE$ is a linear combination of orthogonal polynomials, see Proposition \ref{prop:unitary_car2b}, \eqref{e:orthogonal}. The entire function $\cE$  belongs to the so-called P\'olya class, and all of its zeros are in the upper half plane. In de Branges \cite{deBranges}  such functions are used to define Hilbert spaces of entire functions, the starting point of the theory of canonical systems. For us, $\cA$, $\cB$, $\zeta_\beta$ and $\hat \zeta_\beta$ can be expressed in terms of $\cE$ and $q$, and the moment formulas are simplest for the function $\cE$. In particular, as shown in Proposition \ref{prop:E-moments}, for $x_j\in \R$ we have
$$
E\,\prod_{j=1}^k \cE(x_j)=\prod_{j=1}^k e^{ix_j/2}, \qquad k<1+\tfrac{\beta}{2}.
$$

\section{Secular functions of Dirac operators}\label{s:secular}

Our approach is to understand the function $\zeta_\beta$ as the ``characteristic polynomial'' of an ``infinite matrix''. More precisely, it is the secular function of a random Dirac operator. We first introduce such secular functions for a general class of deterministic Dirac operators. 

\subsection{A class of Dirac operators}

Let $\sigma>0$ and $R:(0,\sigma]\to \mathbb R^{2\times 2}$ be a function  taking values in nonnegative definite matrices. In this paper, a {\bf Dirac operator} is defined as   
\begin{equation}\label{tau}
{\boldsymbol{\tau}}u=R^{-1} J u', \qquad J=\mat{0}{-1}{1}{0},
\end{equation}
acting on some subset of functions of the form $u:(0,\sigma]\to \R^2$.

We assume the following about $R$.

\begin{assumption}

$R(t)$ is positive definite for all $t\in (0,\sigma]$, $\|R\|, \|R^{-1}\|$ are locally bounded on $(0,\sigma]$. Moreover, $\det R(t)=1/4$ for all $t\in (0,\sigma]$.
\end{assumption}

The assumption $\det R=1/4$ can be replaced by the more general condition $\int_0^{\sigma} \det R(s)\, ds<\infty$. This setting is equivalent to ours up to a time change. 

\begin{assumption}
There is a nonzero vector $\mathfrak  u_0  \in \R^2$ so that with  $\mathfrak  u_0  ^{\perp}=J\mathfrak  u_0  $ we have 
\begin{align}
\int_0^{\sigma} \left\|\mathfrak  u_0  ^t R \right\|ds<\infty,\qquad \int_0^{\sigma} \int_0^t \mathfrak  u_0  ^t R(s) \mathfrak  u_0   \, (\mathfrak  u_0  ^{\perp})^t R(t) \mathfrak  u_0  ^{\perp} ds dt<\infty.\label{eq:assmpns}
\end{align}
\end{assumption}
The function $R$ can be parametrized as
\begin{align}
R=\frac{X^t X}{2\det X}  , \qquad X=\mat{1}{-x}{0}{y}, \qquad y>0, x\in \R,\label{eq:Rxy}
\end{align}
where $X$ is defined as the unique multiple of the square root of $R$ of the above  form. We refer to the group of matrices of such form as the affine group. We record the matrix identities
\begin{equation}\label{eq:x-conj}
yX^{-1}J=JX^t,\qquad  yJX^{-1}=X^tJ, \qquad R^{-1}J=2X^{-1}JX, \qquad 2JR=X^{-1}JX,
\end{equation}
that hold in this setting.

Let  $\ac$ be the set of absolutely continuous functions. Define   $L^2_R=L^2_R[0,\sigma]$ as the space with norm defined by
\begin{equation}\label{eq:L2R}
    \|f\|_R^2=\int_0^\sigma f^t(s) R(s)f(s)\,ds.
\end{equation}
Fix nonzero $\mathfrak  u_0  ,\mathfrak  u_1 \in \R^2$. From this point on, we will consider $\btau$ as the Dirac operator with the  domain
\begin{align}
\dom_\btau=\{v\in L^2_R \cap  \ac\,:  \, \btau v\in L^2_R, \,\lim_{s\downarrow 0} v(s)^t\,J\,\mathfrak  u_0   = 0, \,\,v(\sigma)^t\,J\,\mathfrak  u_1 =0\}.
\end{align}
We will also use the notation 
\begin{align}\label{eq:Dirop}
    \btau=\Dirop(X_{\cdot},\mathfrak  u_0  ,\mathfrak  u_1 ).
\end{align}
We will often have the following. 
\begin{assumption}
\begin{equation}\label{unorm}
    \mathfrak  u_0  ^t J \mathfrak  u_1 =1.
\end{equation}
\end{assumption}

In most of the paper, we will use special cases of $\mathfrak  u_0  $, $\mathfrak  u_1 $  satisfying  Assumption 3, namely 
\begin{align}\label{eq:qbndry}
\mathfrak  u_0  =\bin{1}{0}, \qquad \mathfrak  u_1 =\bin{-q}{-1}, \qquad     q\in \mathbb R,
\end{align}
which we call {\bf ${q}$-boundary conditions}.

\subsection{The inverse of a Dirac operator}

Under Assumptions 1-3, the operator $\btau$ has an inverse.  $\btau^{-1}$ is an integral operator on $L^2_R[0,\sigma]$  with kernel
\[
K(s,t)=\left(\mathfrak  u_0   \mathfrak  u_1 ^t \ind(s<t)+\mathfrak  u_1  \mathfrak  u_0  ^t \ind(s\ge t)\right),
\]
and by Theorem 9 of  \cite{BVBV_op}  it is Hilbert-Schmidt. Hence $\btau$ has a pure point spectrum. The eigenvalues are real, nonzero and have multiplicity one.  We label the eigenvalues 
in an increasing order by $\lambda_k, k\in \Z$ so that  $\lambda_{-1}<0<\lambda_0$.

Let $Y=X/\sqrt{\det X}$. The conjugate operator $Y \btau Y^{-1}$ is  self-adjoint on $\{v: Yv\in \dom_\btau\}$ with the same spectrum as $\tau$. We denote $(Y \btau Y^{-1})^{-1}$ by $\res \btau$.  This is an integral operator acting on $L^2[0,\sigma]$ with kernel
\begin{align}\label{resint}
K_{\res \btau}(s,t)=\tfrac{1}{2}\,a(s) c(t)^t \ind(s<t)\;+\;\tfrac{1}{2}\,c(s) a(t)^t \ind(s\ge t),   
\end{align}
where 
\begin{equation}\label{eq:ac}
a=\frac{X\mathfrak  u_0  }{\sqrt{\det X}},\qquad c=\frac{X\mathfrak  u_1 }{\sqrt{\det X}}.
\end{equation}
The first integral condition of Assumption 2 is equivalent to the  following pair of bounds:
\begin{align}\label{eq:int_cond}
\int_0^{\sigma} \|a(s)\|^2 ds<\infty, \qquad \int_0^{\sigma} \left|a(s)^t c(s) \right|ds<\infty.
\end{align}
With the $q$-boundary conditions \eqref{eq:qbndry} we have
\[
a(s)=\frac{1}{\sqrt{y(s)}}\bin{1}{0}, \qquad c(s)=\frac{1}{\sqrt{y(s)}}\bin{x(s)-q}{-y(s)},
\]
and the pair of bounds in \eqref{eq:int_cond} is equivalent to 
\begin{align}\label{eq:Asmp_2_1}
  \int_0^\sigma \frac{1+|x(s)|}{y(s)}\,ds<\infty.  
\end{align}
For $q$-boundary conditions, the second integral bound in Assumption 2 is given by 
\begin{align}\label{eq:Asmp_2_2}
\int_0^\sigma \int_0^t \frac{x^2(t)+y^2(t)}{y(s)y(t)} ds\, dt<\infty.
\end{align}

We introduce two simple examples.
\begin{example}[Deterministic Sine operator]\label{ex:bulk}
The simplest example is the operator $$\btau: u \to 2J \frac{d}{dt}u$$ on $(0,\sigma]$. In that case $R(s)=\frac{1}{2} I$, $x(s)=0$, and $y(s)=1$. Set
$$\mathfrak  u_0  =[1,0]^t, \qquad \mathfrak  u_1 =[-q,-1]^t,\qquad \text{ with } q=\cot(\theta/2).$$
Then
we have
$a(s)=[1,0]^t$, and $c(s)=[-q,-1]^t$.
The eigenvalues of $\btau$ are given by $\lambda_k=\frac{2 \pi k -\theta}{\sigma}$, $k\in \Z$, and the eigenfunction corresponding to $\lambda_k$ is $$\left[\cos(\tfrac{s(2k\pi -\theta)}{2\sigma}),-\sin(\tfrac{s(2k\pi -\theta)}{2\sigma})\right]^t.$$
\end{example}

 \begin{example}[Deterministic Bessel operator]\label{ex:edge}
Assume $\alpha>0$, and set $x(s)=0$ and $y(s)= s^{-\alpha}$. Then we have
\[
R(s)=\mat{s^{\alpha}}{0}{0}{s^{-\alpha}}, \qquad \btau u=2\mat{s^{-\alpha}}{0}{0}{s^{\alpha}}J \frac{d}{dt} u.
\]
The vector $\mathfrak  u_0  =[1,0]^t$ satisfies the integral conditions (\ref{eq:assmpns}), we set the other boundary condition as $\mathfrak  u_1 =[0,-1]^t$.

Then $a(s)=[s^{\alpha},0]^t$ and $c(s)=[0,s^{-\alpha}]^t$. 
Let  $\mathrm{J}_p$ be the Bessel function of the first kind with parameter $p$.  Then 
the eigenvalues of $\btau$ are given by $\pm \tfrac{2\gamma_k}{\sigma}, k\ge 1$, where $\gamma_k$ are  the positive zeros of  $\mathrm{J}_{\frac{\alpha -1}{2}}(\cdot)$.
The eigenfunction corresponding to eigenvalue $\lambda$ is given by
\begin{align}\label{eq:BesselH}
u(s)=\left(\tfrac{\lambda s}{4}\right)^{\frac{1-\alpha}{2}} \Gamma(\tfrac{\alpha +1}{2})\left[ \mathrm{J}_{\frac{\alpha -1}{2}}\left(\tfrac{s \lambda }{2}\right) , -s^{\alpha} \mathrm{J}_{\frac{\alpha +1}{2}}\left(\tfrac{s \lambda }{2}\right)   \right]^t.
\end{align}
Note that allowing $\alpha=0$ would give us Example \ref{ex:bulk}.

\end{example}

\subsection{Definition of the secular function $\zeta_\btau$}

Our goal is to define the analogue of the characteristic polynomial for the operator $\btau$. This will be an entire function $\zeta_\btau$ whose zeros are  the eigenvalues of $\btau$ and can be considered to be a version of  $\det(I-z\,  \btau^{-1})$. We will call $\zeta_\btau$ the secular function of $\btau$. The term secular function has been used in the past for the characteristic polynomial.


For a Hilbert-Schmidt operator $A$ with eigenvalues $\nu_k$  the regularized determinant
\begin{align}\label{eq:det2}
{\det}_2(I-z A)=\prod_k (1-z \,\nu_k)e^{z\, \nu_k}.
\end{align}
is an entire function in $z$ that vanishes exactly at the eigenvalues, see \cite{SimonTrace}, Chapter 9. The product is uniformly convergent on compact subsets of $\CC$. When $A$ is also trace class, then   $$\det(I-z A)=\prod_k (1-z \,\nu_k)$$ is a well defined entire function and the product is  absolutely convergent, see \cite{SimonTrace}, Chapter 3. Moreover, one has
\begin{equation}\label{eq:detdet2}
\det(I-z A)={\det}_2(I-z A) e^{-z \Tr A}.
\end{equation}

Our operator $\btau^{-1}$ or, equivalently $\res \btau$, is not a trace class operator, but the {\bf integral trace}
\begin{align}
\mathfrak{t}_\btau=\int_0^{\sigma} \Tr K_{\res \tau}(s,s) ds=\tfrac12 \int_0^{\sigma} a(s)^t c(s) ds
    \label{eq:int_tr}
\end{align}
is finite by Assumption 2. Under the $q$-boundary conditions \eqref{eq:qbndry} we have
\[
\mathfrak{t}_\btau=\int_0^{\sigma} \frac{x(s)-q}{2y(s)} \,ds.
\]

Note that if an integral operator with kernel of the form (\ref{resint}) is trace class, its trace is equal to its integral trace.

By analogy with formula \eqref{eq:detdet2}, since $\res\btau$ is Hilbert-Schmidt, we define
\begin{definition}[Secular function]
Assume that $\btau$ satisfies Assumptions 1-3. Then the secular function is defined as
\begin{align}\label{zeta}
\zeta_\btau(z)&=e^{-z{\mathfrak t}_\btau}{\det}_2(I-z \,\res \btau) \\&=e^{-\tfrac{z}{2} \int_0^{\sigma} a(s)^t c(s) ds} \prod_k (1-z/ \,\lambda_k)e^{z/ \lambda_k},\label{zeta2}
\end{align}
where $a$ and $c$ are defined in \eqref{eq:ac}. The infinite product is uniformly convergent on compact sets of $z$.
\end{definition}

The eigenvalues of $\res \btau$ are $\frac{1}{\lambda_k}$.
The infinite product $\prod_k(1-\tfrac{z}{\lambda_k})$ is not necessarily absolutely convergent, so  $\det(I-z \res \tau)$ might not be defined.  In Section \ref{subs:prod} we show that with some additional assumptions, $\zeta$ is equal  to the principal value product  
\[
\lim_{r\to \infty} \prod_{|\lambda_k|\le r}(1-z/\lambda_k).
\]

In the rest of the section we discuss other representations for $\zeta$: an explicit Taylor expansion and a representation by an ordinary differential equation.

\subsection{The Taylor expansion of $\zeta$}

We will work towards a more explicit expression for $\zeta_\btau$ in terms of $X,\mathfrak u_1,\mathfrak u_2$. The following proposition gives the  Taylor expansion of $\zeta$ with explicit coefficients.


\begin{proposition}\label{prop:Fred_exp}
Let $\btau$ satisfy Assumptions 1-3. Then the secular function  $\zeta_{\btau}$ \eqref{zeta} has the following Taylor expansion with infinite radius of convergence:
\begin{align}\label{eq:zeta_exp}
\zeta_\btau(z)&=1+\sum_{n=1}^\infty r_n  z^n,\qquad
r_n=-\hskip-20pt \iiint\limits_{0<s_1<s_2<\dots<s_n\le \sigma} \mathfrak  u_0  ^{t} R(s_1)JR(s_2)J\cdots R(s_n)  \mathfrak  u_1  ds_1\dots ds_n.
\end{align}
The coefficient $r_n$ can be bounded as
\begin{align}\label{coef_bnd0}
|r_n|\le \left(\int_0^{\sigma} |\mathfrak  u_0  ^t R(s) \mathfrak  u_1 | ds+ \int_0^{\sigma} \int_0^t \mathfrak  u_0  ^t R(s) \mathfrak  u_0   \mathfrak  u_1 ^t R(t) \mathfrak  u_1  ds dt\right)^n.
\end{align}
\end{proposition}

We start with a statement about the Fredholm expansion of regularized determinants of Hilbert-Schmidt integral operators with a finite diagonal integral.
\begin{lemma}\label{l:HStrace}
Suppose that $G$ is a bounded interval on $\R$,   $k:G^2\to \R$ is a measurable function with $\iint_{G^2}|k(s,t)|^2 ds dt<\infty$ and $\int_G |k(s,s)|ds<\infty$. Let  $Af(s)=\int_G k(s,t) f(s) ds$ be the Hilbert-Schmidt integral operator  acting on real valued $L^2(G)$ functions. Then
\begin{align}\label{eq:Fr_2}
{\det}_2(I+z A) e^{z \int_G k(s,s)ds}=1+\sum_{n=1}^\infty  \frac{z^n}{n!}  \int_{G^n}\det \left[k(t_i,t_j) \right]_{i,j=1}^n dt_1\dots dt_n,
\end{align}
where the $n$-variable integrals on the right are all finite and the series converges on $\CC$.

Moreover, we also have the bound
\begin{align}\label{coef_bnd}
 \int_{G^n} \left| \det \left[k(t_i,t_j) \right]_{i,j=1}^n \right|dt_1\dots dt_n\le n! \left(\int_G|k(s,s)|ds+\|A\|_2  \right)^n
\end{align}

\end{lemma}

Note that we do not require $k$ to be continuous near the diagonal.

\begin{proof}
Since $A$ is a Hilbert-Schmidt integral operator, the  classical theory, see \cite{GGK}, \cite{SimonTrace}, implies that
\begin{align}\label{det2}
{\det}_2(I+z A)=1+\sum_{n=2}^\infty \frac{z^n}{n!}  \int_{G^n}\det \left[k(t_i,t_j) \ind(i\neq j)\right]_{i,j=1}^n dt_1\dots dt_n,
\end{align}
with the series on the right converging for all $z\in \CC$. In particular we also have that
\begin{equation}\label{finite-integral}\int_{G^n}\det \left[k(t_i,t_j) \ind(i\neq j)\right]_{i,j=1}^n dt_1\dots dt_n<\infty, \qquad \text{for $n\ge 2$.}
\end{equation}

Note that for any $n\ge 1$, and $t_1, \dots, t_n\in G$  we have
\begin{align}\label{FR1}
\det[k(t_i,t_j)]_{i,j=1}^n=\sum_{\substack{B\subset\{1,\dots, n\}\\
|B|\ge 2
}}\det\left[k(t_i,t_j) \ind(i\neq j)\right]_{i,j\in B} \prod_{\substack{\ell\notin B\\ 1\le \ell \le n}} k(t_\ell,t_\ell).
\end{align}
This identity follows by expanding the determinant on the left and collecting the terms based on the number of fixed points in the permutations.
Integrating the quantity \eqref{FR1} on $G^n$ we get
\begin{align}\label{qq1}
\sum_{\ell=2}^n  \binom{n}{\ell}\int_{G^\ell} \det\left[k(t_i,t_j) \ind(i\neq j)\right]_{i,j\le \ell} dt_1\dots dt_\ell \left( \int_G k(t,t) dt\right)^{n-\ell}.
\end{align}
This expression  is finite by \eqref{finite-integral}, and the assumption $\int_G |k(s,s)|ds<\infty$. The claim \eqref{eq:Fr_2} follows after multiplying  the Taylor expansions of the entire functions ${\det}_2(I+z A)$ and $e^{z \int_G k(s,s)ds}$.

To prove the bound (\ref{coef_bnd}) we first show that for $\ell\ge 1$ one has
\begin{align}\label{cycle_bnd}
\int_{G_n} |k(t_1, t_2) k(t_2, t_3) \cdots k(t_\ell, t_1)|dt_1\dots dt_\ell\le \left( \int_G |k(s,s)| ds+\|A\|_2\right)^{\ell}.
\end{align}
For $\ell=1$ this follows from the assumption. Define $B$ by  $$B f(x)=\int_G |k(x,y)| f(y) dy.$$
For $\ell\ge 2$ the integral operator $B^\ell$ is trace class, and its trace is equal to the left side of \eqref{cycle_bnd}. Moreover,  $\Tr(B^\ell)\le \Tr(B^2)^{\ell/2}=\|A\|_2^\ell$. This shows \eqref{cycle_bnd} for $\ell\ge 2$.
The bound (\ref{coef_bnd}) follows by expanding the determinant $\det \left[k(t_i,t_j) \right]_{i,j=1}^n$, and applying (\ref{cycle_bnd}) the the cycles of the permutation $\pi$ in $\int_{G^n} \prod_{i=1}^n |k(s_i,s_{\pi(i)}) |ds_1\dots ds_n$.
\end{proof}

\begin{proof}[Proof of Proposition \ref{prop:Fred_exp}]
We start by adapting  Lemma \ref{l:HStrace} to integral operators with matrix valued kernels. We can embed the space of $\R^2$-valued $L^2(0, \sigma]$ functions into the space of real valued $L^2(0,2 \sigma]$ functions using the following invertible isometry:
\begin{align}g=[g_1,g_2]^t, \qquad
\mathcal{J}g(t)=\begin{cases}
g_1(s), \quad& 0<s\le  \sigma\\
g_2(s- \sigma), \quad & \sigma<s\le 2 \sigma.
\end{cases}
\end{align}
If $B$ is a Hilbert-Schmidt integral operator acting on $\R^2$-valued functions on $(0, \sigma]$ with kernel $K=\mat{K_{11}}{K_{12}}{K_{21}}{K_{22}}$ then $A=\mathcal{J} B \mathcal{J}^{-1}$ is a Hilbert-Schmidt integral operator acting on  scalar functions on $(0,2 \sigma]$, and the integral kernel is
\begin{align}\label{eq:kK}
k(s,t)=K_{1+\ind(s\ge  \sigma), 1+\ind(t\ge  \sigma)}(s,t).
\end{align}
We have $\int_0^{ \sigma} \|K(s,t)\|^2_2 ds\, dt=\int_0^{2 \sigma} |k(s,t)|^2 ds\, dt$ and $\int_0^{2 \sigma} k(s,s)ds=\int_0^{ \sigma}  \Tr K(s,s) ds$. Assuming that both of these integrals are finite we may apply Lemma \ref{l:HStrace} to the integral operator $A$. Since the spectrum of $A$ is the same as the spectrum of $B$, we have ${\det}_2(I-z B)={\det}_2(I-z A)$ and hence
\begin{align}\label{sorfejtes}
{\det}_2(I-z B) e^{-z \int_0^{ \sigma}  \Tr K(s,s)}=1+\sum_{n=1}^\infty  (-1)^n \frac{z^n}{n!}  \int_{(0,2 \sigma]^n}\det \left[k(t_i,t_j) \right]_{i,j=1}^n dt_1\dots dt_n.
\end{align}
From (\ref{eq:kK}) we have
\begin{align*}
&\int_{(0,2 \sigma]^n}\det \left[k(t_i,t_j) \right]_{i,j=1}^n dt_1\dots dt_n=\sum_{i_1=1}^2 \cdots \sum_{i_n=1}^2 \int_{(0, \sigma]^n} \det(K_{i_a, i_b}(s_a,s_b))_{a,b=1}^n ds_1\dots ds_n,
\end{align*}
Now set $B=\res  \btau$. From (\ref{resint}) we get that the entries $K_{ij}$ of the matrix valued kernel are given by
\begin{align}\label{Kij}
K_{i,j}(s,t)=\tfrac{1}{2}\left(a_i(s) c_j(t)\ind(s<t)+c_i(s)a_j(t)\ind(t\le s)\right).
\end{align}
Note that $K(s,t)^t=K(t,s)$ and thus $K_{i,j}(s,t)=K_{j,i}(t,s)$. Because of this we have
\begin{align*}
&\frac{1}{n!}\sum_{i_1=1}^2 \cdots \sum_{i_n=1}^2 \int_{(0, \sigma]^n} \det(K_{i_j, i_\ell}(s_j,s_\ell))_{j,\ell=1}^n ds_1\dots ds_n\\&\hskip100pt=\sum_{i_1=1}^2 \cdots \sum_{i_n=1}^2\, \, \iiint\limits_{0<s_1<\dots<s_n\le  \sigma}  \det(K_{i_j, i_\ell}(s_j,s_\ell))_{j,\ell=1}^n ds_1\dots ds_n.
\end{align*}
%

Fix $i_1, \dots, i_n\in \{1,2\}$ and $0<s_1<\dots s_n\le \sigma$. Introduce the temporary notation
\[
p_k=a_{i_k}(s_k),\qquad q_k=c_{i_k}(s_k),
\]
then by \eqref{Kij} we have 
$2 K_{i_j, i_\ell}(s_j,s_\ell)= p_{\min(j,\ell)} \cdot  q_{\max(j,\ell)}$.
For example, for $n=3$ we have
\[
(2K_{i_j, i_\ell}(s_j,s_\ell))_{j,\ell=1}^n=\left(
\begin{array}{ccc}
 p_1 q_1 & p_1 q_2 & p_1 q_3  \\
 p_1 q_2 & p_2 q_2 & p_2 q_3 \\
 p_1 q_3 & p_2 q_3 & p_3 q_3
\end{array}
\right).
\]
We show that
\begin{align}\label{pq_det}
\det( p_{\min(j,\ell)} \cdot  q_{\max(j,\ell)})_{j,\ell=1}^n=p_1 q_n\prod_{j=1}^{n-1} (p_{j+1} q_{j}-p_{j}q_{j+1}).
\end{align}
Subtract row $n-1$ times $q_{n}/q_{n-1}$ from row $n$. Then the last row becomes $$[0,\dots,0, p_nq_n-p_{n-1} q_n^2/q_{n-1}].$$ The identity \eqref{pq_det} now follows by induction.

Note that $p_{j+1} q_{j}-p_{j}q_{j+1}=[p_j,q_j]J [p_{j+1},q_{j+1}]^t$, hence, with $v_{i_k}(s_k)=[p_k,q_k]^t$ we have
\begin{align}\label{111}
\det( p_{\min(j,\ell)} \cdot  q_{\max(j,\ell)})_{j,\ell=1}^n=\left[v_{i_1}(s_1) v_{i_1}(s_1)^t J v_{i_2}(s_2) v_{i_2}(s_2)^t J\cdots v_{i_n}(s_n) v_{i_n}(s_n)^t J\right]_{1,1}.
\end{align}
Note that
$$
v_{1}(s)v_1(s)^t+v_2(s)v_2(s)^t=2U^t R(s)U, \qquad  U=[\mathfrak  u_0  ,\mathfrak  u_1 ].
$$
Summing \eqref{111} for all choices of $i_1, \dots, i_n$ gives
\[
2^n\left[U^t R(s_1) U J U^t R(s_2) U J\cdots U^t R(s_n) U J\right]_{1,1}=-(-2)^n \mathfrak  u_0  ^t R(s_1)JR(s_2)J\dots R(s_n)\mathfrak  u_1 .
\]
In the last step we used $ U J U^t=-J$, which is equivalent to the assumption \eqref{unorm}.

The statement of the proposition now follows from \eqref{sorfejtes}.
\end{proof}

\begin{example}[Deterministic Sine operator, continued]\label{ex:bulk_2}
Consider $\btau$ from Example \ref{ex:bulk}. The eigenvalues are $\frac{2\pi k-\theta}{\sigma}, k\in \Z$, and $a(s)^t c(s)=-q=-\cot(\theta/2)$.
Definition (\ref{zeta}) gives
\begin{align*}
 \zeta(z)=   e^{\frac{\sigma z}{2}\cot(\theta/2)}\prod_{k} \left(1-\tfrac{\sigma z}{2\pi k-\theta}\right) e^{ \tfrac{\sigma z}{2\pi k-\theta}}.
\end{align*}
It is an exercise in complex analysis to show that
\begin{align}\label{zeta_Ex1}
\zeta(z)=
\frac{\sin \left(\tfrac{\theta +\sigma z}{2}\right)}{\sin(\theta/2)}.
\end{align}
This also follows from the power series representation of
Proposition \ref{prop:Fred_exp}:
\[
\zeta(z)=1+\sum_{n=1}^\infty r_n z^n , \qquad r_n=-2^{-n} \iiint_{0<s_1<\dots<s_n\le \sigma}\mathfrak  u_0  ^tJ^{n-1}\mathfrak  u_1  ds_1\dots ds_n.
\]
Using $J^{2k+1}=(-1)^k J$  and $J^{2k}=(-1)^k I$ we get
\[
r_{2k+1}=\frac{(\sigma/2)^{2k+1}}{(2k+1)!} (-1)^k q, \qquad r_{2k}=\frac{(\sigma/2)^{2k}}{(2k)!} (-1)^k
\]
which give the series expansion of $\zeta$, and proves  (\ref{zeta_Ex1}).
\end{example}

\begin{example}[Deterministic Bessel operator, continued]\label{ex:edge_2}
For $\btau$ from Example \ref{ex:edge} the definition (\ref{zeta}) gives
\begin{align}\label{Ex2_product}
    \zeta(z)=\prod_{k=1}^\infty \left(1-\tfrac{\sigma^2 z^2}{4\gamma_k^2}\right),
\end{align}
where $\gamma_k$ is the $k$th root of $\mathrm{J}_{\frac{\alpha-1}{2}}$. The infinite product representation of the Bessel function,  see 9.5.10 in \cite{AbSt} gives
\begin{align}\label{zeta_Ex2}
\zeta(z)=\Gamma(\tfrac{\alpha+1}{2})\left(\tfrac{\sigma z}{4}\right)^{-\frac{\alpha-1}{2}} \mathrm{J}_{\frac{\alpha-1}{2}} \left(\tfrac{\sigma z}{2}\right)=\, _0F_1\left(;\tfrac{\alpha +1}{2};-\tfrac{\sigma^2 z^2}{16}\right),
\end{align}
where $_0F_1(;a;w)=\sum _{k=0}^{\infty } \frac{ \Gamma(a)}{k!\Gamma(a+k)}w^k$ is the confluent hypergeometric function.

The series representation given by Proposition \ref{prop:Fred_exp} gives the coefficients
\[
r_{2n+1}=0, \qquad r_{2n}=(-1)^n 2^{-2n} \iiint_{0<s_1<\dots<s_{2n}\le \sigma} \tfrac{s_1^\alpha s_3^\alpha \cdots s_{2n-1}^\alpha}{s_2^\alpha s_4^\alpha\cdots s_{2n}^\alpha} ds_1\dots ds_{2n}.
\]
The multiple integral evaluates to 
\[
r_{2n}=(-1)^n 2^{-4n} \frac{\sigma^{2n} \Gamma(\frac{\alpha+1}{2})}{n! \Gamma(\frac{\alpha+1}{2}+n)},
\]
which agrees with (\ref{zeta_Ex2}) and the  series expansion of $_0F_1$.
\end{example}

\subsection{ODE representation for $\zeta$}

The function $\zeta$ can also be characterized using the solution of a vector-valued ODE, this is the statement of our next proposition.  
\begin{proposition}\label{prop:zetaODE}
Suppose that $R$ and $\mathfrak  u_0  $ satisfy Assumptions 1 and 2.
There is a  unique vector-valued function  $H: (0,\sigma]\times \CC\to \CC^2$ so that for every $z\in \CC$  the function $H(\cdot, z)$ is the solution of the ordinary differential equation
\begin{align}
\label{eq:E_1}
J\frac{d}{dt}H(t,z)&=z R(t) H(t,z), \qquad t\in(0,\sigma], \qquad\lim_{t\to 0} H(t,z)=  \mathfrak  u_0  .
\end{align}
For any $t\in (0,\sigma]$ the  function $H(t,z)$ satisfies $\|H(t,z)\|>0$, and its two coordinates are entire functions of $z$ mapping the reals to the reals.

Moreover for any $\mathfrak  u_1 $ with $\mathfrak  u_0  ^tJ\mathfrak  u_1 =1$ the corresponding $\zeta_{\btau}(z)$ satisfies
\begin{align}\label{zetaH}
\zeta_\btau(z)=H(\sigma,z)^tJ \mathfrak  u_1 .
\end{align}
\end{proposition}

We prove Proposition \ref{prop:zetaODE}  in two parts. Proposition \ref{prop:unique} shows that there is at most one function $H$ satisfying the conditions, while Proposition \ref{prop:E} provides a power series solution.

\begin{remark}
Our proof is self-contained and relies on the Taylor series expansion of $\zeta$ in Proposition \ref{prop:Fred_exp}. This proposition is a version of standard results in de Branges' theory of Hilbert spaces of analytic functions adapted to our setting. For example, 
Theorem 41 of de Branges \cite{deBranges} is a version  of the existence and uniqueness part of  Proposition \ref{prop:zetaODE} with slightly different assumptions. The proof of that theorem
relies on a deep understanding of  de Branges' theory.
\end{remark}

\begin{remark}
The function $H$ can be considered as the solution of the eigenvalue equation $\btau H=z H$ with initial condition $H(0,z)=\mathfrak  u_0  $. The expression $H(\sigma, z)^t J \mathfrak  u_1 $ gives a linear transform of $H(\sigma, z)$ which is zero exactly when $H(\sigma, z)\parallel \mathfrak  u_1 $, and it is equal to $1$ at $z=0$.
\end{remark}

\begin{remark}\label{rem:ODElin}
The ODE in \eqref{eq:E_1} is linear, and by Assumption 1 the function $\|R(s)\|$ is bounded for any compact subset of $(0,\sigma]$. Hence for any $0<a<b\le \sigma$, $z\in \CC$ and $v\in \CC^2$ the ODE
\begin{align}
J\frac{d}{dt}G(t,z)=z R(t) G(t,z), \qquad G(a,z)=v,
\end{align}
has a unique solution in $[a,b]$ by the standard theory of ordinary differential equations, and the solution is analytic in $z$ for any given $t\in [a,b]$. This holds also if we assume that the initial condition $v=v(z)$ is an analytic function of $z$.

Under the additional assumption  $\int_0^{\sigma} \|R(s)\|ds< \infty$ this extends to the $a=0$ case. Hence in this case existence and uniqueness of $H(t,z)$ in Proposition \ref{prop:zetaODE}  follows immediately. If $\|R(t)\|$ is bounded on  $(0,\sigma]$ as opposed to just locally bounded, then $\int_0^{\sigma} \|R(s)\|ds< \infty$. Our assumptions allow $\|R(t)\|$ to blow up near zero, and most of our applications have this property.

\begin{remark}
Gesztesy and Makarov  \cite{GM} give an explicit formula for the modified Fredholm determinant of a Hilbert-Schmidt integral operator with a semi-separable kernel $$M(s,t)=f_1(s) g_1(t)\ind(s<t)+f_2(s)g_2(t)\ind(t\le s)$$ on an interval $(a,b)$, assuming that the matrix valued functions $f_1, g_1, f_2, g_2$ are all in the appropriate $L^2$ spaces.

Note that $K_{\res \tau}$ is a semi-separable kernel by (\ref{resint}), and in the case when $\int_0^{\sigma} \|R(s)\|ds<\infty$ the vector valued functions $a(\cdot)$ and $c(\cdot)$ are both in $L^2(0,\sigma]$.
Hence in this case the results of \cite{GM} apply, and it can be checked that the derived formula leads to  (\ref{zetaH}) again. \end{remark}

\end{remark}

\begin{proposition}[Uniqueness of $H$]\label{prop:unique}
Suppose that for a given $z\in \CC$ the functions $H_1(t,z), H_2(t,z)$  both satisfy \eqref{eq:E_1}.
Then $H_1(t,z)=H_2(t,z)$. Moreover, $\|H_1(t,z)\|>0$ for all $t,z$.
\end{proposition}
\begin{proof}
 From (\ref{eq:E_1}) we get for $t\in (0,\sigma]$
\begin{align*}
&\frac{d}{dt} \left(H_2(t,z)^t J H_1(t,z)\right)=
0,
\end{align*}
so $H_2(t,z)^t J H_1(t,z)$ is constant in $t$. Since $\lim_{t\to 0}H_k(t,z)=\mathfrak  u_0  $ for $k=1,2$, we see that this constant has to be $\mathfrak  u_0  ^t J \mathfrak  u_0  =0$. The identity $H_2(t,z)^t J H_1(t,z)=0$ implies that $H_2(t,z) \parallel H_1(t,z)$ for all $t\in (0,\sigma]$.

By Remark \ref{rem:ODElin} he equation \eqref{eq:E_1} is linear, so if  $H_k(t_0,z)=(0,0)^t$ for a particular $t_0\in (0,\sigma]$  then $H_k(t, z)=(0,0)^t$ would hold for all $t\in (0,\sigma]$, which would contradict the behavior at $t\to 0$. Hence $\|H_k(t,z)\|$ cannot be zero. Since  $H_2(t,z) \parallel H_1(t,z)$ for all $t\in (0,\sigma]$, it follows that  there exists a function $f(t,z)$ so that $H_2(t,z)=f(t,z) H_1(t,z)$, $f(t,z)\neq 0$ for $t\in (0,\sigma]$.

Using the linearity of \eqref{eq:E_1} again we see that $f(t,z)$ must be a constant in $t$. But the initial condition in (\ref{eq:E_1}) then implies that $f(t,z)=1$ and $H_1=H_2$.
\end{proof}

\begin{proposition}[A power series solution]\label{prop:E}
For $t\in (0,\sigma]$, $z\in \CC$ define $H(t,z)$ using the following series expansion:
\begin{align}\notag
&\qquad H(t,z)^t=\sum_{n=0}^{\infty} d_n(t)  z^n,\qquad\qquad  d_0(t)=\mathfrak  u_0  ^t\\
d_n(t)=&\!\!\iiint\limits_{0<s_1<s_2<\dots<s_n\le t} \mathfrak  u_0  ^t R(s_1)JR(s_2)J\cdots R(s_n) J ds_1\dots ds_n, \qquad n\ge 1.\label{eq:E_exp}
\end{align}
The series converges for all $z\in \CC$.  For any $t\in (0,\sigma]$ the two coordinates of the function $H(t,z)$ are entire functions of $z$ mapping the reals to the reals. The function $H(t,z)$  satisfies (\ref{eq:E_1}) and (\ref{zetaH}).
%
\end{proposition}
\begin{proof}
Let $\mathfrak u$ be any vector not parallel to $\mathfrak  u_0  $. Then $a=\mathfrak  u_0  ^tJ\mathfrak u\neq0$. Set $\mathfrak  u_1 =\mathfrak u/a$ so that  $\mathfrak  u_0  ^tJ\mathfrak  u_1 =1$. Set $r_n=d_n J \mathfrak  u_1 $.  Proposition \ref{prop:Fred_exp}  shows that the series
$$1+\sum_{n=1}^\infty r_nz^n= \frac{1}{a}\sum_{n=0}^\infty d_n(t)J \mathfrak  uz^n$$
converges everywhere. Using two linearly independent $\mathfrak u$-s we see that  $H$ is well-defined and its two coordinates are analytic in $z$. Since the coefficients $d_n$  are real vectors, the coordinates of $H$ map reals to reals for any  $t\in (0,\sigma]$.

Proposition \ref{prop:Fred_exp} also shows that the identity (\ref{zetaH}) holds. The only thing left is to show that $H$ satisfies the ODE (\ref{eq:E_1}) for all $z\in \CC$.

By estimate (\ref{coef_bnd0})  of Proposition \ref{prop:Fred_exp} we have
\begin{align}
|d_n(t)J \mathfrak  u_1 |\le r(t)^n, \qquad r(t)=\int_0^t \left|\mathfrak  u_0  R(s) \mathfrak  u_1 \right| ds+\int_0^{\sigma} \int_0^t \mathfrak  u_0  ^t R(s) \mathfrak  u_0   \mathfrak  u_1 ^t R(t) \mathfrak  u_1  ds dt.
\end{align}
By Assumption 2, the function $r(t)$ is finite, non-decreasing and $\lim_{t\downarrow 0} r(t)=0$. This implies the bound $\|d_n(t)\|\le c r(t)^n$ for a finite $c>0$ for all $n\ge 1$, from which we get  $\lim_{t\downarrow 0} H(t,z)=  \mathfrak  u_0  $ for all $z\in \CC$.

From  (\ref{eq:E_exp}) it  follows that for $ n\ge 1$ and $0\le t_1<t_2\le \sigma$ we have
\[
d_n(t_2)-d_n(t_1)=\int_{t_1}^{t_2} d_{n-1}(s)  R(s) J ds.
\]
Now assume $t_1>0$, then  $R$ is uniformly bounded in $[t_1,t_2]$.  Multiplying both sides by $z^n$ and summing over $n$ we get that both sides are absolutely convergent for $|z|<r^{-1}(t_2)$. Thus when this inequality holds, we have
\[
H(t_2,z)^t-H(t_1,z)^t =\int_{t_1}^{t_2} z H(t,z)^t R(t)J \, ds.
\]
Since $J^{-1}=J^t=-J$, the fundamental theorem of calculus gives that the  differential equation \eqref{eq:E_1} holds for $z\in \CC$, $t\in(0,\sigma]$ with $|z|<r(t)^{-1}$.

It suffices to  show that $H(t,z)$ satisfies the ODE for $|z|\le b$, $t\in(0,\sigma]$  for any fixed $b>0$.
Given $b>0$, pick $t_0$ so that $b<r(t_0)^{-1}$. Then $H(t,z)$ is a solution in $(0,t_0]$ for $|z|\le b$.

For all $|z|\le b$ let $G(t,z)$ be the solution of the ordinary differential equation \eqref{eq:E_1} on $[t_0,\sigma]$ with initial condition $H(t_0,z)$.
By Remark \ref{rem:ODElin} for $ t_0\le t\le\sigma$, the function $G(t,z)$ is analytic in $z$. Moreover, $G(t,z)=H(t,z)$ when $|z|<r^{-1}(\sigma)$, since $H(t,z)$ also solves \eqref{eq:E_1} there, and it agrees with $G(t,z)$ when $t=t_0$, $|z|\le r^{-1}(\sigma)\le b$.
Thus for $t\in [t_0,\sigma]$ the analytic functions $G(t,z)$ and $H(t,z)$ must agree for all $|z|\le b$, which implies that $H(t,z)$ solves the ODE for $|z|\le b$, $t\in(0,\sigma]$.
\end{proof}

The next proposition provides a way to approximate $H$ with the solutions of more regular ODE systems for which the uniqueness of the solution is immediate.

\begin{proposition}[Regular approximation of $H$]\label{prop:Heps-convergence}
Let $0<\eps<\sigma$, $z\in \CC$ and let $H_\eps(t,z)$ be the solution of the ODE
\begin{align}
J\frac{d}{dt}H_\eps(t,z)=z R(t) H_\eps(t,z), \qquad t\in[\eps,\sigma], \qquad H_\eps(\eps,z)&=  \mathfrak  u_0  . \label{eq:E_eps}
\end{align}
Extend the definition of $H_\eps(t,z)$ for $t\in(0,\eps)$ with $H_\eps(t,z)=\mathfrak  u_0  $.

Then as $\eps\to 0$ we have $H_\eps(\cdot,z)\to H(\cdot, z)$ uniformly on compact subsets of $(0,\sigma]\times \mathbb C$.
\end{proposition}
\begin{proof}
By Remark \ref{rem:ODElin}
the differential equation (\ref{eq:E_eps}) has a unique solution.

Moreover,  for any $\mathfrak  u_1 $ with $\mathfrak  u_0  ^tJ\mathfrak  u_1 =1$ and $\eps\le t\le \sigma$ by Proposition \ref{prop:zetaODE} the function $H_\eps(t,z)^t J \mathfrak  u_1 $ is the secular function of $\btau_{\eps,t,\mathfrak  u_1 }$, which is $\btau$ restricted to $[\eps, t]$ with boundary condition $\mathfrak  u_0  $ at $\eps$ and $\mathfrak  u_1 $ at $t$.

Note that $\res \btau_{\eps,t,\mathfrak  u_1 }$ is the integral operator with kernel given in (\ref{resint}), but restricted to $[\eps,t]\times[\eps, t]$. In other words, we have
\begin{align}
H_\eps(t,z)^t J \mathfrak  u_1 ={\det}_2(I-z \,\res \btau_{\eps, t, \mathfrak  u_1 }) e^{-\tfrac{z}{2} \int_{\eps}^{t} a(s)^t c(s) ds}.
\end{align}
For $0<t\le \sigma$ we denote by $\btau_{t,\mathfrak  u_1 }$ the version of $\btau$ on $(0,t]$ with boundary conditions $\mathfrak  u_0  , \mathfrak  u_1 $. This operator satisfies Assumptions 1 and 2, and Proposition \ref{prop:zetaODE} shows that
\begin{align}
H(t,z)^t J \mathfrak  u_1 ={\det}_2(I-z \,\res \btau_{ t, \mathfrak  u_1 }) e^{-\tfrac{z}{2} \int_{0}^{t} a(s)^t c(s) ds}.
\end{align}
Note that the Hilbert-Schmidt norm of $\res \btau_{ t, \mathfrak  u_1 }$ is uniformly bounded in $t$ as
\[
\|\res \btau_{ t, \mathfrak  u_1 }\|_2\le \|\res \btau\|_2<\infty.
\]
By the triangle inequality we have
 \begin{align}\notag
&\left|(H(t,z)-H_\eps(t,z)^t) J \mathfrak  u_1 \right|
  \le \left|{\det}_2(I-z \,\res \btau_{\eps, t, \mathfrak  u_1 })-{\det}_2(I-z \,\res \btau_{t, \mathfrak  u_1 })\right| e^{-\tfrac{z}{2} \int_{\eps}^{t} a(s)^t c(s) ds}\\
  &\qquad +\left|{\det}_2(I-z \,\res \btau_{t, \mathfrak  u_1 }) \right| e^{-\tfrac{z}{2} \int_{0}^{t} a(s)^t c(s) ds} \left|e^{-\tfrac{z}{2} \int_{0}^{\eps} a(s)^t c(s) ds}   -1\right|.\label{triangle1}
 \end{align}
 If $\kappa_1, \kappa_2$ are Hilbert-Schmidt  operators on the same domain then
\begin{align}\label{det2_tri}
|{\det}_2(I-z \kappa_1)-{\det}_2(I-z \kappa_2)|\le |z|\cdot  \|\kappa_1-\kappa_2\|_2 \exp(c |z|^2 (\|\kappa_1\|_2^2+ \|\kappa_2\|_2^2)+c)
\end{align}
with an absolute constant $c$, see Theorem 9.2(c) in \cite{SimonTrace}.
In particular, using this for $\kappa_2=0$ we get the bound
\begin{align}\label{det2_tri2}
|{\det}_2(I-z \kappa_1)-1|\le    |z|\cdot  \|\kappa_1\|_2 \exp(c |z|^2 \|\kappa_1\|_2^2+c).
\end{align}
This implies that the $\det_2$ part of the second term on the right of (\ref{triangle1}) is uniformly bounded in compact sets of $t,z$. Since $\int_0^{\sigma} \left| a(s)^t c(s) \right| ds<\infty$ by assumption, the entire term converges to $0$ as $\eps\to 0$ uniformly.

We will show that the same holds for the other term in (\ref{triangle1}).

Extend the domain of the integral operator $\res \tau_{\eps, t, \mathfrak  u_1 }$ to $(0,t]^2$ by setting the kernel equal to the constant 0 matrix on $(0,t]^2\setminus (\eps,t]^2$.
We use the temporary notation $\kappa_\eps$ for the new integral operator. The spectrum of   $\kappa_\eps$ is  given by  the spectrum of $\res \tau_{\eps, t, \mathfrak  u_1 }$ and the value 0 with infinite multiplicity.
This means that
\[
{\det}_2(I-z \kappa_\eps)={\det}_2(I-z \,\res \tau_{\eps, t, \mathfrak  u_1 }).
\]
Hence
\begin{align}\label{triangle2}
\left|{\det}_2(I-z \,\res \tau_{\eps, t, \mathfrak  u_1 })-{\det}_2(I-z \,\res \tau_{t, \mathfrak  u_1 })\right|=\left|{\det}_2(I-z \,\kappa_\eps)-{\det}_2(I-z \,\res \tau_{t, \mathfrak  u_1 })\right|.
\end{align}
Since $\res \tau_{t, \mathfrak  u_1 }$ is Hilbert-Schmidt and its kernel agrees with that of $\res \tau_{\eps, t, \mathfrak  u_1 }$ on $[\eps,t]^2$ we have
\begin{align}\label{K_appr}
\lim_{\eps\to 0}\|\res \tau_{t, \mathfrak  u_1 }-\kappa_\eps\|_2=0
\end{align}
By (\ref{det2_tri}) and (\ref{K_appr}) the  term on the right of (\ref{triangle2}) converges to 0 as $\eps\to 0$ uniformly for $t, z$ in a compact set.

Collecting all of our estimates, and recalling that $\int_0^{\sigma} \left|a(s)^t c(s) \right| ds<\infty$ by assumption, we get that for a given $\mathfrak  u_1 $ with $\mathfrak  u_0  ^tJ\mathfrak  u_1 =1$ we have
\[
\lim_{\eps\to 0} H_\eps(t,z)J\mathfrak  u_1 =H(t,z)J\mathfrak  u_1 ,
\]
and the convergence is uniform on compact sets of $t, z$. This statement is true for any $\mathfrak  u_1  \not \,\parallel \mathfrak  u_0  $, which implies the proposition.
\end{proof}


The bounds (\ref{det2_tri}, \ref{det2_tri2}) together with the  definition (\ref{zeta}) of $\zeta$ and the triangle inequality gives the following.

\begin{proposition}[Continuity of $\btau\mapsto \zeta_\btau$]\label{prop:zeta_tri}
Let $\btau_1 $, $\btau_2$ be two Dirac operators on $(0,\sigma]$ satisfying our assumptions.  Denote by $\zeta_i, \res_i, \mathfrak t_i$ the  secular function, resolvent and integral trace of $\btau_i$. 
Let  $\|\cdot \|$ denote the Hilbert-Schmidt norm.
Then there is a universal constant $a>1$ so that for all $z\in \CC$
\begin{align}
|\zeta_1(z)-\zeta_2(z)|\le \Big(e^{|z| |\mathfrak t_1-\mathfrak t_2|}-1+|z|\big\|\res_1-\res_2\big\|\Big)
a^{|z|^2(\|\res_1\|^2+\|\res_2\|^2)+|z|(|\mathfrak{t}_1|+|\mathfrak{t}_2|)+1}.
\end{align}
\end{proposition}

Proposition \ref{prop:zeta_tri} provides a sufficient condition for the convergence of secular functions. It shows that if we have a sequence of Dirac operators on $(0,\sigma]$ for 
which the resolvents converge in Hilbert-Schmidt norm, and the integral traces converge, then the secular functions converge uniformly on compacts of $\CC$.

\subsection{The structure function $E$ and the functions $A,B$}

From this point we will assume that we are in the case of $q$-boundary conditions \eqref{eq:qbndry}.

Consider the unique vector-valued function $H$ described in Proposition \ref{prop:zetaODE}. Set
\begin{align}\label{def:AB}
   A(t,z):= [1,0] H(t,z), \qquad B(t,z):=[0,1] H(t,z),\\[3pt]
   E(t,z):=[1,-i] H(t,z)=A(t,z)-i B(t,z).\label{def:E}
\end{align}
For $t$ fixed, $z\mapsto E(t,z)$ is  called the {\bf structure function} in the context of de Branges' theory of Hilbert spaces of analytic functions. 

\begin{lemma}\label{lem:Polya}
The entire functions $A(t,z), B(t,z)$ are real for real $z$, and have only real zeros. If the entire function $E(t,z)$ has  zeros then they are in the open upper half plane. Moreover, 
\begin{align}\label{E_basic}
&E(t,0)=1, \qquad \bar A(t,z)=A(t, \bar z), \qquad \bar B(t,z)=B(t, \bar z),\\[5pt]
&|E(t,z)|\le  |E(t,\bar z)|, \qquad \Im z>0.\label{E_AB}
\end{align}
\end{lemma}

\begin{proof}
Proposition  \ref{prop:zetaODE} implies that $E(t,z)$ is an entire function of $z$ for any $t$. The vector valued function $H(t,z)$ is not equal to $[0,0]^t$ for any $t,z$, and it takes values from $\R^2$ for $z\in \R$.
Hence $E$ is nonzero for real $z$, and $A$ and $B$ are real for $z\in \R$. 
The statements of (\ref{E_basic}) follow from $H(t,0)=\mathfrak u_0=[1,0]^t$ and the reflection principle. Since for every $r\in \R$ the function $A+rB$ is a secular function, it  only has real zeros. Moreover, $B=\lim_{r\to \infty} (A+rB)/r$ has only real zeros by Hurwitz's theorem.

From (\ref{eq:E_1}) we get
\begin{align*}
-i \frac{d}{dt}\overline{H}(t,z)^{t} J H(t,z)=2 \Im z \overline{H}(t,z)^t R(t) H(t,z).
\end{align*}
We also have
$$
-2i\overline{H}(t,z)^{t} J H(t,z)=
|E(t,\bar z)|^2-|E(t,z)|^2.
$$
Since $\det R(t)=\frac14$ and the entries of $R(t)$ are real, one can check that
$
    R(t)+\frac{i}2 J
$
is nonnegative definite, which implies that for $\Im z>0$ we have
$$  
\frac{d}{dt}\left(|E(t,\bar z)|^2-|E(t,z)|^2\right)\ge {\Im z} \left(|E(t,\bar z)|^2-|E(t,z)|^2\right),
$$
which shows (\ref{E_AB}). 

Finally, we note that if $E(t,z)=0$ for $\Im z<0$ then \eqref{E_AB} would imply $E(t,z)=E(t,\bar z)=0$ and $A(t,z)=B(t,z)=0$. By Proposition \ref{prop:zetaODE} this is not possible, hence $E(t,z)=0$ can only hold for $\Im z>0$.
\end{proof}

For the interested reader, the following  proposition summarizes  additional properties of $H$ and the structure function $E$. The proofs can be found in  de Branges \cite{deBranges} with  some  required exercises.  We will not use these results in the present paper. 

\begin{proposition}
\begin{enumerate}[(i)]
\item
The function $H(t,z)$ is in $L^2_R$ for any fixed $z\in \CC$. In particular, if $\lambda\in \CC$ is an eigenvalue of $\btau$ then the corresponding eigenfunction is $H(\cdot, \lambda)$.
\item 
The  function $|E(t,x-i y)|$ is non-decreasing in $y$ for $y\ge 0$.
\item
The structure function $E=A-iB$ satisfies the following bound with all derivatives  with respect to $z$:
\begin{align}\label{eq:Ebound}
\log|E(t,z)|\le \Re z A'(t,0)+\Im z B'(0)+\frac12\left(A'(t,0)-A''(t,0)+B'(t,0)\right)^2 |z|^2.
\end{align}
\end{enumerate}
\end{proposition}

\begin{example}[Deterministic Sine operator, part 3]\label{ex:bulk_3}
Let us return to  Example \ref{ex:bulk} with boundary conditions $\mathfrak  u_0  =[1,0]^t$, $\mathfrak  u_1 =[-\cot(\theta/2), -1]^t$. The solution of (\ref{eq:E_1}) is given by
\begin{align}
    H(t,z)=\left[
 \cos (t z/2),
 -\sin (t z/2)
\right]^t.
\end{align}
Proposition \ref{prop:zetaODE} gives
\[
\zeta(z)=H(\sigma,z) J \mathfrak  u_1 =\cot (\sigma/2)  \sin(\sigma z/2)+\cos(\sigma z/2)
\]
which  agrees with  the expression in  (\ref{zeta_Ex1}).
Note also that
\begin{align*}
A(t,z)=\cos(\tfrac{t z}{2}), \quad B(t,z)=-\sin(\tfrac{t z}{2}), \quad E(t,z)=e^{\frac{i tz}{2}}.
\end{align*}
\end{example}

\begin{example}[Deterministic Bessel operator, part 3]\label{ex:edge_3}
Now consider Example \ref{ex:edge} with boundary conditions $\mathfrak  u_0  =[1,0]^t$, $\mathfrak  u_1 =[0, -1]^t$.
Then the solution of (\ref{eq:E_1}) is given by
\begin{align}\label{eq:BesselH_2}
H(t,z)=\left(\tfrac{z t}{4}\right)^{\frac{1-\alpha}{2}} \Gamma(\tfrac{\alpha +1}{2})\left[ \mathrm{J}_{\frac{\alpha -1}{2}}\left(\tfrac{z t }{2}\right) , -t^{\alpha} \mathrm{J}_{\frac{\alpha +1}{2}}\left(\tfrac{z t }{2}\right)   \right]^t.
\end{align}
Proposition \ref{prop:zetaODE} leads to the expression (\ref{zeta_Ex2}).
We also have
\begin{align*}
&A(t,z)=\left(\tfrac{z t}{4}\right)^{\frac{1-\alpha}{2}} \Gamma(\tfrac{\alpha +1}{2})\mathrm{J}_{\frac{\alpha -1}{2}}\left(\tfrac{z t }{2}\right),  \qquad B(t,z)=-\left(\tfrac{z t}{4}\right)^{\frac{1-\alpha}{2}} \Gamma(\tfrac{\alpha +1}{2}) t^{\alpha} \mathrm{J}_{\frac{\alpha +1}{2}}\left(\tfrac{z t }{2}\right)\\[3pt]
& \hskip50pt E(t,z)=\left(\tfrac{z t}{4}\right)^{\frac{1-\alpha}{2}} \Gamma(\tfrac{\alpha +1}{2})\left( \mathrm{J}_{\frac{\alpha -1}{2}}\left(\tfrac{z t }{2}\right) + i t^{\alpha} \mathrm{J}_{\frac{\alpha +1}{2}}\left(\tfrac{z t }{2}\right)   \right).
\end{align*}
\end{example}

\subsection{Infinite product representation 
} \label{subs:prod}

An entire function $f$ is  of  {\bf finite exponential type} if there exists $c>1$ so that
\begin{align}\label{eq:exponential type}
|f(z)|\le c^{1+ |z|} ,\qquad \text{for all $z\in \CC$}
\end{align}
If, in addition,  the integral condition
\begin{align}\label{eq:log_int}
\int_{-\infty}^\infty \frac{\log_+|f(x)|}{1+x^2}dx<\infty.
\end{align}
holds, then we say that $f$ is of {\bf Cartwright class}. In this section we study $\zeta_\btau$ when it falls in the Cartwright class.  The role of the Cartwright class of functions in the world of Dirac operators is similar to the role of polynomials in the world of matrices.

When the secular function $\zeta_\btau$ is of Cartwright class, it can be  represented as a principal value product.

\begin{proposition}[Principle value product]\label{prop:inf_prod}
Under Assumptions 1-3, when the secular function  $\zeta=\zeta_\btau$ is of Cartwright class,  we have
\begin{align}\label{eq:zeta0}
\zeta(z)=\lim\limits_{r\to \infty} \prod_{|\lambda_k|<r}(1-z/\lambda_k)
\end{align}
where the convergence is uniform on compact sets of $z\in \CC$. In particular for $x\in \R$,
\begin{equation}\label{eq:triangle}
|\zeta(x+iy)|\qquad \text{ is increasing in }y\ge0.
\end{equation}
Moreover, the following three quantities are finite and  equal:
\begin{align}\label{eq:exp_order}
       \limsup_{|z|\to\infty}\frac{\log|\zeta(z)|}{|z|}=\limsup_{y\to\infty} \frac{\log|\zeta(iy)|}{y}=\lim_{|k|\to \infty} \frac{\pi k}{\lambda_k}.
\end{align}

We also have
 \begin{align}\label{eq:traceid}
 \lim_{r\to \infty} \sum_{|\lambda_k|<r}\lambda_k^{-1}=\frac12 \int_0^{\sigma} a(s)^t c(s) ds.
 \end{align}
\end{proposition}
The final statement of the proposition shows that even though the operator $\res \tau$ might not be trace class, the integral trace $\int_0^{\sigma} \tr K_{\res \tau}(s,s)ds$ is equal to the principal value sum of $\lambda_k^{-1}$.
\begin{proof}
According to Theorem 11 of Section V.4.4 in Levin \cite{Levin} for an  entire function $\zeta$ of exponential type with (\ref{eq:log_int})  \begin{align}
    \zeta(z)=c z^m e^{i  b z} \lim\limits_{r\to \infty} \prod_{|\lambda_k|<r} (1-\frac{z}{\lambda_k}), \label{eq:C_product}
\end{align}
where $b$ is real, $m$ is a non-negative integer, and $\lambda_k$, $k=1,2,\dots$ are the nonzero zeros of $\zeta$.
We apply this to the secular function $\zeta$. Since  $\zeta(0)=1$,  we see that $m=0$ and $c=1$. Since $\zeta$ maps reals to the reals and has real zeros it follows that $b=0$. This completes the proof of (\ref{eq:zeta0}) for pointwise convergence. Claim \eqref{eq:triangle} holds factor-by-factor and is preserved in the limit.  

Recall $\mathfrak t = \tfrac12 \int_0^{\sigma} a(s)^t c(s) ds$, \eqref{eq:int_tr}. By the definition (\ref{zeta2}) of $\zeta$, we have
\begin{align}
  e^{\mathfrak t  z}\zeta(z) =  \lim\limits_{r\to \infty} e^{z\sum_{|\lambda_k|<r}1/\lambda_k}\prod_{|\lambda_k|<r} (1-z/\lambda_k),\label{eq:zeta_prod_99}
\end{align}
and the convergence is uniform on compacts in $z$. Choosing a  $z\in \CC$ with  $z\neq 0$,  $\zeta(z)\not=0$   now implies \eqref{eq:traceid}.
As a consequence,
\[
e^{-\mathfrak t z}=
 \lim\limits_{r\to \infty} e^{-z\sum_{|\lambda_k|<r}1/\lambda_k}
\]
uniformly on compacts. Multiplying this by \eqref{eq:zeta_prod_99} we get that the convergence in (\ref{eq:zeta0}) is uniform on compacts.

To prove \eqref{eq:exp_order} we use further results from Levin's books \cite{Levin,Levin2}. The relevant theorems can be difficult to find because this is a simple case of a general theory. 

Theorem 11 (2) of Section V.4.4, \cite{Levin} shows that the limit on the right of \eqref{eq:exp_order} exists and equals $d_h/2$, where $d_h$ is the length of the indicator diagram of $\zeta$.  The indicator diagram is a convex set given by a simple geometric transform of the exponential growth rate function (called indicator function)  
\begin{equation}\label{eq:indicator}
h(\theta)=\limsup_{r\to\infty} \log|\zeta(e^{i\theta}r)|/r,
\end{equation}
see equation (1.64) for the definition of $h$ and Section I.19 for a description of the transform. Theorem V.7 identifies $h(\theta)=k|\sin \theta|$ for some $k$.  By Section I.19 \cite{Levin} we have $d_h=2k$. Taking $\theta=\pi/2$ in \eqref{eq:indicator} gives the second equality of \eqref{eq:exp_order}, and the $\ge$ part of the first equality is clear. 

Since $h$ is continuous,  Theorem 2 of Section 8 on page 56 of \cite{Levin2} shows that for every $\eps>0$ there is $r_0$ so that for all $z=re^{i\theta}$ with $r\ge r_0$ we have $|\zeta(z)|\le e^{r(h(\theta)+\eps)}$, which shows the $\le$ part of the first equality in \eqref{eq:exp_order}. 
\end{proof}

The secular function $\zeta_\btau$ can also be defined as follows. This is analogous to the definition of the characteristic polynomial through its zeros. The proof of the proposition follows the argument in pages 156-157 of Chapter III of Levin, \cite{Levin}.

 \begin{proposition}\label{prop:inf_prod_1}
Suppose that $\btau$  satisfies Assumptions 1-3, and for some $0<\rho$ and $\eps<1$  the ordered sequence of eigenvalues $\lambda_n$  satisfies
\begin{align}
    |\lambda_n-\rho  n|n^{-\eps}\to 0\qquad \text{  as} \quad |n|\to \infty.\label{eq:lambda_gr}
\end{align}
Assume further that $\zeta_\btau$ satisfies the integral condition \eqref{eq:log_int}.
Then $\zeta_\btau$ is of  exponential type, and hence it is of Cartwright class.
 \end{proposition}

\begin{proof}
Let 
$
\mathfrak{t}_r=\sum_{|\lambda_k|\le r} \frac{1}{\lambda_k},
$
by  assumption (\ref{eq:lambda_gr}) the limit $ \mathfrak{t}=\lim_{r\to \infty} \mathfrak{t}_r$ exists. Then we can write 
\[
\zeta_\btau(z)=e^{(\mathfrak{t}_r-\mathfrak{t})z} \prod_{|\lambda_k|\le r} (1-\tfrac{z}{\lambda_k}) \prod_{|\lambda_k|>r} (1-\tfrac{z}{\lambda_k}) e^{z/\lambda_k},
\]
where both products are well defined. There is an absolute constant $c>0$ so that for any $u\in \CC$ we have the bounds
\[
\log|1-u|\le \log(1+|u|), \qquad \log\left|(1-u)e^{u} \right|\le c \frac{|u|^2}{1+|u|}.
\]
Hence 
\[
\log|\zeta_\btau(z)|\le |\mathfrak{t}_r-\mathfrak{t}||z|+\sum_{|\lambda_k|\le r} \log (1+\frac{|z|}{|\lambda_k|})+\sum_{|\lambda_k|>r} c \frac{|z/\lambda_k|^2}{1+|z/\lambda_k|}.
\]
Introduce 
$
n(t)=\left|\{k: |\lambda_k|\le t\}\right|,
$
by our assumptions there are finite positive constants $c, \delta$ so that $
n(t)\le \rho t + c$ for $t>0$, $n(\delta)=0.$
We set $r=|z|$. Then 
\begin{align*}
\sum_{|\lambda_k|\le r} \log (1+\frac{|z|}{|\lambda_k|})&= \int_0^r \log(1+\frac{r}{t}) d n(t)=r \int_0^r \frac{n(t)}{t(t+r)}dt+n(r) \log(2)\\
&\le  \int_\delta^r \frac{n(t)}{t} dt+n(t) \log(2).
\end{align*}
Similarly,
\begin{align*}
   \sum_{|\lambda_k|>r}  \frac{|z/\lambda_k|^2}{1+|z/\lambda_k|} &=\int_r^\infty \frac{r^2}{t(r+t)}dn(t)=\int_r^\infty \frac{r^2 (r+2 t)}{t^2 (r+t)^2} n(t) dt -\frac{1}{2} n(r)
   \le 3 r^2 \int_{r \vee \delta}^\infty \frac{n(t)}{t^3} dt.
\end{align*}
From these bounds we get
$
\log |\zeta(z)|\le c_0+c_1 |z|+c_2 \log(1+|z|),
$
which implies the statement of the lemma. 
\end{proof}

The following folklore proposition implies that when $\zeta_\btau$ is of Cartwright class, it is determined by its zeroes in this class.

\begin{proposition}\label{prop:Cart-unique}
Assume that $f,g$ are of Cartwright class, and have the same zeros with multiplicities. Assume further that $f(0)=g(0)\not=0$, and $f, g$ map reals to reals. Then $f=g$.
\end{proposition}

\begin{proof}
The function $h=\log(f/g)$ is an entire function with at most linear growth, so by Liouville's theorem it is linear with $h(0)=0$. Since $h(\mathbb R)\subset \mathbb R$, we  have $h(z)=rz$ for real $r$. Since $h(x)_+\le \log_+|f|+\log_+|g|$, the function $h_+(x)/(x^2+1)$ is integrable by the condition \eqref{eq:log_int}. Hence $r=0$.
\end{proof}

\begin{proposition}\label{prop:logder}
Suppose that the assumptions of Proposition \ref{prop:inf_prod_1} hold. Let $N:\R\to \R$ be the counting function of the zeros of $\zeta_\btau$ with $N(0)=0$.  
Then for $z\notin \R$ we have
\begin{equation}\label{eq:logder}
 \frac{\zeta_\btau'(z)}{\zeta_\btau(z)}=\int_{-\infty}^\infty \frac{1}{(z-\lambda)^2} ( \lambda/\rho-N(\lambda)) d\lambda+\operatorname{sign}(\Im z) \pi i/\rho.
\end{equation}
\end{proposition}
\begin{proof}
The partial products converge uniformly on compacts in the upper/lower half plane. Because of that, the log-derivatives of the partial products converge as well:
\begin{align}\label{eq:partial_frac}
\frac{\zeta'(z)}{\zeta(z)}=\lim_{r\to \infty} \sum_{|\lambda_k|\le r} \frac{1}{z-\lambda_k}
\end{align}
For a fixed $r$, integration by parts gives
\begin{align*}
\sum_{|\lambda_k|\le r} \frac{1}{z-\lambda_k}&=
-\int_{-r}^r \frac{1}{(z-\lambda)^2} N(\lambda) d\lambda+ \frac{1}{z-r}N(r)-\frac{1}{z+r}N(-r).
\end{align*}
With a compensation term, the right hand side can be written as
\begin{align*}
&\int_{-r}^r \frac{1}{(z-\lambda)^2} (\lambda/\rho-N(\lambda) ) d\lambda  -\int_{-r}^r \frac{1}{(z-\lambda)^2} \cdot \frac{\lambda}{\rho}d\lambda\\
&+\left(\frac{1}{z-r}(N(r)-r/\rho)-\frac{1}{z+r}(N(-r)+r/\rho)\right) +\frac{r/\rho}{z-r}+\frac{r/\rho}{z+r}.
 \end{align*}
The terms in the second line vanish as $r\to \infty$ since $|N(\lambda)-\lambda/\rho|=O(\lambda^{1-\eps})$.
 The first term converges to the first term of our formula, which is absolutely integrable.
The claim about the second term follows by residue calculus with a contour integral over a radius $r$ semicircle $C_r$ in the closed upper half plane:
\[
\lim_{r\to \infty}\int_{-r}^r \frac{\lambda}{\rho(z-\lambda)^2} d\lambda=\lim_{r\to \infty}\int_{C_r} \left(\frac{1}{(z-\lambda)^2}-\frac{1}{(z+\lambda)^2}\right)\frac{\lambda}{2\rho}d\lambda=\frac{\operatorname{sign}(\Im z) \pi i}{\rho}. \qedhere
\]
\end{proof}

\begin{example}[Deterministic Sine and Bessel operators, part 4]\label{ex:part_4}
Returning to $\btau$ from Example \ref{ex:bulk} we see that $ \zeta(z)= \sin (\tfrac{\sigma z+\theta}{2})/\sin (\tfrac{\theta }{2})$ is of Cartwright class with zeros  $\lambda_k=\frac{2 \pi k -\theta}{\sigma}$. Proposition \ref{prop:inf_prod} leads to the following well-known identities:
\[
\frac{ \sin (\tfrac{\sigma z+\theta}{2})}{\sin (\tfrac{\theta }{2})}=\lim_{r\to \infty} \prod_{|k|\le r} \left(1-\frac{\sigma z}{2\pi k-\theta}\right), \qquad -\cot(\theta/2)=\lim_{r\to \infty} \sum_{|k|\le r} \frac{2}{2\pi k-\theta}.
\]
For $\btau$ from Example \ref{ex:edge} the well-known asymptotics of the Bessel function show that $\zeta_\btau$ given in (\ref{zeta_Ex2}) is of Cartwright class. Since the zero set of $\zeta_\btau$ is symmetric about 0, the identity (\ref{eq:zeta0}) is equivalent to (\ref{Ex2_product}) (which follows from the definition),
and (\ref{eq:traceid}) becomes trivial.
\end{example}

\begin{proposition}\label{prop:B_product}
Suppose that $\btau$  satisfies Assumptions 1-3,  and fix $t\in(0,\sigma]$. If the function function $B(z)= [0,1] H(t,z)$
is of Cartwright class, then it has the product representation
\begin{align}
    B(z)=  - z \int_0^t \frac{1}{2y_s}\,ds 
    \lim_{r\to \infty} 
    \prod_{0<|\lambda_k|<r} \left(1-\frac{z}{\lambda_k}\right),
\end{align}
where $\lambda_k$, $k\in \ZZ$ are the ordered sequence of zeros of $B(z)$ with $\lambda_0=0$.
\end{proposition}

\begin{proof}
By Proposition \ref{prop:E} we have $B(0)=0$ and 
\[
B'(0)=[0,1] \int_0^1  J^t R(s) \mathfrak u_0 ds=-\int_0^t \frac{1}{2y_s}ds.
\]
Since $B$ is of Cartwright class by assumption, it has the product representation \eqref{eq:C_product}. By Lemma \ref{lem:Polya} $B(z)$ is real for $z\in \R$, and it has real zeros.  Since $y_s>0$ we have $B'(0)< 0$.  From these it follows   that $c=B'(0)$, $m=1$, $b=0$ in the product representation \eqref{eq:C_product}, which is the statement of the proposition. 
\end{proof}

\section{Characteristic polynomials of unitary matrices}
\label{s:Verblunsky}

Section 5 of \cite{BVBV_sbo} provides a connection between discrete measures on the circle and Dirac operators. Here we explain how this connection ties orthogonal polynomials to the structure function, and the characteristic polynomial to the secular function.


Let $\mu$ be a probability measure whose support is exactly $n$ points $e^{i\lambda_j}, 1\le j \le n$ on the unit circle centered at zero in $\mathbb C$, and assume $\mu(\{1\})=0$. The characteristic polynomial of $\mu$, normalized at $1$,  is be defined as
\begin{align}\label{eq:charpol}
p(z)=p_\mu(z)=\prod_{j=1}^n\frac{z-e^{i\lambda j}}{1-e^{i\lambda j}}.
\end{align}
For $k\le n$, the $k$th orthogonal polynomial $\psi_k(z)$ normalized at 1, is defined as the unique polynomial with $\psi_k(1)=1$ of degree $k$ that is orthogonal to $1,\ldots,z^{k-1}$ in $L^2(\mu)$. With this definition,  $p=\psi_n$. Let $[\psi_k]$ be the main coefficient of $\psi_k$, and for $0\le k\le n-1$ let $\gamma_k\in \mathbb C$, $w_k,v_k\in \mathbb R$  be so that
$$
\frac{2\gamma_k}{1-\gamma_k}= w_k-iv_k=-\frac{2\psi_{k+1}(0)}{\overline{[\psi_k]}}.
$$
The $\gamma_k$ are called the modified, or deformed, Verblunsky coefficients of the measure $\mu$ introduced by Bourgade, Najnudel and Rouault \cite{BNR2009}. They satisfy $|\gamma_k|<1$, for $0\le k\le n-2$, and $|\gamma_{n-1}|=1$. 

Let $x_0=0$, $y_0=1$, and define recursively
\begin{align}\label{xyrec}
x_{k+1}=x_k+v_k y_k, \qquad y_{k+1}=y_k(1+w_k).
\end{align}
Note that $y_k>0$ for $1\le k\le n-1$ and $y_n=0$. 

The next proposition shows how the orthogonal polynomials $\psi_k$ can be expressed using a Dirac operator built from the path $x_{\lfloor nt \rfloor}+i y_{\lfloor nt \rfloor}$,  $t\in [0,1]$.


\begin{proposition}
\label{prop:unitary_car2b}
Set $x(t)+i y(t)=x_{\lfloor nt \rfloor}+i y_{\lfloor nt \rfloor}$ for $t\in [0,1]$.
Let
\begin{align}\label{eq:dscrtDirop}
\btau = R^{-1} \mat{0}{-1}{1}{0} \frac{d}{dt}, \qquad R=\frac{X^t X}{2\det X}, \qquad X=\mat{1}{-x}{0}{y},
\end{align}
with boundary conditions $\mathfrak  u_0  =[1,0]^t$, $\mathfrak  u_1 =[-x(1),-1]^t$.%
For $(t,z)\in [0,1]\times \mathbb C$ let  $H(t,z)\in\CC^2$ be the unique solution of
\begin{equation}\label{eq:H-discrete}
\btau H= z H, \qquad H(0,z)=[1,0]^t.
\end{equation}
\begin{enumerate}[(i)]
\item The orthogonal polynomials $\psi_k, 0\le k\le n$ satisfy
\begin{equation}\label{e:orthogonal}
\psi_k(e^{iz/n})=e^{izk/(2n)}\,[1,-(x_k+iy_k)]\,H(k/n,z).
\end{equation}
In particular the normalized characteristic polynomial satisfies
$$p(e^{iz/n})=\psi_n(e^{iz/n})=e^{iz/2}\,[1,-x_n]\,H(1,z)
$$
\item The secular function of $\btau$ is $\zeta_\btau(z)=p(e^{iz/n})e^{-iz/2}$. 

\item The spectrum of $\btau$ is given by the set $\{n \lambda_k+2\pi n j: 1\le k\le n, j\in \ZZ\}$\, see also Proposition 17 from \cite{BVBV_sbo}.
 \item We have $$\zeta_\btau(z)=
 \prod_{j=1}^n \frac{\sin(\lambda_j/2-z/(2n))}{\sin(\lambda_j/2)}.
 $$
\end{enumerate}
\end{proposition}
Note that $\|R^{-1}\|$ is bounded on $[0,1]$, hence the solution of \eqref{eq:H-discrete} is indeed unique. 
\begin{proof}
Let $\psi^*_k(u)=u^k \overline{\psi_k(1/\bar u)}$ be the reversed polynomials. These polynomials satisfy the modified Szeg\H{o} recursion, see \cite{BVBV_sbo} and also \cite{BNR2009}:
\begin{align}\label{eq:mSzego}
\binom{\psi_{k+1}}{\psi_{k+1}^*}= A_k\mat{u}{0}{0}{1} \binom{\psi_k}{\psi_k^*}, \qquad \binom{\psi_0}{\psi_0^*}=\binom{1}{1}, \qquad 0\le k\le n-1,
\end{align}
with  
\[
A_k=\mat{\frac{1}{1-\gamma_k}}{-\frac{\gamma_k}{1-\gamma_k}}{-\frac{\bar \gamma_k}{1-\bar \gamma_k}}{\frac{1}{1-\bar \gamma_k}}
=
\mat{1}{0}{0}{1}+\frac12\mat{w_k-iv_k}{-w_k+iv_k}{-w_k-iv_k}{w_k+iv_k}.
\]
With  $U=\mat{1}{-i}{1}{i}$ and the notation  $Y^X=X^{-1}YX$  we have
\begin{align}\label{eq:Apathrec}
A_k^{U}=\mat{1}{-v_k}{0}{1+w_k}, \qquad 0\le k\le n-1,
\end{align}
and we set
\begin{align}
 X_0:=I, \qquad X_k:=A_{k-1}^U\cdots A_0^U=\mat{1}{-x_k}{0}{y_k}, \qquad 1\le k\le n.  \label{eq:X_k}
\end{align}
With $u=e^{iz/n}$ let
\begin{equation}\label{eq:H-def}
H_k(z)=e^{-\frac{iz k}{2n}}X_k^{-1}U^{-1}\binom{\psi_k(u)}{\psi_k^*(u)}, \qquad 0\le k\le n-1.
\end{equation}
The functions $H_k$ satisfy the recursion
\begin{align}\label{eq:H_rec}
 H_{k+1}=X_k^{-1}\mat{e^{\frac{iz }{2n}}}{0}{0}{e^{\frac{-iz }{2n}}}^{U} X_k H_k, \qquad H_0=\binom{1}{0},\qquad 0\le k\le n-1.
\end{align}
Define $H(t,z)$ for $0\le t\le 1$ as follows:
\begin{align*}
H(\tfrac{k}{n},z)&=H_k(z), \qquad k=0,1,\dots,n-1,\\
H(t,z)&=X_{k}^{-1} \mat{e^{\frac{iz}{2}(t-k/n))}}{0}{0}{e^{\frac{-iz }{2}(t-k/n)}}^{ U} X_k\, H_k, \qquad t\in (k/n,(k+1)/n].
\end{align*}
The function $H(t,z)$ is continuous in $t$ and analytic in $z$. For $t$ in $t\in[k/n,(k+1)/n]$ differentiating in $t$ we get
\begin{align}\label{eq:ODEHdef}
H'=\frac{z}{2} X_{k}^{-1}\mat{i}{0}{0}{-i}^{U} X_k H=\frac{z}{2} {\mat{0}{1}{-1}{0}}^{X_k} H=z \mat{0}{1}{-1}{0} R_t H,
\end{align}
where we used \eqref{eq:dscrtDirop} and \eqref{eq:x-conj}. This means that $H$ is the unique solution of \eqref{eq:H-discrete}, and \eqref{eq:H-def} implies { (i)} for $0\le k\le n-1$. From \eqref{eq:mSzego}-\eqref{eq:ODEHdef} and the fact that $y_n=0$ it follows that 
\[
\binom{\psi_n}{\psi_n^*}=e^{\frac{iz}{2}} U X_n H(1,z)=e^{\frac{iz}{2}} \mat{1}{-x_n}{1}{-x_n} H(1,z).
\]
This implies { (i)} for $k=n$. The operator $\btau$ with $\mathfrak  u_0  =[1,0]^t$, $\mathfrak  u_1 =[-x_n,-1]^t$ satisfies Assumptions 1 and 2, hence by Proposition \ref{prop:zetaODE} we have {(ii)}.
Since the eigenvalues of $\btau$ are the zeros of $\zeta_\btau$, we get {(iii)}. Finally, {(iv)} follows from the identity
\[
\frac{e^{iz/n}-e^{i\lambda}}{1-e^{i\lambda}}=e^{iz/2n} \frac{\sin(\lambda/2-z/(2n))}{\sin(\lambda/2)}.
\qedhere \]
\end{proof}

\section{The stochastic zeta function}\label{s:sineop}

The $\Sineop$ operator was introduced in \cite{BVBV_sbo}. It is a Dirac operator on $[0,1)\to \R^2$ functions built from  standard hyperbolic Brownian motion. In \cite{BVBV_sbo}  it was shown that its spectrum is given by the bulk scaling limit of the Gaussian and circular beta ensembles. In \cite{BVBV_op} it was shown that the operator can be derived as the limit Dirac operators constructed from finite circular beta ensembles. 

Recall the definition of the $\Sineop$ operator from Theorem 25 of \cite{BVBV_sbo}. 
\begin{definition}\label{def:Sinop}
Let $\Xi_u, u\ge 0$ be standard hyperbolic Brownian motion in $\HH=\{z\in \CC: \Im z>0\}$ with $\Xi_0=i$. Let $\Xi_{\infty}=\lim_{u\to \infty} \Xi_u$ be the a.s.~limit of $\Xi$. Then $\Sineop$ is the Dirac operator built from the path $\Xi_{-\frac{4}{\beta}\log(1-t)}, t\in [0,1)$ with boundary conditions $\mathfrak  u_0  =[1,0]^t$, $\mathfrak  u_1 =[-\Xi_\infty, -1]$.
\end{definition}

We will study a conjugate of this operator that fits into our framework.

\subsection{$\btau_\beta$ and the stochastic zeta function}

Let $b_1,b_2$ be independent copies of two-sided Brownian motion on $\R$, and set 
\begin{align}\label{eq:xy}
y_u=e^{b_2(u)-u/2}, \qquad     x_u=\begin{cases}
-\int_u^0 e^{b_2(s)-\tfrac{s}{2}} d b_1,\qquad & u\le 0,\\
\;\;\;\,\int_0^u e^{b_2(s)-\tfrac{s}{2}} d b_1 &u\ge0.
\end{cases}
\end{align}
Note that $x_u+i y_u, u\in \R$ is an almost surely continuous process in $\HH$. 
\begin{remark}
The process 
\begin{equation}
\label{eq:X}
X_u=\left(\begin{array}{cc}
 1 &-x_u \\ 0 & y_u 
 \end{array}
\right), \qquad u\in \R,
\end{equation}
is two-sided Brownian motion in the affine group of matrices of the form 
\begin{align}\label{eq:affine}\left(
\begin{array}{cc}
 1 &-x \\ 0 & y 
 \end{array}
\right)\qquad x,y\in \mathbb R, y>0.
\end{align}
The increments $X_uX_s^{-1}$ are stationary and independent over disjoint time intervals $(s,u)$. In particular, $y$ is two-sided geometric Brownian motion.  Moreover, $x_u+i y_u, u\ge 0$ is standard hyperbolic Brownian motion in $\HH$ started from $i$, while  $x_{-u}+i y_{-u}, u\ge 0$ is the same process conditioned to converge to $\infty$.
\end{remark}
Consider the time-change
\begin{align}\label{eq:timechng}
   u(t)=\tfrac{4}{\beta} \log t .
\end{align}
Let $q$ be a standard Cauchy distributed random variable independent of $b_1, b_2$. We consider the  random Dirac operator built from $X_{u(t)}, t\in(0,1]$ with $q$-boundary conditions \eqref{eq:qbndry}
and its secular function.
\begin{definition}\label{def:zetab}
Let $\mathfrak{u}_0=[1,0]^t$, $\mathfrak{u}_1=[-q,-1]^t$, and set
$$
\btau_\beta=\Dirop(X_{u(\cdot)},\mathfrak  u_0  , \mathfrak  u_1 ),\qquad  \zeta_\beta=\zeta_{\btau_\beta}.
$$
We call $\zeta_\beta$ the \textbf{stochastic zeta function}.
\end{definition}
To see that $\zeta_\beta$ is well defined, we  need to check  Assumptions 1 and 2, this will be done in Section \ref{subs:stochzeta} below. 
To motivate the definition we first show the  $\btau_\beta$ is orthogonal equivalent to the  $\Sineop$ operator. 


\subsection{The $\Sineop$ operator is orthogonal equivalent to $\btau_\beta$ }

Our goal is to show that $\Sineop$ and $\btau_\beta$ are orthogonal equivalent operators. To do this we first review  how a Dirac operator behaves under simple transformations of its parameters. Then we discuss the relationship between affine and hyperbolic Brownian motion and their time reversal.  

We consider three transformations. The time reversal transformation $\rho$ maps a function from $(0,1]$ to any space to a function from $[0,1)$ by reversing its time $\rho f(t)= f(1-t)$. The transformation from the affine group \eqref{eq:affine} to itself
\begin{align}
\iota: X\mapsto SXS,\qquad S=S^{-1}=\mat{1}{0}{0}{-1} 
\end{align}
simply reverses the sign of the $(1,2)$ entry of $X$. This is just the reflection $z\to -\bar z$ for the corresponding element of $\HH$.  It is an automorphism of the group. Finally, given a $2\times 2$ orthogonal matrix $Q$ of determinant 1, the corresponding linear fractional transformation $\mathcal Q$ maps $z\in \overline \HH$ to the ratio of entries of $Q[z,1]^t$. If we identify the matrix \eqref{eq:affine} with the complex number  $x+iy$, then $\mathcal Q$ acts on the affine group. Further, it acts on paths in the affine group pointwise.

The proof of the following lemma is just straightforward arrow chasing and simple calculation, so we omit it. 
\begin{lemma}\label{l:ort-eq}
Given $X$ and boundary conditions $\mathfrak  u_0  ,\mathfrak  u_1 \in \mathbb R\cup \{\infty\}$, we have the following identities.
\begin{align*}
\rho^{-1} \Dirop(X,\mathfrak  u_0  ,\mathfrak  u_1 ) \rho&=-\Dirop(\rho  X,\mathfrak  u_1 ,\mathfrak  u_0  ),\\
S\, \Dirop(X,\mathfrak  u_0  ,\mathfrak  u_1 ) S&=-\Dirop(\iota X,-\mathfrak  u_0  ,-\mathfrak  u_1 ),\\
Q \Dirop(X, \mathfrak  u_0  ,\mathfrak  u_1 ) Q^{-1}&=\Dirop(\mathcal Q X,\mathcal Q \mathfrak  u_0  ,\mathcal Q \mathfrak  u_1 ).
\end{align*}
In particular, the operator $\Dirop(X, \mathfrak  u_0  ,\mathfrak  u_1 )$ is orthogonally equivalent to the operators $\Dirop(\mathcal Q X,\mathcal Q \mathfrak  u_0  ,\mathcal Q \mathfrak  u_1 )$ and $\Dirop(\rho \iota X,-\mathfrak  u_1  ,- \mathfrak  u_0 )$ in the respective $L^2$ spaces, and they have the same integral traces \eqref{eq:int_tr} as well. 
\end{lemma}

\begin{lemma}\label{l:HBM}
Let $X$ be two-sided affine Brownian motion defined via \eqref{eq:xy}-\eqref{eq:X}, and let $q$ an independent standard Cauchy random variable. The orthogonal matrix  
$$
Q=\frac{1}{\sqrt{q^2+1}}\left(
\begin{array}{cc}
 q&1 \\ -1 &q \end{array}
\right).
$$ 
corresponds to a fractional linear transformation   $\mathcal Q$ that maps $q$ to $\infty$.
Then 
$$ u\mapsto\mathcal Q X_{- u}, \qquad  u\ge 0$$ 
is hyperbolic Brownian motion started from $i$, and  $\iota \mathcal Q X_{- u}$ has the same law. Moreover, as $u\to \infty$ the hyperbolic Brownian motions $(\mathcal Q X)_{- u}$,  $(\iota \mathcal Q X)_{- u}$ converge to the boundary points $\mathcal Q(\infty)$, and $-\mathcal Q(\infty)$, respectively. 
\end{lemma}
\begin{proof}
The transformation $\mathcal Q$ is a hyperbolic rotation of $\HH$ about the point $i$, mapping $q$ to $\infty$ when extended to $\bar \HH$.  

The first claim follows from the well-known disintegration theorem about hyperbolic Brownian motion into two independent pieces,  Proposition X.3.1 in \cite{hypBM}.   The first is the boundary point that it converges to. This has Cauchy distribution.  The second is the process rotated to converge to $\infty$, which has the same distribution as hyperbolic Brownian motion conditioned to converge to $\infty$.

The second claim follows from the fact that the law of hyperbolic Brownian motion is invariant under the reflection $z\mapsto -\bar z$. $X_{- u}$,$u\ge 0$ is hyperbolic Brownian motion conditioned to converge to $\infty$, and the limit points follow the transformations, demonstrating the last two claims. 
\end{proof}

\begin{proposition}\label{prop:Sine_tau}
Let $\btau_\beta$ be as in Definition \ref{def:zetab} and let $Q$ and the corresponding fractional linear transformation $\mathcal Q$ be defined as in Lemma \ref{l:HBM}. 
Then the operator $\rho^{-1} (S \mathcal Q)\btau_\beta (S \mathcal Q)^{-1} \rho$ is orthogonal equivalent to $\btau_\beta$ and it has the same distribution as the $\Sineop$ operator defined in Definition \ref{def:Sinop}. In particular, its eigenvalues agree with those of $\btau_\beta$ and have the law of the  $\Sineb$ process.
\end{proposition}
\begin{proof} 
By Lemma \ref{l:ort-eq} the operator $\rho^{-1} (S \mathcal Q)\btau_\beta (S \mathcal Q)^{-1} \rho$ is given by 
\begin{align}\label{eq:transformed_tau}
    \Dirop(\rho \iota \mathcal Q(X_{u(\cdot)}),\infty, -\mathcal Q(\infty)).
\end{align}
By Lemma \ref{l:HBM}, $\mathcal (\iota \mathcal QX)_{-u}, u\ge 0$ is standard hyperbolic Brownian motion. Hence
\[\rho \iota \mathcal Q(X_{u(\cdot)})\ed\rho \Xi_{-u(\cdot)}=\Xi_{-\frac{4}{\beta}\log(1-\cdot)}. \]
where $\Xi$ is standard hyperbolic Brownian motion. 
\end{proof}

\subsection{The stochastic zeta function and its approximations } \label{subs:stochzeta}

The operator $\btau_\beta$ has approximate versions that will be useful in the sequel. They are defined in terms of the increment 
\begin{equation}\label{eq:Xnu}
    X^\nu_u=X_uX_\nu^{-1}, \qquad u\ge \nu
\end{equation}
of the process $X$ on the interval $[\nu, 0]$ with $\nu<0$. We have
\begin{equation}
\label{eq:X_nu}
X_u^{\nu}=\left(\begin{array}{cc}
 1 &-x_u^\nu \\ 0 & y_u^\nu 
 \end{array}
\right),\qquad 
x_u^\nu=\frac{x_u-x_{\nu}}{y_\nu}, \quad y_u^\nu= \frac{y_u}{y_\nu}. 
\end{equation}
We set  
\begin{align*}
\btau_{\beta,\nu}=\Dirop(X^{\nu}_{u(\cdot)},\infty, q),\qquad  \zeta_{\beta,\nu}=\zeta_{\btau_{\beta,\nu}}.
\end{align*}
Note that the operator  $\tau_{\beta,\nu}$ acts on $\mathbb R^2$-valued functions on the  interval 
$[t(\nu),1]$, 
where 
\begin{align}\label{eq:inversetimechange}
  t(u)=e^{4u /\beta}  
\end{align}
is the inverse of the time-change function $u$ \eqref{eq:timechng}.

By convention, the $\nu=-\infty$ case will refer to the undecorated $\btau_\beta$. The finite-$\nu$ operators and the corresponding stochastic zeta functions are better behaved than $\btau_\beta$ and $\zeta_\beta$. See, for example  Section \ref{s:moments}.

\begin{proposition}\label{prop:tau-beta}
Almost surely  the operator $\btau_{\beta,\nu}$ satisfies Assumptions 1-3 for all  $\nu\in [-\infty,0)$. In particular,  $\zeta_{\beta, \nu}$, $\nu\in [-\infty,0)$ are well-defined entire functions with probability one.  
\end{proposition}

\begin{proof}
Since $x_u+i y_u, u\in \R$ is a continuous process in $\HH$ with probability one, this is also true for $x^\nu+i y^\nu$ for all $\nu\in (-\infty,0)$. The continuity of $x+i y$ and $x^\nu +i y^\nu$ implies Assumption 1. 
Assumption 3 is  satisfied  as we are using $q$-boundary conditions \eqref{eq:qbndry}.

For $\nu\in (-\infty,0)$ Assumption 2 follows from the fact that $x^\nu_{u(\cdot)}+i y^\nu_{u(\cdot)}$ is continuous on the compact interval $[e^{4\nu/\beta},1]$ with probability one. Hence we only need to verify the assumption for $\nu=-\infty$. 
For this we need to check  \eqref{eq:Asmp_2_1} and \eqref{eq:Asmp_2_2} on the interval  $(0,1]$ using the process $x_{u(\cdot)}+i y_{u(\cdot)}$.
We will show that  for a given $\eps>0$ there is a random constant $C>0$ so that for all $u\in (-\infty,0]$ we have
\begin{align}\label{eq:hypBM_bnd}
C^{-1} e^{-\eps |u|}\le y_{u} e^{u/2}\le C e^{\eps  |u|}, \qquad |x_{ u} e^{- u/2}|\le C e^{\eps  |u|}.
\end{align}
From these the integral bounds of \eqref{eq:Asmp_2_1} and \eqref{eq:Asmp_2_2} follow almost surely for $x_{u(\cdot)}+i y_{u(\cdot)}$. 

The first bound in \eqref{eq:hypBM_bnd} follows from the definition \eqref{eq:xy} and the law of iterated logarithm for Brownian motion, see Theorem \ref{thm:lil} in the appendix. From \eqref{eq:xy} it follows that $x_{-u}, u\ge 0$ has the same distribution as a Brownian motion run with the time $u\to \int_{-u}^0 y^2_s ds$. Using the first bound of \eqref{eq:hypBM_bnd} with the law of iterated logarithm again we obtain the bound on $x$ in \eqref{eq:hypBM_bnd}.
\end{proof}

\begin{remark}\label{rem:large_int}
We defined $\btau_\beta$ from $X_{u(\cdot)}, t\in (0,1]$. The same argument as in the proof of Proposition \ref{prop:tau-beta} shows that for any fixed $\sigma>1$ the operator defined from $X_{u(\cdot)}, t\in (0,\sigma]$ also satisfies Assumptions 1-3, and the same holds for the operators defined from $X^\nu_{u(\cdot)}, t\in (t(\nu),\sigma]$.
\end{remark}


\subsection{Convergence of random characteristic polynomials } \label{subs:char_circ_beta}

The size $n$ circular beta ensemble is a random vector of unit length complex numbers with joint density proportional to 
\[
\prod_{1\le j< k\le n} |z_j-z_k|^\beta.
\]
on $\{|z|=1\}^n$. For $\beta=2$ this has the same  distribution as the eigenvalues of a Haar distributed unitary $n\times n$ matrix. 
For general $\beta$, there are explicit five-diagonal unitary matrices $U_\beta$ of CMV type \cite{CMV} with this eigenvalue distribution, see Killip and Nenciu \cite{KillipNenciu}. We denote by 
\begin{align}\label{eq:char_pol}
p_n(z)=p_{\beta,n}(z)=\prod_{j=1}^n \frac{z-z_j}{1-z_j}
\end{align}
the normalized   characteristic polynomial of the size $n$ circular beta ensemble, see \eqref{eq:charpol}.

The next theorem shows that the stochastic zeta function $\zeta_\beta$ is the  limit of the characteristic polynomials of the circular beta ensembles. This restates Theorem \ref{thm:intro-limit}.

\begin{theorem}\label{thm:char_conv} Fix $\beta>0$.
There exists a coupling of the random  polynomials $p_{\beta, n}$ and the  stochastic zeta function $\zeta_\beta$ together with a random variable $C$ so that 
for all $z\in \mathbb C$ and all $n>1$
\begin{align}
|p_n(e^{i z/n})e^{-i z/2}-\zeta_\beta(z)| \le  \left(e^{|z| \frac{\log^3 n}{\sqrt{n}}}-1\right)
C^{|z|^2+1}.\label{circ_cp_bnd}
\end{align}
\end{theorem}
\begin{proof}
Recall the definition of the modified Verblunsky coefficients of a discrete measure supported on $n$ points on the unit circle from Section \ref{s:Verblunsky}.

Consider the following random discrete measure $\sigma_{n,\beta}$: the support of the measure is given by a size $n$ circular beta ensemble $z_1, \dots, z_n$, and the vector of weights $(\sigma_{n,\beta}(z_1), \dots, \sigma_{n,\beta}(z_n))$ is independent and has Dirichlet distribution with parameter  $(\beta/2, \dots, \beta/2)$. 

Bourgade, Najnudel and Rouault  \cite{BNR2009}  -- building on the work of Killip and Nenciu in \cite{KillipNenciu} --  showed that the  modified Verblunsky coefficients corresponding to  $\sigma_{n,\beta}$ are independent of each other, and identified their marginal distributions. 

The construction given in Proposition \ref{prop:unitary_car2b} defines a Dirac operator $\Circop$ corresponding to $\sigma_{n,\beta}$ acting on functions $[0,1]\to \R^2$,  with eigenvalues given by $n \Lambda_{n,\beta}+2\pi n \Z$ where $e^{i\Lambda_{n,\beta}}$ is a size $n$ circular beta ensemble.  
By Proposition \ref{prop:unitary_car2b} the secular function of the $\Circop$ operator is given by $\zeta_n(z)=p_n(e^{iz/n}) e^{-iz/2}$.

In \cite{BVBV_op} it was shown that there is a coupling of the $\Circop$ operators for $n\ge 1$ and the $\Sineop$ operator so that 
\begin{equation}\label{resbound}
\|\res \Sineop - \res \Circop\|^2_{\textup{HS}} \le  \frac{\log^{6} n}{n}
\end{equation}
holds for all $n\ge N$ with a random variable $N$. 
In this coupling the operators all share the same starting and end conditions $\mathfrak  u_0  =[1,0]^t$ and $\mathfrak  u_1 =[-q,-1]^t$, where $q$ is Cauchy distributed. 


We will use Proposition \ref{prop:zeta_tri} to estimate $|\zeta_n(z)-\zeta_\beta(z)|$. To apply the proposition to operators satisfying Assumptions 1-3 we  first need to conjugate $\Circop$, $\Sineop$ with the transformations  appearing in Proposition \ref{prop:Sine_tau}. However,  Lemma \ref{l:ort-eq} shows that this conjugation does not change the Hilbert-Schmidt norm of the resolvent and the integral trace, hence we can estimate the appropriate quantities directly for $\Circop$, $\Sineop$. 

By  \eqref{resbound} and since  $\|\res \Sineop\|<\infty$, there is a random  $C_0$ with 
\begin{align}\label{eq:resbnd_2}
\|\res \Sineop - \res \Circop\|_{\textup{HS}} \le C_0 \frac{\log^{3} n}{\sqrt{n}}, \qquad \|\res \Sineop\|_{\textup{HS}},  \| \res \Circop\|_{\textup{HS}}\le C_0
\end{align}
for all $n> 1$. We  need similar bounds for the integral traces of $\res \Sineop$ and $\res \Circop$,  denoted by $\mathfrak{t}_\beta$ and $\mathfrak{t}_n$, respectively. The tools developed in  \cite{BVBV_op} to prove \eqref{resbound}  also imply 
\begin{align}\label{tr_bnd}
|\mathfrak{t}_\beta-\mathfrak{t}_n|\le \frac{\log^3 n}{\sqrt{n}}
\end{align}
for $n\ge N_1$, see Proposition \ref{prop:circtrace} below. Since $\mathfrak t_\beta$ is a.s.~finite, this implies the existence of a random  $C_1$ with 
\begin{align}\label{eq:tr_bnd_2}
|\mathfrak{t}_\beta-\mathfrak{t}_n| \le C_1 \frac{\log^{3} n}{\sqrt{n}}, \qquad |\mathfrak{t}_\beta|, |\mathfrak{t}_n|\le C_1
\end{align}
for all $n> 1$.
The bound \eqref{circ_cp_bnd} now follows from  \eqref{eq:resbnd_2}, \eqref{eq:tr_bnd_2} and Proposition \ref{prop:zeta_tri}.  
\end{proof}
\begin{proposition}\label{prop:circtrace} In the coupling of \cite{BVBV_op} the bound \eqref{tr_bnd} holds for $n\ge N_1$ random.
\end{proposition}
\begin{proof}
 Let  $ \cB_n(t)\in \HH, t\in [0,1]$ denote the path corresponding to $\Circop$ built from the random modified Verblunsky coefficients according to \eqref{xyrec}. Let $ \cB(t)\in \HH, t\in [0,1)$ denote the time-changed hyperbolic Brownian motion in the construction of $\Sineop$,  see Definition \ref{def:Sinop}. In \cite{BVBV_op} the coupling of the $\Circop$, $\Sineop$  operators is constructed in a way that the  processes $ \cB_n,  \cB$ are close  to each other in the hyperbolic metric.  More precisely, Proposition 13 of \cite{BVBV_op} states that there is a random $N_0$ so that for all $n\ge N_0$ we have the uniform bounds
 \begin{align}\label{eq:coupling_1}
     d_{\HH}(\cB_n(t), \cB(t))\le& \frac{\log^{3-1/8}n}{\sqrt{(1-t)n}}, 
    \qquad 0\le t\le T_n=1-\tfrac{1}{n}\log^6 n,\\
    d_{\HH}(\cB_n(T_n), \cB(t))\le& \frac{144}{\beta} (\log \log n)^2, \qquad T_n\le t<1.\label{eq:coupling_2}
 \end{align}
 Here $d_{\HH}$ denotes the hyperbolic distance.

Proposition 15 of \cite{BVBV_op} gives a way to estimate the effect of a truncation on the Hilbert-Schmidt norm of an integral operator constructed from a path that converges to a boundary point of the hyperbolic plane. Lemma 21  of \cite{BVBV_op} shows that $\cB$ satisfies the conditions of Proposition 15  of \cite{BVBV_op}.

Proposition 16  of \cite{BVBV_op} gives an estimate on the Hilbert-Schmidt norm of the difference of two integral operators where the paths from which they are constructed are close enough. The bound (\ref{resbound}) is proved by putting together the results of Propositions 13, 14, 15 and 16  of \cite{BVBV_op}. To get the estimate (\ref{tr_bnd}) we can use some of the intermediate steps in these propositions. We need to estimate
\begin{align*}
\left|\int_0^1 (a_n(s)^t c_n(s) -a(s)^t c(s) ) ds \right|,
\end{align*}
where $a_n, c_n, a, c$ are the vector-valued functions \eqref{eq:ac} defined for $\Circop$ and $\Sineop$. By the triangle inequality  it is enough to bound the following three integrals:
\begin{align}\label{eq:three_terms}
\int_0^{T_n} |a_n(s)^t c_n(s) -a(s)^t c(s) |ds , \qquad \int_{T_n}^1\left| a_n(s)^t c_n(s)\right| ds , \qquad \int_{T_n}^1\left| a(s)^t c(s)\right| ds.
\end{align}
The proofs of Propositions 14, 15, and Lemma 21 of \cite{BVBV_op} imply the  following bounds on $a, c$:
\begin{align}
  |a(s)|\le C (1-s)^{\frac{1}{\beta}-\eps}, \qquad |c(s)|\le C (1-s)^{-\frac1{\beta}-\eps}, \qquad s\in [0,1), \label{eq:acbounds}  
\end{align}
 with $\eps>0$ arbitrarily small, and $C=C_\eps$  a finite random variable.  
 
The arguments in the proof of  Proposition 16 of \cite{BVBV_op} imply the bounds
 \begin{align*}
 \frac{|a(s)-a_n(s) |}{|a(s)|}, \frac{|c(s)-c_n(s) |}{|c(s)|}&\le 2 \sinh(\tfrac12 d_{\HH}(\cB_n(s), \cB(s))), \qquad 0\le s<1,\\[3pt]
 \frac{|a_n(s)-a_n(t) |}{|a_n(s)|},  \frac{|c_n(s)-c_n(t) |}{|c_n(s)|}&\le 2 \sinh(\tfrac12 d_{\HH}(\cB_n(s), \cB_n(t))), \qquad 0\le s<t<1.
 \end{align*}
The coupling bounds \eqref{eq:coupling_1} and \eqref{eq:coupling_2}
give 
\begin{align*}
    \frac{|a(s)-a_n(s) |}{|a(s)|}, \frac{|c(s)-c_n(s) |}{|c(s)|}&\le c \frac{\log^{3-1/8} n}{\sqrt{(1-s) n}}, \qquad 0\le s\le T_n,\\
    \frac{|a_n(T_n)-a_n(t) |}{|a_n(T_n)|},  \frac{|c_n(T_n)-c_n(t) |}{|c_n(T_n)|}&\le 
    e^{c (\log \log n)^2}
    , \qquad  T_n<t<1,
  \end{align*}
for all $n\ge N_0$. Together with \eqref{eq:acbounds} these bounds are sufficient to estimate all three terms in \eqref{eq:three_terms} with a repeated use of the triangle inequality. Taking $\eps$ small enough in \eqref{eq:acbounds}  leads to the bound \eqref{tr_bnd}.
\end{proof}

\section{SDE characterization of $\zeta_\beta$} \label{s:SDE_zeta}

Proposition \ref{prop:zetaODE} gives an ordinary differential equation description of $\zeta_\beta$. In this section we obtain a description using stochastic differential equations.




\subsection{Stochastic differential equation description of  $\zeta_\beta$}

Let $X$ be defined as in \eqref{eq:X}, and $u=u(t)$ as in \eqref{eq:timechng}. 
We set
\begin{align}\label{eq:R_sde}
R_t=-\frac12  JX_{u}^{-1} J X_{u}, 
\end{align}
according to \eqref{eq:Rxy} and \eqref{eq:x-conj}.
Proposition \ref{prop:zetaODE} states that for every $z\in \CC$ there is a  unique vector-valued solution  $H: (0,1]\times \CC\to \CC^2$ of the ordinary differential equation
\begin{align}\label{eq:ODE_11}
R^{-1} J \frac{d}{dt} H  =z H, \qquad t\in (0,1], \qquad \lim_{t\to 0} H(t,z)=[1,0]^t
\end{align}
and
\begin{align}\label{zetaH2}
\zeta_\beta(z)=-H(1,z)^tJ [q,1]= [1,-q] H(1,z).
\end{align}

We consider two approximations of $H$. The first one, $H_\eps$, for  $0<\eps<1$ is the approximation introduced in Proposition \ref{prop:Heps-convergence}. This is the unique solution of the differential equation \eqref{eq:ODE_11} on $[\eps,1]$ with initial condition $H_\eps(\eps,z)=[1,0]^t$.

The second approximation is constructed from the process $X^\nu$ for $\nu<0$, introduced in  \eqref{eq:Xnu}. Recall the definition of $t(\cdot)$ from \eqref{eq:inversetimechange}, 
and define
\begin{align}\label{eq:R_nu}
R_t^{t(\nu)}=-\frac12  J(X^\nu_{u})^{-1} J X^\nu_{u}.
\end{align}
Now define the functions $H^{t(\nu)}$ as  the unique solutions of the  differential equation
\begin{align}\label{eq:nu_ODE}
(R^{t(\nu)})^{-1} J \frac{d}{dt} H  =z H, \qquad t\in [t(\nu),1],   \qquad H^{t(\nu)}(t(\nu),z)=[1,0]^t.
\end{align}
Note that we have $\zeta_{\beta,\nu}=[1,-q] H^{t(\nu)}(1,z)$.

The two approximations are connected via the identity 
\begin{align}\label{eq:H_H}
H^{t(\nu)}(t,z)= X_\nu  H_\eps(t,z),\qquad  t\in [t(\nu),1], \eps=t(\nu).
\end{align}
Define
\begin{align}\label{def:V}
\cH_u(z)=X_{u}H(t(u),z),\qquad \cH_u^\nu(z)=X^\nu_{u}H^{t(\nu)}(t(u),z).
\end{align}
The process $\cH$ also describes $\zeta_\beta$, since $\cH_0=H(1,\cdot)$, and so
\begin{align}
\zeta_\beta=[1,-q]\cH_0. \label{eq:zeta_cH}
\end{align}



\begin{proposition}\label{prop:HSDE}
Consider the independent copies of two-sided Browninan motion $b_1, b_2$ from \eqref{eq:xy}, and let  $\mathcal F_u$ be the $\sigma$-field generated by the  increments  $b_k(u)-b_k(s), s<u$, $k=1,2$. 

The processes $X^\nu$, $\mathcal H^\nu$ and $\mathcal H$ are all adapted to the filtration $\mathcal F_u, u\in \mathbb R$, and they satisfy the stochastic differential equations 
\begin{align}\label{XnuSDE}
d X^\nu&= \mat{0}{-db_1}{0}{db_2} X^\nu, \qquad u>\nu,\\[5pt]
\label{HSDE}
d \cH&= \mat{0}{-db_1}{0}{db_2} \cH-z\tfrac{\beta}{8}e^{\beta u/4} J \cH \, du, \qquad u\in \R.
\end{align}
The processes $\mathcal H^\nu$ satisfy the equation \eqref{HSDE} on $[\nu,\infty)$.  

The processes $\mathcal H, \cH^\nu$ can also be determined as follows.
$\cH^\nu$ 
is the unique strong solution of the  SDE family \eqref{HSDE}  on $[\nu,\infty)$ with boundary condition $\cH^{\nu}_\nu=[1,0]^t$.
Moreover, a.s.~as $\nu\to -\infty$ we have $\cH^{\nu}\to \cH$ uniformly on compacts of $\mathbb R\times \CC$. 
\end{proposition}

\begin{proof}
$X^\nu_u$ is  $\mathcal F_u$-measurable by the definitions \eqref{eq:X} and \eqref{eq:Xnu}.  It\^o's formula shows that it satisfies \eqref{XnuSDE}. 

Note that for $u<0$, $X_u$ is not $\mathcal F_u$-measurable, in fact it is independent of $\mathcal F_u$ as it is built from $b_1(s), b_2(s), s\in [u,0]$.  This issue already arises with two-sided Brownian motion: $b_1(u), b_2(u)$ are not $\mathcal F_u$-measurable either for $u<0$. 
It is not a priori clear that $\cH$ is $\mathcal F$-adapted, because it is  defined in terms of $X$. 
However, $\cH^\nu$ in \eqref{def:V} is defined in terms of $X^\nu$ (see also  \eqref{eq:nu_ODE}), so it is  $\mathcal F$-adapted. The definition and It\^o's formula implies that it satisfies the SDE \eqref{HSDE} with the stated boundary conditions. 

By \eqref{eq:H_H} we have
\[
\cH_u^\nu(z)=X^\nu_u H^{t(\nu)}(t(u),z)=X_u H_\eps(t(u), z)
\]
where $H_\eps$ solves \eqref{eq:ODE_11} on $[t(\nu),1]$ with $H_\eps(t(\nu),z)=[1,0]^t$.
Proposition \ref{prop:Heps-convergence} now implies that as $\nu\to-\infty$, we have $\mathcal H^\nu\to \cH$ uniformly on compact subsets of $\mathbb R \times \mathbb C$. In particular, this shows that  $\mathcal H$ is $\mathcal F$-adapted.  It\^o's formula implies that $\mathcal H$ satisfies \eqref{HSDE}. 

Note that $H$ and $H^{t(\nu)}$ are defined on $(0,1]$ and $[t(\nu), 1]$, respectively, so $\cH, \cH^\nu$ are only defined on $(-\infty,0]$ and $[\nu,0]$ a priori. However, by Remark \ref{rem:large_int} we can extend the definitions for $(-\infty,u(\sigma)]$, $\nu,u(\sigma)]$ for any $\sigma>1$, which allows us to extend the definitions to $\R$ and $(\nu,\infty)$, respectively. 
\end{proof}

Equation \eqref{HSDE} and It\^o's formula implies the following. 

\begin{proposition}\label{prop:H_stat}
Consider  $\cH$ defined in \eqref{def:V}. The random function 
\begin{equation}\label{eq:stat}
(u,z)\mapsto \cH_u(ze^{-\beta u/4})
\end{equation}
is stationary under $u$-shifts. In particular, for each $u\in \R$ the random analytic function $z\to \cH_u(ze^{-\beta u/4})$ has the same distribution as $z\to \cH_0(z)$. 
\end{proposition}



\subsection{The Taylor expansion of $\cH$ and $\zeta_\beta$}

Proposition \ref{prop:HSDE} gives a characterisation of $\cH$ as the uniform on compacts limit of strong solutions \eqref{HSDE} on $[\nu,\infty)$ with $\nu\to -\infty$ with an initial condition $[1,0]^t$. In Corollary \ref{cor:H_char} below we show that $\cH$ is the unique strong solution of \eqref{HSDE} under some additional conditions. The main ingredient for this characterisation is the analysis of the  Taylor expansion of $\cH$ in the variable $z$.

Let 
\begin{align}\label{def:cA_cB}
\cH=[\cA, \cB]^t,    
\end{align}
and recall from \eqref{eq:zeta_cH} that $\zeta_\beta=[1,-q]\cH=\cA-q \cB$, where $\cH, \cA, \cB$ are evaluated at time $u=0$. 
Let $\cA_{n,s},\cB_{n,s}$ be the Taylor coefficients  of $\cA$, $\cB$ at $z=0$ and time $s$. Then 
$$
\zeta_\beta(z)=\sum_{n=0}^\infty(\cA_{n,0}-q \cB_{n,0})z^n.
$$
Surprisingly, any initial sequence of $\cB_1,\cA_1,\cB_2,\ldots$ has a closed SDE description! Moreover, the SDEs can be explicitly solved. 

To describe the solution we first introduce a sequence of processes in Proposition \ref{prop:Taylor}, show that they are well-defined, and provide a.s.~growth bounds. Proposition \ref{prop:Taylor_1} below shows that the introduced processes are actually the Taylor coefficients $\cA_n, \cB_n$.

\begin{proposition}\label{prop:Taylor} Let $b_1, b_2$ be  independent copies of two-sided Brownian motion on $\R$, and set $y_u=e^{b_2-u/2}$ as in \eqref{eq:xy}.  The recursive system $A_0\equiv 1, B_0 \equiv 0$, and 
\begin{align}\nonumber
B_{n}&=y_u\int_{-\infty}^u \tfrac{\beta}{8}\,e^{\beta s/4 }A_{n-1,s}\,  y_s^{-1}\,ds,\\
A_{n}&=\int_{-\infty}^u \tfrac{\beta}{8}\, e^{\beta s/4}\,B_{{n-1},s}\,ds- B_{n,s} \, db_1.\label{eq:AnBn}
\end{align}
is well-defined for $u\le 0$. Moreover, given $a<1/4$ and $\beta_0\in(0,1]$ there is a random constant $C$, so that for all $\beta\ge \beta_0$ a.s.\;for all  $n\ge 0$ and $u\le 0$ we have
\begin{equation}
    \label{eq:induc}
    |A_{n,u}|,|B_{n,u}|\le  \frac{C^n}{n!} e^{a \beta n u}. 
\end{equation}
\end{proposition}

By the scaling properties of the processes $A_k,B_k$, Proposition \ref{prop:Taylor} extends to arbitrary $u\in \R$, and \eqref{eq:induc} holds for $u\le u_0$  with $C$ depending on $u_0$.

\begin{remark}[Dufresne's identity for $B_1$] \label{rem:Dufresne}
We have
\[
B_1(0)=\frac{\beta}{8}\int_{-\infty}^0 e^{-b_2(s)+(\frac{\beta}{4}+\frac12)s}ds\ed \frac{\beta}{2} \int_0^\infty e^{2(b_s-(\frac{\beta}{2}+1) s)}ds\ed \frac{\beta}{4\, G}, \]
where $G$ has Gamma distribution with rate 1 and shape parameter $1+\frac{\beta}{2}$. The last step follows from Dufresne's identity \cite{Dufresne}.
\end{remark}

\begin{proof}[Proof of Proposition \ref{prop:Taylor}.]
We first show by induction that $A_n,B_n$ are well-defined. 
Set $a< 1/4$ and $\beta_0>0$. Set $\eps=\min(1,(1-a)\beta_0/8)$ so that 
\begin{align}
-a\beta+\beta/4-\eps\ge \eps,
\end{align}
for $\beta\ge \beta_0$. For $n\ge 1$ define
$\bar A_n=\sup_{u\le 0} |A_n|/e^{a\beta nu}$, and define $\bar B_n$ similarly. In addition, set $\bar B_n^\eps=\sup_{u\le 0} |B_n|/e^{(a\beta n+\eps)u}$; we set these quantities to $\infty$ if the corresponding $A_n,B_n$ is not well-defined. 

Brownian motion is sublinear, so for each $\eps>0$ there is a random variable $C_\eps$ so that
 \begin{align}\label{eq:C_eps}
 C_{\eps}^{-1} e^{\eps u} \le e^{u} y_u^2\le C_\eps e^{-\eps u}, \qquad u\le 0.
 \end{align} 
If $\bar A_{n-1}<\infty$  then 
\[
|B_{n,u}|\le C_\eps e^{-(1+\eps)/2 u} \int_{-\infty}^u \frac{\beta}{8} e^{(\beta/4+1/2-\eps/2)s} |A_{n-1,s}| ds
\]
so we conclude that 
\begin{align}\label{eq:B_bnd_000}
\bar B_{n}\le 
\bar B_{n}^\eps\le \frac{\beta}{8} \frac{ C_\eps\bar A_{n-1}}{a(n-1)\beta +\beta/4} \le  \frac{ C_\eps\bar A_{n-1}}{8an}. 
\end{align}
The first term in the definition \eqref{eq:AnBn}  of $A_n$ for $u\le 0$ can be bounded as 
\begin{align}
\Big|\int_{-\infty}^u \tfrac{\beta}{8}\, e^{\beta s/4}\,B_{{n-1},s}\,ds\Big|\le \bar A_{n,1} e^{(a+\beta/4) u}, \qquad \bar A_{n,1}:=\frac{\beta\bar B_{n-1}}{8a+2\beta}.
    \label{eq:A_bnd_001}
\end{align}
By Proposition \ref{prop:BR_mart} there is a random variable $Z_n$ with rate 1 exponential tails  \eqref{eq:Z_tail} so that \begin{align}\notag
\Big|\int_{-\infty}^u B_{n,u} db_1\Big|\le
 \frac{2\bar B_n^\eps}{\sqrt{a'}}  e^{a'u} \left(Z_n+ \log(1+2a'|u|)+|\hspace{-0.15em}\log a'| +\log(1+2|\log \bar B_n^\eps|)\right)
\end{align}
with $a'=a'_n=an\beta+\eps$.
This and \eqref{eq:A_bnd_001} implies that $\bar A_n \le \bar A_{n.1}+\bar A_{n,2}$ with 
$$
\bar A_{n,2}=\frac{2\bar B_n^\eps}{\sqrt{a'}}\left(Z_n+|\hspace{-0.15em}\log a'| +\log(1+2|\hspace{-0.15em}\log \bar B_n^\eps|)+\sup_{u\le 0}e^{\eps u} \log(1+2a'
|u|)\right).
$$
We conclude that all $\bar A_n,\bar B_n$ are finite and so all $A_n,B_n$ are well defined. 

The tail bound $P(Z_n>y)<ce^{-y}$ and the Borel-Cantelli lemma shows that 
\begin{align}
  Z:=\sup_{n \ge 1} (Z_n-2\log n)  <\infty. \label{eq:Z_n_bnd}
\end{align}
Recall the definition of $C_\eps\ge 1$ from \eqref{eq:C_eps}, and define  
\begin{align}\notag
    C=\frac{C_\eps^2}{16 a^2}+10^8\left(Z+\log (1+\tfrac1\eps)+\frac{1}{a\beta_0}\right)^2.
\end{align}
We will show by induction that $\bar A_n,\bar B_n \le C^n/n!$, which is equivalent to  \eqref{eq:induc}. This holds for $n=0$ since $C\ge 1$. Assuming the bounds for $n-1$ by  \eqref{eq:B_bnd_000} we get
\begin{align}\label{eq:Bbound_2}
    \bar B_n\le \bar B_n^\eps
    \le \frac{C^{n-1} C_\eps}{n!8a}\le  \frac{C^{n-1/2} }{n!}.
\end{align}
This is stronger than the hypothesis for $B_{n}$, but we need  extra room to bound  $A_n$. Next,
$$
\bar A_{n,1}= \frac{1}{2}\frac{\bar B_{n-1}}{4a/\beta+1}\le \frac{C^{n-1}}{2(n-1)!(4a/\beta_0+1)}\le \frac{C^{n-1}}{2(n-1)!}.
$$
For $\bar A_{n,2}$, change variables $u\to \eps u$ and use $\log(1+ab)\le \log(1+a)+\log(1+b)$ to get
$$\Big(\sup_{u\le 0}e^{\eps u} \log(1+2a'
|u|)\Big) -\log(1+a')-\log(1+\tfrac1\eps) \le \sup_{u\le 0}e^u\log(1-2u)\le 1.$$ Use $\log(1+x)\le 1+|\hspace{-0.15em}\log x|$, \eqref{eq:Z_n_bnd} and  \eqref{eq:Bbound_2} to get 
$$
\bar A_{n,2} \le \frac{2C^{n-1/2}}{n!\sqrt{a'}}\Big(2+Z+\log(1+\tfrac1\eps)+2\log n+2|\hspace{-0.15em}\log a'| +\log(1+2|\log (C^{n-1/2}/n!)|)
\Big).
$$
Finally, use  $|\hspace{-0.15em}
\log x|x^{-1/2}\le 1+x^{-1}$ and $\log\left(1+2\left|\log (x^{n-1/2}/n!)\right|\right)\le 4 x^{1/4} n^{1/2}$
to get
$$
\frac{n!\bar A_{n,2} }{2C^{n-1/2}}\le \frac{2+Z+\log(1+\tfrac1\eps)}{\sqrt{a \beta}}+\frac{4}{\sqrt{a\beta}}+2+\frac{2}{\sqrt{an\beta}}+
\frac{4C^{1/4}}{\sqrt{a\beta}}\le \frac{C^{1/2}}{4}.
$$
Thus  $\bar A_n \le C^{n}/n!$, closing the induction.
\end{proof}

\begin{proposition}\label{prop:Taylor_1}
Consider the Taylor coefficient processes $\cA_{n,u}, \cB_{n,u}$ at $z=0$ for the  $\cA, \cB$ with $\cH=[\cA, \cB]^t$.
They 
satisfy
the following system of stochastic differential equations:
\begin{align}\nonumber
 \cB_0&\equiv 0, \qquad \cA_0\equiv 1,\\ 
 \nonumber 
d\cB_n&=\cB_n db_2-\tfrac{\beta}{8}e^{\beta u/4}\cA_{n-1} du,\\
d\cA_n&=- \cB_n db_1+\tfrac{\beta}{8}e^{\beta u/4}\cB_{n-1} du.   
\label{eq:derivatives}
\end{align}
For any $k$, the equations for the first $k$ elements of the sequence $\cB_0, \cA_0, \cB_1, \cA_1, \dots$ form an autonomous system. In fact, $\cA_n, \cB_n$, $n\ge 0$ is the unique solution of this system so that for $n \ge 1$   $\{|\cB_{n,u}|,u\le 0\}$ is a tight family and  $\cA_{n,u} \to 0$ in probability as $u\to -\infty$. Moreover, $\cA_n=A_n$ and $\cB_n=B_n$, as defined in \eqref{eq:AnBn}.
\end{proposition}
\begin{proof}
Since the SDE system \eqref{HSDE} depends analytically on its parameter $z$, It\^o's formula can be applied to get SDEs for derivatives in this parameter as well, see Protter \cite{Protter} Section V.7. 
Differentiate \eqref{HSDE} $n$ times in $z$ and evaluate at $z=0$ to get 
\begin{align}\nonumber 
d \cH^{(n)}= \mat{0}{-db_1}{0}{db_2} \cH^{(n)}-n\frac{\beta e^{\beta u/4}}{8} J \cH^{(n-1)} \, du.
\end{align}
By definition, $\cH^{(n)}_u=\cH^{(n)}_u(0)=n![\cA_{n,u},\cB_{n,u}]^t$. This shows that $\cA_n,\cB_n$ satisfy  \eqref{eq:derivatives}.

It\^o's formula shows that $A_n, B_n$ from Proposition \ref{prop:Taylor} solve the equations \eqref{eq:derivatives}. Our goal is to show that these are equal to $\cA_n, \cB_n$. 

It follows from Proposition \ref{prop:H_stat} that $\cH_u(e^{-\beta u/4}z)$ is time-stationary. As a consequence,  for $n\ge 1$, $\cB_{n,u}e^{-n\beta u/4}$ is also time-stationary, so $\{\cB_{n,u},u\le 0\}$ is tight. Similarly, $\cA_{n,u}e^{-n\beta u/4}$ is tight, so $\cA_{n,u}\to 0$ in probability as $u\to -\infty$. Using the definition of $A_n, B_n$ we check that $B_{n,u} e^{-n \beta u/4}, A_{n,u} e^{-n \beta u/4}$ are time-stationary, hence $\{B_{n,0}, u\le 0\}$ is tight, and $\cA_{n,u}\to 0$ in probability as $u\to -\infty$.

So it suffices to show that solutions with the stated tightness and convergence conditions are unique. Consider two such solutions, played by $A_n,B_n$ and $\cA_n,\cB_n$ here. 

We show first that $Y=B_1-\cB_1\equiv 0$. Indeed, we have $dY=Y db_2$, so $Y=Cy$ for some random constant $C$ with $y$ as in \eqref{eq:xy}. But $Y$ is tight and $y$ is not, so $C=0$.
Now let $Z=A_1-\cA_1$. Then $dZ=0$, so $Z$ is constant. But $Z_u\to 0$ in probability, so $Z=0$. Repeating this argument inductively for each $n$ gives uniqueness. 
\end{proof}

\begin{corollary}\label{cor:H_char}
Almost surely, for all $(u,z) \in \mathbb R_-\times \mathbb C$ we have
\begin{equation}\label{eq:Hbound}
\|\cH_u(z) - [1,0]^t\|\le e^{Ca^u|z|} -1,
\end{equation}
for some fixed $a=a_\beta\ge 1$ and a random constant $C>0$. In particular, 
we have $\cH_u\to [1,0]^t$ a.s.~as $u\to -\infty$ uniformly on compact subsets of $\mathbb C$. 

Conversely, $\cH$ defined in \eqref{def:V}  is the unique solution of  the system $\eqref{HSDE}$ consisting of entire functions in $z$ so that 
\begin{equation}\label{eq:zinD}
    \sup_{z\in D} |\cH_u(z)-[1,0]^t| \to 0
\end{equation}
in probability as $u\to \infty$ for some open neighborhood $D$ of $0$.
In particular, $\cH$ is the unique solution, consisting of entire functions, so that  the entire function valued process $t\mapsto \cH_u(ze^{-\beta u/4})$ is time-stationary. 
\end{corollary}
\begin{proof}
The bound \eqref{eq:Hbound} follows from writing $H=[\cA,\cB]^t$ as a convergent power series in $z$ and the inequalities \eqref{eq:induc}:
$$
\|\cH_u(z)-[1,0]^t\| \le \sum_{n=1}^\infty |\cA_nz^n|+|\cB_n z^n|\le \sum_{n=1}^\infty 
2 \frac{C^n e^{\beta a n u}|z|^n}{n!}\le e^{2 C e^{\beta a u}|z|}-1.
$$
This also implies \eqref{eq:zinD} for any bounded set $D$.

Now consider an entire function solution of \eqref{HSDE} for which $\sup_{z\in D}|\cH_u-[1,0]^t|\to 0$ in probability. By Cauchy's identity, the Taylor coefficients converge to those of $[1,0]$ in probability, in particular each  $\cB_n$ is tight  and each $\cA_n\to 0$ in probability as $u\to -\infty$. But \eqref{HSDE} implies that the Taylor coefficients satisfy \eqref{eq:derivatives}, which has a unique solution with these conditions. The Taylor coefficients determine $\cH$, so it is also unique. 

For the last statement, time-stationarity implies \eqref{eq:zinD} for any bounded set $D$.
\end{proof}

\subsection{Differential equations for the Taylor coefficient densities}

By Proposition \ref{prop:H_stat}  the random function $(u,z)\mapsto \cH_u(ze^{-\beta u/4})$
is stationary under $u$-shifts. In fact, as a function of $u$, this process is stationary and measurable with respect to the filtration generated by the increments of $b_1, b_2$. 
This observation allows us to give a finite closed system of partial differential equations for the joint density of the first finitely many Taylor coefficients. 

\begin{proposition}\label{prop:StationaryTaylor}
The stationary Taylor coefficients $A_n=e^{-n\beta u/4}\cA_n$, 
$B_n=e^{-n\beta u/4}\cB_n$
satisfy
the following system of stochastic differential equations:
\begin{align}\nonumber
 B_0&\equiv 0, \qquad A_0\equiv 1,\\ 
 \nonumber 
dB_n&=B_n db_2+\left(\frac{\beta}{8} A_{n-1}-\frac{\beta n}{4} B_n\right)du,
\\
dA_n&=-B_n db_1+\left(\frac{\beta}{8}B_{n-1}-\frac{n \beta}{4} A_n \right)du.
\end{align}
In particular, for a given $n\ge 1$ the joint density $p_n(x_1, y_1, \dots, x_n, y_n)$ of $B_1, A_1 \dots, B_n, A_n$ satisfies the following partial differential equation:
\begin{align}\label{e:PDE}
     \sum_{k=1}^n \;\tfrac{1}{2}(\partial^2_{x_k}+\partial^2_{y_k})(x_k^2 p_n)-\tfrac{\beta}{8} \partial_{x_k}\left((y_{k-1}-2k x_k)p_n\right)-
    \tfrac{\beta}{8} \partial_{y_k}\left((x_{k-1}-2k y_k)p_n\right)=0.
\end{align}
Here $x_0=0, y_0=1$, and $x_k\in (0,\infty), y_k\in \R$. The same PDE with the $x_n$ terms dropped is satisfied by the joint density of $B_1, A_1, \dots, B_n$.
\end{proposition}
For $n=1$ we get the following PDE for $p=p_1$. 
\[
x^2 (p_{xx}+p_{yy})+\tfrac{\beta}{2} y p_y+ ((\tfrac{\beta}{2} +4) x-\tfrac{\beta}{4} ) p_x+(\beta +2) p=0.
\]

\begin{proof}
This follows from It\^o's formula, Proposition \ref{prop:Taylor_1} and Kolmogorov's forward equation for Markov processes. 
\end{proof}
The distribution of $B_1$ is given in Remark \ref{rem:Dufresne} by the  Dufresne identity \cite{Dufresne}.

\section{Moment and tail bounds} \label{s:moments_tails}

The goal of this section is to understand the tails of the structure function. The results here will be used later to show analytic properties of $\zeta_\beta$, and to rigorously justify moment computations based on SDEs.

The main tool is the phase function  $\alpha_{\lambda, \nu}$, which, essentially, counts eigenvalues. The bounds will be uniform in the approximation parameter $\nu\in (-\infty,0)$.   They imply similar bounds for the limiting processes.

\subsection{The structure function and the phase function}

Recall the definition of $\cH=[\cA,\cB]^t$ from \eqref{def:V},  and introduce the  processes
\begin{align}\nonumber
\cE&=[1,-i]\cdot \cH=\cA-i \cB,  \\
    \cE^*&=[1,i]\cdot \cH=\cA+i \cB. \nonumber
\end{align}
Note that we have $\cE_0(z)=E(1,z)$ where $E$ is the structure function defined in \eqref{def:E},
we will use the same name for the entire process $\cE$. The name originates in the theory of canonical systems, and refers to an  infinite-dimensional analogue of  orthogonal polynomials on the unit circle. Proposition \ref{prop:H_stat} implies that for $u\in \R$ the function $z\to \cE_u(z)$ has the same distribution as $z\to E(1,e^{\beta u/4} z)$.

From \eqref{eq:zeta_cH} we have
\begin{align}\label{eq:zeta_cE}\zeta_\beta=\mathcal A-q\mathcal B=\frac{1-iq}2\mathcal E + \frac{1+iq}2\mathcal E^*,\end{align}
where the processes are taken at time $u=0$.

For a given $\nu\in (-\infty, 0)$ we use $\cH^\nu$ in the above definitions to obtain $ \cE_\nu$ and $\cE^{*}_\nu$.

The following is an immediate consequence of  Proposition \ref{prop:HSDE}.
\begin{corollary} \label{cor:cE_SDE} Let $\nu\in [-\infty,0)$, and  
\begin{align}
f_\beta(u)=\frac{\beta}{4}e^{\frac{\beta}{4}u}.
\end{align}
The processes $\cE_\nu, \cE^*_{\nu}$ satisfy the stochastic differential equations
\begin{align}\nonumber
d\cE&= \phantom{-}\frac{i z }{2}f_\beta \cE   du\;+ \frac{\cE^*-\cE}{2}(idb_1-db_2),
\\
d\cE^*&=- \frac{i z}{2} f_\beta \cE^*   du+ \frac{\cE^*-\cE}{2}(i db_1+db_2).\label{eq:cE}
\end{align}
For any fixed time $u$ the entire function $z\to \cE_{\nu,u}(z)$   has no zeros on the real line, and $\cE^*=\bar \cE$ there.
For $\Im z<0, u\in \R$ we have 
\begin{align}\label{eq:cE_Polya}
 |\cE_{\nu,u}(z)|\ge |\cE_{\nu,u}^*(z)|.  
\end{align}
For $\nu\not=-\infty$, the processes $\cE_\nu, \cE_\nu^*$ are the unique strong solutions of \eqref{eq:cE} on $[\nu,\infty)$ with initial condition $\cE_{\nu,0}=\cE^*_{\nu,0}=1$. The processes $\cE, \cE^*$ can be obtained as uniformly on compact limits of  $\cE_\nu, \cE_\nu^*$ as $\nu\to -\infty$.
\end{corollary}
The inequality \eqref{eq:cE_Polya} follows from \eqref{E_AB} of Lemma \ref{lem:Polya}.

The real and imaginary parts of $\log \cE_\nu$ for $\nu\in [-\infty,0)$ will play an important role in the upcoming analysis. Since $\cE_\nu$ has no zeros on $\R$ and $\cE_\nu(0)=1$, the branch of the logarithm can be chosen  as the unique continuous function which takes the value  $0$ at $0$. 

\begin{corollary}[The phase function]\label{cor:phase}
For $\nu\in [-\infty,0)$ and $\lambda\in \R$ we can write  $2\log \cE_{\nu,u}(\lambda)=\cL_{\lambda,\nu}(u)+i \alpha_{\lambda,\nu}(u)$, where $\alpha,\mathcal L\in \R$ satisfy the stochastic differential equations
\begin{align}\notag
 d \alpha&=\lambda f_\beta\,du-\Re\left[(e^{-i \alpha}-1) (d b_1+i d b_2)\right]\\&=\lambda f_\beta\, du+2 \sin(\alpha/2) dW_1  \label{eq:SDE_alpha_l}\\
 d \cL&=\Im[(e^{-i \alpha}-1) (d b_1+i db_2)]=2 \sin(\alpha/2) dW_2.\notag
\end{align}
Here the $W_j$ defined above are independent copies of Brownian motion that depend on $\lambda$.

For $\nu\neq-\infty$ the processes $\alpha_{\lambda,\nu}$, $\cL_{\lambda,\nu}$ are the unique strong solutions of \eqref{eq:SDE_alpha_l} with initial conditions $\alpha_{\lambda,\nu}(\nu)=\cL_{\lambda,\nu}(\nu)=0$. As $\nu\to -\infty$ these processes converge a.s.~uniformly on compacts to $\alpha_{\lambda}$, $\cL_{\lambda}$. 

For $\nu\in [-\infty,0)$, the eigenvalues of $\btau_{\beta,\nu}$ are of multiplicity one and are given by 
\begin{align}\label{eq:Lambda}
    \Lambda=\{\lambda\in \R: \alpha_{\lambda, \nu}(0)=U \mod 2\pi\},
\end{align}
where $U$ is a uniform random variable on $[0,2\pi]$, independent of $b_1, b_2$. Moreover, for $\lambda>0$ we have
\begin{align}\label{eq:counting}
|\#(\Lambda\cap [0,\lambda])-\tfrac{1}{2\pi}\,\alpha_{\lambda,\nu}(0)|\le 1.
\end{align}
\end{corollary}
\begin{proof}
Proposition \ref{prop:HSDE} and It\^o's formula implies the SDE characterization of $\alpha$ and $\cL$.

The eigenvalues of $\btau_{\beta,\nu}$ are given by the zero set of $\cA_{\nu,0}-q \cB_{\nu,0}$. By definition $\cA=q\cB$ if and only if 
$\alpha=2\Im \log(\cA-i\cB)=2\Im \log(q-i)$ mod $2\pi$. Since $q$ has Cauchy distribution, $2\Im \log (q-i)=-2\operatorname{arccot} q$ is uniform on $[-\pi,\pi]$.

Note that $\alpha=2\Im\log E(1,\lambda)$. By \eqref{eq:E_1} this is nondecreasing in $\lambda$, a standard result in oscillation theory. Since $\alpha$ equals $0$ for $\lambda=0$, equation \eqref{eq:Lambda}  implies \eqref{eq:counting}.
\end{proof}

Note that  $\alpha$ is the relative phase function introduced by Killip and Stoiciu \cite{KS}, see also \cite{BVBV_op} for the connection to the Brownian carousel \cite{BVBV}.



\subsection{Quadratic variation bounds for the process $\alpha$}

Our goal is to estimate the exponential moment of the quadratic variation of the phase function   $\alpha_{\lambda, \nu}$ introduced in  
Corollary \ref{cor:phase}. This will be needed for bounding moments of $\zeta$.
In this section $\lambda\in \R, \nu\in (-\infty,0)$ are fixed, and we often abbreviate $\alpha=\alpha_{\lambda,\nu}$.

Our first lemma controls the tails of the time when $\alpha$  could start collecting  significant quadratic variation. 

\begin{lemma}[Early behavior  of $\alpha$]\label{lem:E_mom_1}
For $\eps>0$ define the stopping time %
\begin{align}\label{def:taueps}
\tau_\eps=\tau_{\eps, \lambda, \nu}=\inf\{t\le 0: \sin^2(\alpha(t)/2)\ge e^{\eps t}\}.
\end{align}
For all $\beta>0$, $\lambda\in \R$ and $a<\tfrac{\beta(\beta+2)}{8}$ there exists $\eps, c$ so that for all $\nu$ and all $t\le 0$ we have
\begin{align*}
P(\tau_{\eps}<t)\le c e^{at}.
\end{align*}
\end{lemma}

\begin{proof}
Note that
$X=\log (\tan(\alpha_{\lambda,\nu}/4))$ satisfies
\begin{align}\label{eq:XSDE}
dX=\frac{\lambda\beta}{8}e^{\frac{\beta}{4}t}\cosh X\,dt+\frac12 \tanh X\, dt+dB, \qquad X(\nu)=-\infty,
\end{align}
and $\sin^2(\alpha/2)=\sech^2(X)\le 4e^{2X}.$ (This process is studied in detail in  \cite{BVBV2}.) It suffices to show a bound of the form
\begin{align}\label{eq:tauX}
P(\tau_X\le t)\le c e^{at}
\qquad\text{ for }\qquad \tau_X =\tau_{X,\eps}= \inf\{t\le 0 : X(t)\ge \eps t\}.
\end{align}
Moreover, it suffices to show \eqref{eq:tauX} for $t\le k$ with a fixed $k<0$, as the general case follows by changing $c$ accordingly.

Let $k<0,\eps<\theta<\tfrac{\beta}{4}$.
The drift term for $X$ in the region  $t\le k$ with $X(t)\in [\theta t,\eps t]$, satisfies
\begin{align}\label{eq:Xdrift}
\frac{\lambda\beta}{8}e^{\frac{\beta}{4}t}\cosh X +\frac12 \tanh X \le \frac{\lambda\beta}{8}e^{(\frac{\beta}{4}-\theta) k}+\frac{1}{2}(2e^{2\eps k}-1)\le -q,
\end{align}
where $q<1/2$ can be made arbitrarily close to $1/2$ for  $\theta, \lambda$ fixed and appropriate choice of $k$.

For the  Brownian motion $B$ driving $X$ in \eqref{eq:XSDE}, let $Y$ be $B-(\theta+\q)t$ reflected at $0$. This is the nonnegative process defined as
\begin{align*}
Y(t)&=\sup_{u\le t}\left(B(t)-B(u)-(q+\theta)(t-u)\right).
\end{align*}
We claim that
\begin{equation}\label{E:XY}
    X(s)\le Y(s)+\theta s \qquad \text{  for }s\le \tau_X\wedge k.
\end{equation}Clearly this holds if $X(s)\le \theta s$. Otherwise, since $X$ starts at $-\infty$,  there is a maximal time  $\sigma$  at most $s$ so that $X(\sigma)=\theta \sigma$. By \eqref{eq:Xdrift}  for $\sigma\le s\le \tau_X\wedge k$ we have
\begin{align*}
X(s)&=X(\sigma)+\int_\sigma^s \frac{\lambda\beta}{8}e^{\frac{\beta}{4}u}\cosh X_u+\frac12 \tanh X_u \,du+B(s)-B(\sigma)\\
&\le \theta \sigma-q(s-\sigma) +B(s)-B(\sigma)\le Y(s)+\theta s.
\end{align*}
The process $Y$ is a stationary Markov process and $Y(s)$ has exponential distribution with rate $2(q+\theta)$ see, for example \cite{MY_PS2}.  Let $\tau_Y$ be the first hitting time of the line $(\eps-\theta)s$ by $Y$.  By \eqref{E:XY} we have $\tau_Y\le \tau_X \wedge k$.

Let $m(s)=(\eps-\theta)s-1$. For $t\le -1$ 
we have
\begin{align}\notag
1(\tau_Y\le t)1(Y(s)>m(s) \text{ for all }s\in[\tau_Y, \tau_Y+1])
\le \int_{-\infty}^{t+1} 1(Y(s)>m(s))ds.
\end{align}
By the strong Markov property, the conditional expectation of the left hand side given $\mathcal F_{\tau_Y}$ is at least $r1(\tau_Y\le t)$ with
$$
r=r_{\theta-\eps}=P\left(B(u)>m(u) \text{ for all } u\in[0,1]\right)>0,$$
where $B$ is a standard Brownian motion. Taking  expectations we get \begin{align*}
rP( \tau_Y\le t)&\le \int_{-\infty}^{t+1} P( Y(s)>m(s))ds\\
&=  \int^{t+1}_{-\infty} e^{2(q+\theta) ((\theta-\eps) s+1)}ds= \frac{e^{2(q+\theta)}}{2(q+\theta)(\theta-\eps)}e^{2(q+\theta) (\theta-\eps)(t+1)} .
\end{align*}
By choosing $q<1/2$, $\theta<\beta/4$, and $\eps>0$ appropriately we can make the coefficient of $t$ in the exponent equal to $a<\tfrac{\beta(\beta+2)}{8}$. Since $ \tau_Y\le \tau_X\wedge k$ this gives  the desired bound \eqref{eq:tauX} on $t\le k$, and the lemma follows.
\end{proof}

The next lemma is about a generalized version of the process $\alpha$. It controls the amount of quadratic variation collected on a finite time interval.
\begin{lemma}[Quadratic variation of $\alpha$ on a finite interval]\label{lem:E_mom_2}
For $p\ge 0$ there is a constant $c_p$ so that the following holds. Consider the SDE
\begin{align}\label{eq:SDE_alpha_99}
 d \alpha&=\lambda(t)dt+2 \sin(\alpha/2) dW
\end{align}
on $[0,\sigma]$
with deterministic $0<\lambda(t)\le \eps$ and arbitrary initial condition. Then
\begin{align}
E\,e^{p \int_ 0^{\sigma} \sin^2(\alpha(s)/2) ds}\le (1+\tfrac{1}{\eps}) e^{((p-\sqrt{p/2})^++\eps c_p) \sigma}.
\end{align}
\end{lemma}

\begin{proof}[Proof of Lemma \ref{lem:E_mom_2}]
First assume $p>1/2$.
Set
$$
M_t=\exp\left(\int_{0}^t g(\alpha_s) ds-v t\right) (f(\alpha_t)+\eps),
$$
where
$$
v=p-\sqrt{p/2}+\delta, \qquad g(\alpha)=p \sin^2(\alpha/2), \qquad f(\alpha)=|\hspace{-0.15em}\sin(\alpha/2)|^{\sqrt{2p}},
$$
and $\delta=c_p \eps$. We will show that  for an appropriate choice of $c_p$ the process $M_t, t\in [0,\sigma]$ is a supermartingale. If this holds then
\[
E M_\sigma\le E M_0=f(\alpha_0)+\eps\le 1+\eps.
\]
We also have
\begin{align*}
EM_\sigma=E\exp\left(\int_ 0^{\sigma} g(\alpha_s) ds-v \sigma \right) (f(\alpha_{\sigma})+\eps)
\ge e^{-v \sigma} \eps  E\exp\left(\int_ 0^{\sigma} p \sin^2(\alpha_s/2) ds \right).
\end{align*}
This gives
\[
 E\exp\left(\int_ 0^{\sigma} p \sin^2(\alpha_s/2) ds \right)\le \tfrac{1}{\eps} e^{(p-\sqrt{p/2}+\delta)\sigma}(1+\eps),
\]
which implies the statement of the lemma.

A coupling argument,  see e.g.~Proposition 9 of \cite{BVBV}, shows that if $\tau_k$ is the first hitting time of $2\pi k$, 
then  $\alpha(t)> 2k\pi$ for $t> \tau_k$. This can be seen from the SDE \eqref{eq:SDE_alpha_99}: when $\alpha$ is a multiple of $2\pi$, then the noise term vanishes and the drift term is positive. 

The function $f$ is twice continuously differentiable apart from the set $2\pi \ZZ$. 
It\^o's formula applies to $M$ between times $\tau_k,\tau_{k+1}$ and hence over the whole interval $[0,\sigma]$. Let  $\cG=2 \sin^2(\alpha/2) \partial_{\alpha}^2$
be the generator of the process $\alpha$ from (\ref{eq:SDE_alpha_99}) with no drift. 
We get
\[
dM_t=e^{\int_{0}^t g(\alpha_s) ds-v t} (\cG f+(g(\alpha_t)-v)(f(\alpha)+\eps)+\lambda(t) f'(\alpha_t))dt+dW \text{ terms}.
\]
A nonnegative local supermartingale with constant initial condition is a supermartingale, so it suffices to show that
\begin{equation}\label{eq:op-ineq}
\cG f+(g-v)(f+\eps)+\lambda(t) f'\le 0.
\end{equation}
Since
$(\cG+g) f=(p-\sqrt{p/2})f$,
the inequality \eqref{eq:op-ineq} reduces to
\begin{align}\label{eq:supmart}
\left(\sqrt{p/2}- p \cos^{2}(\alpha/2)\right)\eps + \lambda(t) f'\le \delta f+\delta \eps.
\end{align}
We will show that for $p>1/2$ with the appropriate choice of  $c_p$ so that  the following holds:
\[
\sqrt{p/2}- p \cos^{2}(\alpha/2) + |f'|\le c_p f,
\]
this implies  \eqref{eq:supmart}. We have $|f'(\alpha)|=\sqrt{p/2} |\cos \left(\frac{\alpha}{2}\right)| \left|\sin \left(\frac{\alpha}{2}\right)\right|^{\sqrt{2p}-1}$, so it is enough to show that the function
\[
\left(\sqrt{p/2}- p \cos^{2}(\alpha/2)\right)\left|\sin(\alpha/2)\right|^{-\sqrt{2p}}+\sqrt{p/2} \left|\cos \left(\alpha/2\right)\right| \left|\sin \left(\alpha/2\right)\right|^{-1}
\]
is bounded from above. By continuity, it suffices to check the behavior of this function near $\alpha\in 2\pi \ZZ$, and since $p>1/2$ we see that the function converges to $-\infty$ there. The statement of the lemma follows for $p>1/2$.

For $p\le 1/2$, we use that for $x\ge 0$, we have $e^{px}\le 1\vee e^{(1/2+\delta)x}$ to reduce to the $p>1/2$ case. 
\end{proof}
We now combine the results in the previous lemmas and give a bound on the quadratic variation of $\alpha$.
\begin{proposition}[Quadratic variation of $\alpha$]\label{prop:E_mom}
If $0<p<\frac12 (1+\tfrac{\beta}{2})^2$  then
$$E e^{p\int_{\nu}^0 \sin^2(\alpha_{\lambda,\nu}/2) ds}<c_{p,\beta}(1+|\lambda|)^{\frac{4}{\beta}p},$$
with a constant $c_{p,\beta}$ which does not depend on $\nu$.
\end{proposition}


\begin{proof}[Proof of Proposition \ref{prop:E_mom}]

Let $\eps>0$, $k<0$, $\lambda=1$.  For  $t\le \tau=\tau_\eps$ as in \eqref{def:taueps} we have  $\sin^2(\alpha_{\lambda,\nu}/2)\le e^{\eps t}$, so
\begin{align*}
\int_{\nu}^0 \sin^2(\alpha/2) ds&\le \tfrac{1}{\eps}+\int_{\tau\wedge k}^{k} \sin^2(\alpha/2) ds + |k|.
\end{align*}
Set $\ell(t)=2 \arcsin(e^{\eps t/2})$. 
By Lemma \ref{lem:E_mom_2} for arbitrarily small $\delta>0$ if $|k|,\tfrac{1}{\eps}$ are large enough then
$$
E_{t,\ell(t)}e^{p\int_{t}^{ k} \sin^2(\alpha/2) ds}\le c e^{-((p-\sqrt{p/2})^++\delta)t} \qquad \text{ for all } t\le k.
$$
Here $E_{t_0,\ell_0}$ is the expectation for  the process $\alpha$ satisfying  (\ref{eq:SDE_alpha_l}) on $[t_0,\infty)$ with initial condition $\alpha(t_0)=\ell_0$. By the strong Markov property
\begin{align*}
E e^{p\int_{\tau \wedge k}^{k} \sin^2(\alpha_\lambda/2) ds}&\le c E e^{-((p-\sqrt{p/2})^++\delta)\tau},
\end{align*}
where $c$ does not depend on $\nu$.
 Lemma  \ref{lem:E_mom_1} provides an  upper bound 
on $E e^{-((p-\sqrt{p/2})^++\delta)\tau}$ if   $\frac{\beta(\beta+2)}{8}>p-\sqrt{p/2}$ and  $\delta$ is small enough. Hence there is a constant $c_{p,\beta}$ so that for $\lambda=1$
\begin{align}\label{eq:alpha_bnd_00}
    E e^{p\int_{\nu}^0 \sin^2(\alpha_{\lambda,\nu}/2) ds}\le c_{p,\beta}.
\end{align}

Now let $0<\lambda$. Note that by the scaling properties of the drift term of the SDE \eqref{eq:SDE_alpha_l} we have
\[
\{\alpha_{\lambda,\nu}(t), \nu\le t\le 0\}\ed \{\alpha_{1,\nu_\lambda}(t+\tfrac{4}{\beta}\log \lambda), \nu\le t\le 0\}.
\]
where $\nu_\lambda=\nu+\tfrac{4}{\beta}\log \lambda$.
This implies
\[
E e^{p\int_{\nu}^0 \sin^2(\alpha_{\lambda,v}/2) ds}=E e^{p\int_{\nu_\lambda}^{\frac{4}{\beta}\log \lambda} \sin^2(\alpha_{1,\nu_\lambda}/2) ds}.
\]
If $0<\lambda<1$ then we have
\[
E e^{p\int_{\nu_\lambda}^{\frac{4}{\beta}\log \lambda} \sin^2(\alpha_{1,\nu_\lambda}/2) ds}\le E e^{p\int_{\nu_\lambda}^{0} \sin^2(\alpha_{1,\nu_\lambda}/2) ds},
\]
and the statement follows from the $\lambda=1$ case  \eqref{eq:alpha_bnd_00}.

If $1<\lambda<e^{-\frac{\beta}{4}\nu}$ then $\nu_\lambda<0$ and we have
\[
E e^{p\int_{\nu_\lambda}^{\frac4{\beta}\log \lambda} \sin^2(\alpha_{1,\nu_\lambda}/2) ds}\le E e^{p(\frac{4}{\beta}\log \lambda+\int_{\nu_\lambda}^{0} \sin^2(\alpha_{1,\nu_\lambda}/2) ds)} =\lambda^{\frac{4}{\beta}p} E e^{p(\int_{\nu_\lambda}^{0} \sin^2(\alpha_{1,\nu_\lambda}/2) ds)}
\]
which yields the desired bound. 

Finally, if $1<e^{-\frac{\beta}{4}\nu}\le \lambda$ then  $\nu\ge -\frac{4}{\beta}\log \lambda$ and we get
\[
E e^{p\int_{\nu}^{0} \sin^2(\alpha_{1,\nu_\lambda}/2) ds}\le e^{-p \nu}\le \lambda^{\frac{4}{\beta} p}.
\]
The $\lambda<0$ case follows by symmetry.
\end{proof}






\subsection{Moment bounds for $\cE$}

We can use the quadratic variation bounds of the previous section to bound the moments of $\mathcal E$.

\begin{theorem}\label{thm:E_mom}
For $\lambda,\gamma\in \mathbb{R}$, $\nu\in [\infty,0)$ and $|\gamma|<1+\tfrac{\beta}{2}$,  we have 
\begin{align}\label{eq:Emom}
E[|\cE_{\nu,0}(\lambda)|^\gamma]&<c(1+|\lambda|)^{2\gamma^2/\beta},\\
\label{eq:alphamom}
Ee^{\tfrac{\gamma}{4} |\alpha_{\lambda,\nu}(0)-\lambda(1-e^{\beta\nu/4})|}&<c(1+|\lambda|)^{\gamma^2/\beta}.
\end{align}
Here $c$ depends on $\beta, \gamma$, but not on $\nu$.
For all $\gamma\in \mathbb R$ with $c$ depending on $\beta$ only we have 
\begin{align}
  E\,e^{\gamma \alpha_{\lambda, \nu}(0)}\le    
2e^{c |\gamma| \left(1+|\lambda|+(|\gamma|+\log |\lambda|)^+\right)}. \label{eq:Emom_bnd}
\end{align}
\end{theorem}

%
%
%
%

\begin{proof} 
First let $\nu\neq-\infty$. Recall $2 \log \cE=\cL+i \alpha$, where $\mathcal L, \alpha$ satisfy \eqref{eq:SDE_alpha_l}. 
Thus we have 
$$
\tfrac12 \cL_{\lambda, \nu,0}=\int_{\nu}^0 2\sin(\alpha_{\lambda,\nu}/2)dW_2\ed
\mathcal{N}\cdot \sqrt{\int_{\nu}^0 \sin^2(\alpha_{\lambda,\nu}/2) ds},
$$ 
since $W_2$ is independent of $\alpha$. Here 
$\mathcal{N}$ is a standard normal random variable  independent of $\alpha$. This shows that
$$
E |\cE_{\nu,0}(\lambda)|^\gamma=E e^{\gamma \mathcal{N}\cdot \sqrt{\int_{\nu}^0 \sin^2(\alpha_{\lambda\nu}/2) ds}}
=E e^{\frac{\gamma^2}{2} \int_{\nu}^0 \sin^2(\alpha_{\lambda,\nu}/2) ds}.
$$
The bound \eqref{eq:Emom} follows from Proposition \ref{prop:E_mom}.  

Without loss of generality, we assume $\lambda\ge 0$ for the rest of the proof. 
\begin{align*}
M(t)=\exp\left(\tfrac{\gamma}{2} \left(\alpha_{\lambda,\nu}(t)-\lambda e^{\frac{\beta}{4} t}+\lambda e^{\frac{\beta}{4}\nu}\right)-\tfrac{\gamma^2}{2} \int_{\nu}^t  \sin^2(\alpha_{\lambda,\nu}/2)ds\right), \qquad t\ge \nu,
\end{align*}
is a martingale and $EM(0)=EM(\nu)=1$. We have
\begin{align}\notag
 E\exp\left(\tfrac{\gamma}{4}( \alpha_{\lambda,\nu}(0)-\lambda+\lambda e^{\frac{\beta}{4}\nu})\right)=
 E\left[\sqrt{M(0)} \exp\left(\tfrac{\gamma^2}{4} \int_{\nu}^{0}  \sin^2(\alpha_{\lambda,\nu}(s)/2)ds\right) \right]
\end{align}
Since $EM(0)=1$, Cauchy-Schwarz gives  the upper bound
\begin{align}\notag \left(E\exp\left(\frac{\gamma^2}{2} \int_{\nu}^{0} \sin^2(\alpha_{\lambda,\nu}(s)/2)\right)\right)^{1/2}.
\end{align}
This inequality applied with $\pm\gamma$ together with Proposition \ref{prop:E_mom} yields \eqref{eq:alphamom}.


To prove \eqref{eq:Emom_bnd} we may assume $
\gamma>0$. 
A coupling argument, Proposition 9 of \cite{BVBV} shows that for $\lambda>0$ we have $\alpha_{\lambda,\nu}(t)>0$ for $t>\nu$.
Markov's inequality implies 
\begin{align}\label{eq:Markovbnd1}
    P(\alpha_{\lambda,\nu}(t) \ge 2\pi )\le \frac{\lambda}{2\pi} (e^{\frac{\beta}{4}t}-e^{\frac{\beta}{4} \nu})\le \lambda e^{\frac{\beta}{4}t}.
\end{align}
 Let $k$ be a positive integer, and let $\tau$ be the hitting time of $2\pi k$ for $\alpha_{\lambda, \nu}$.
The strong Markov property implies that given $\tau=\nu_1$ the conditional distribution of $\alpha_{\lambda,\nu}(t-\tau)-2k\pi, t\ge \nu_1 $ is the same as that of $\alpha_{\lambda,\nu_1 }(t), t\ge \nu_1 $. Hence from \eqref{eq:Markovbnd1} we get  
\begin{align*}
P\left(\alpha_{\lambda,\nu}(t)\ge 2\pi (k+1) \,\big|\,\tau<t\right)\le \lambda e^{\frac{\beta}{4}t}.
\end{align*}
It follows that 
\begin{align}
P(\alpha_{\lambda, \nu}(t)>2\pi  k)\le   \big(\lambda e^{\frac{\beta}{4}t}\big)^k. \label{eq:exp_tail_bnd}
\end{align}
By Cauchy-Schwarz for any $\nu<t_0<0$ we have
\begin{align}\label{eq:CS_mom}
E e^{\gamma \alpha_{\lambda,\nu}(0)}\le (E e^{2\gamma \alpha_{\lambda,\nu}(t_0)})^{1/2} (Ee^{2\gamma (\alpha_{\lambda,\nu}(0)-\alpha_{\lambda,\nu}(t_0))})^{1/2}.
\end{align}
From \eqref{eq:exp_tail_bnd} we get
\begin{align}\notag
Ee^{2\gamma \alpha_{\lambda,\nu}(t_0)}&= \sum_{k=0}^\infty E[e^{2\gamma \alpha_{\lambda,\nu}(t_0)}\, \ind(2k\pi\le \alpha_{\lambda,\nu}(t_0) <2(k+1)\pi)]\\
&\le \sum_{k=0}^\infty e^{4 \gamma (k+1) \pi} P(\alpha_{\lambda,\nu}(t_0)\ge 2k \pi)\le \sum_{k=0}^\infty e^{4 \gamma (k+1) \pi} \lambda^k e^{\frac{\beta}{4} k t_0}.\notag
\end{align}
If $\tfrac{\beta}{4} t_0\le -4\gamma \pi-\log \lambda-1$,
then we get the bound
\begin{align}
 E\,e^{2\gamma \alpha_{\lambda,\nu}(t_0)}  \le 2 e^{4 \gamma \pi}.\label{eq:mom_1st_term}
\end{align}
We now use a variant of the martingale in the first part. For $\nu\le t_0\le 0$ the process 
\begin{align*}
M(t)=\exp\left(4\gamma \left(\alpha_{\lambda,\nu}(t)-\alpha_{\lambda,\nu}(t_0)-\lambda( e^{\frac{\beta}{4} t}-e^{\frac{\beta}{4}t_0})\right)-8\gamma\int_{t_0}^t  \sin^2(\alpha_{\lambda,\nu}/2)ds\right), \qquad t\ge t_0.
\end{align*}
is a martingale with $EM(0)=EM(t_0)=1$. By Cauchy-Schwarz we have 
\begin{align}\notag
     Ee^{2\gamma \left(\alpha_{\lambda,\nu}(0)-\alpha_{\lambda,\nu}(t_0)-\lambda(1-e^{\frac{\beta}{4}t_0})\right)}&=E[\sqrt{M(0)} e^{4\gamma \int_{t_0}^{0}  \sin^2(\alpha_{\lambda,\nu}(s)/2)ds} ]\\ 
&\le (Ee^{8\gamma \int_{t_0}^{0} \sin^2(\alpha_{\lambda,\nu}(s)/2)ds} )^{1/2} 
\le e^{-4 \gamma t_0}, \label{eq:mart_bnd_2}
\end{align}
where the last step uses the upper bound 1 on $\sin(\cdot)^2$.
Now set 
\[
t_0=\begin{cases}
\nu, \qquad &\text{if }\quad -4\gamma \pi-\log \lambda-1<\frac{\beta}{4}\nu,\\
-\tfrac4{\beta}(4\gamma \pi+\log \lambda+1), \qquad &\text{if }\quad
\frac{\beta}{4}\nu \le -4\gamma \pi-\log \lambda-1<0,\\
0, \qquad &\text{if }\quad 0\le -4\gamma \pi-\log \lambda-1.
\end{cases}
\]
The estimates \eqref{eq:CS_mom}, \eqref{eq:mom_1st_term} and \eqref{eq:mart_bnd_2} yield
$  E\,e^{\gamma \alpha_{\lambda,\nu}(0)}\le \left(2 e^{4\gamma \pi}\right)^{1/2} e^{-2\gamma t_0+\gamma \lambda}$, and \eqref{eq:Emom_bnd} follows.

The $\nu=\infty$ case of all three bounds follows from Fatou's lemma. 
\end{proof}

\subsection{Precise growth of the number of zeros}

The following corollary provides regularity bounds for the $\Sineb$ process and its approximations. It is closely related to the work of  Holcomb and Paquette \cite{holcomb2018maximum}, where the optimal constant in front of the $\log$ is also determined. 
Bounds of this type that are less precise can be derived from the variance bounds on the counting function in  \cite{KVV}. Such a bound is explicitly given in  \cite{Najnudel-Virag}.

\begin{corollary}[Regularity of the $\Sineb$ process]\label{c:sineb-regularity}
For $\nu \in [-\infty,0)$ let $N_\nu(\lambda)$ be the counting function of the eigenvalue process of the operator $\btau_{\beta, \nu}$. 
Then almost surely, for all large enough $|\lambda|$ we have
\[
|N_\nu(\lambda)-\tfrac{\lambda}{2\pi}(1-e^{\frac{\beta}{4} \nu})|<(1+\beta^{-1})\log |\lambda|.\]
\end{corollary}
\begin{proof}
We show the statement for $\lambda\to \infty$, the $\lambda\to -\infty$ case follows by symmetry. 
By \eqref{eq:counting} we have $|N_\nu(\lambda)-\frac{1}{2\pi}\alpha_{\lambda,\nu}(0)|\le 1$. Using \eqref{eq:alphamom} with $\gamma=1$  and $\lambda\ge 1$ integer we get
\[
P(|\alpha_{\lambda,\nu}(0)-\lambda(1-e^{\frac{\beta}{4}\nu})|\ge 6(1+\beta^{-1}) \log \lambda)\le \frac{E e^{\frac14 |\alpha_{\lambda,\nu}(0)-\lambda(1-e^{\frac{\beta}{4}\nu})|}}{e^{\frac32(1+\beta^{-1}) \log \lambda}}\le \frac{c (1+\lambda)^{1/\beta}} {\lambda^{3/2(1+\beta^{-1})}}
\]
by the exponential Markov's inequality. By the Borel-Cantelli lemma we get that 
\[
\frac{1}{2\pi}|\alpha_{\lambda,\nu}(0)-\lambda(1-e^{\frac{\beta}{4}\nu})|\le \frac{6}{2\pi}(1+\beta^{-1}) \log \lambda
\]
if $\lambda$ is a large enough integer. Since $N_\nu(\lambda)$ is non-decreasing, the claim follows. 
\end{proof}

\section{Analytic properties of $\zeta_{\beta}$}

In this section we show that $\zeta_\beta$ is of Cartwright class. This allows us to show that logarithmic derivative has exponential moments off the real line. 

\subsection{Cartwright class,  exponential type and uniqueness}\label{subs:C_class}


Theorem \ref{thm:E_mom} directly implies that the integral condition \eqref{eq:log_int} needed for Cartwright class is satisfied by $\zeta_\beta$.
\begin{proposition}\label{prop:zeta_C} With  $\zeta=\zeta_{\beta,\nu}$ we have
\begin{equation}\label{eq:elogz}
\ev \log^+|\zeta(x)|\le c_\beta + \tfrac{2}{\beta}\log^+|x|
\end{equation}
for all real $x$, and so 
$$ \ev \int \frac{\log^+|\zeta(x)|}{1+x^2}dx<\infty.$$
\end{proposition}
\begin{proof}
On the real line, we have $2\zeta=(E+\bar E)-qi(E-\bar E)$. Since $\log^+|a+b|\le \log^+|a|+\log^+|b|+1$, and $\log^+|a|\le\log(1+ |a|)$, we get the bound
$$
\log^+|\zeta| \le 3+4\log^+|E|+\log^+q \le  3+4\log(1+ |E|)+\log^+q.
$$
By Theorem \ref{thm:E_mom}, $\ev |E(x)|\le c(1+|x|)^{2/\beta}$.  Since  $\log^+$ of a Cauchy distribution has finite expectation, and $\log$  is concave, Jensen's inequality gives \eqref{eq:elogz}.
\end{proof}

This shows the integral condition for Cartwright class.  We have two proof that $\zeta$ is of finite exponential type. First, it follows from the bounds \eqref{eq:Hbound} in Corollary \eqref{cor:H_char}. It also follows from the regularity of zeros, Corollary \ref{c:sineb-regularity} and Proposition \ref{prop:inf_prod_1}. Hence $\zeta$ is of Cartwright class. 

Next, by Proposition \ref{prop:inf_prod} it satisfies
\begin{align}\label{eq:zeta_prod}
\zeta(z)=\lim\limits_{r\to \infty} \prod_{|\lambda_k|<r}(1-z/\lambda_k)
\end{align}
In particular, $\zeta_2$ agrees with the random analytic function constructed in \cite{CNN}. 

By Proposition \ref{prop:inf_prod} we also have
$$
    \limsup_{x\to\infty}
    \frac{\log|\zeta(ix)|}{x}=
     \limsup_{|z|\to\infty}\frac{\log|\zeta(z)|}{|z|}=1/2,
$$
so $\zeta_\beta$ has exponential type 1/2.

Finally, by Proposition \ref{prop:Cart-unique} we have the following.
\begin{proposition} \label{prop:zetaunique} The law of $\zeta_\beta$ is the unique distribution on the Cartwright class, satisfying $\zeta_\beta(\mathbb R)\subset \mathbb R$,  $\zeta_\beta(0)=1$, with zero distribution $\Sineb$ of multiplicity one.
\end{proposition}

\subsection{Product identities and Cauchy trace}

The simple identity 
$$
2\sin(x)=- \prod_{j=1}^{k} 2\sin((j\pi+x)/k)
$$
has the following analogue for  $\zeta_\beta$.
\begin{corollary}\label{c:forrester}
Let $k\ge 2$ an integer. There exists a coupling of $k$ copies $\zeta_{2k,1},\ldots, \zeta_{2k,k}$ of $\zeta_{2k}$
so that
$$
\zeta_{2/k}(z)=\prod_{j=1}^k\zeta_{2k,j}(z/k).
$$
\end{corollary}
\begin{proof}
Let $\lambda_\ell$ be the ordered points of the Sine$_{2/k}$ process. Let $K$ be a uniform random element of $\{1,\ldots,k\}$.
Let $\Lambda_j=\{\lambda_\ell:\ell=K+j \mod k\}$. By a result of Forrester \cite{ForresterDec}, each  $\Lambda_j$ has the same distribution as $k$ times a Sine$_{2k}$ process. The product formula \eqref{eq:zeta_prod} applied to each $\Lambda_j$ gives the analytic functions $\zeta_{2k,j}(z/k)$. Their product is a Cartwright function with zeros $\lambda_j$. By Proposition \ref{prop:zetaunique} it has the same law as $\zeta_{2/k}$.
\end{proof}
By taking limits of Theorem 2.5.17 in \cite{AGZ}, we see that the distribution of 
every second point in the superposition of two independent $\Sine_1$ processes is a scaled $\Sine_2$ process. It follows that given two independent copies of $\zeta_1$, there exists two dependent copies of $\zeta_2$ so that
$
\zeta_{1,1}\zeta_{1,2}=\zeta_{2,1}\zeta_{2,2}.
$

The next theorem identifies the distribution of the integral trace $\mathfrak{t}_\beta$ for $\btau_\beta$.
\begin{theorem}\label{t:AW} Consider the $\Sineb$ process, and denote the points in the process with $\lambda_k, k\in \ZZ$. Then the  principal value sum $$\lim_{r\to \infty} \sum_{|\lambda_k|\le r} \lambda_k^{-1}$$ has  Cauchy distribution with density $1/(2\pi (x^2+1/4))$.
\end{theorem}
\begin{proof}
This follows from Theorem 5.5 and the following Remark 2 of Aizenman and Warzel \cite{AW_2014}. By \eqref{eq:partial_frac}  the function $\zeta'/\zeta$ is of  Herglotz-Pick class: it maps the upper half-plane into itself. Moreover, $\zeta'/\zeta$ is almost surely continuous at $0$. Since the intensity of the $\Sine_\beta$ process is $1/(2\pi)$, and $\Sineb$ has the same distribution as $-\Sineb$, the parameter of the Cauchy distribution is $1/2$.
\end{proof}

\subsection{Exponential moments of  $(\log  \zeta_\beta)'$}

Theorem \ref{thm:E_mom} and Proposition \ref{prop:zeta_C} imply uniform exponential moment bounds on the log derivative of $\zeta_{\beta,\nu}$ away from the real line, this is the content of the following proposition. Chhaibi, Najnudel and Nikhegbali \cite{CNN} show such bounds in the $\beta=2$ case, and the proof below is inspired by their argument. 

\begin{proposition}\label{prop:expmom}
Let $K$ be a compact set in the upper or lower open half plane, and let $\zeta=\zeta_{\beta,\nu}$. Then for $\nu\in [-\infty,0)$ with  $c$ depending on $p, \beta$ and $K$, but not on $\nu$ we have
\begin{align}\label{eq:expmom}
    E\sup_{z\in K} \exp \Big(p \left|\tfrac{\zeta'(z)}{\zeta(z)}\right|\Big)\le c.
\end{align}
\end{proposition}

\begin{proof}
$\zeta_{\beta,\nu}$ is of Cartwright class by Proposition \ref{prop:zeta_C}. Proposition \ref{prop:logder} and Corollary \eqref{c:sineb-regularity} implies 
$$ \frac{\zeta'(z)}{\zeta(z)}=\int_{-\infty}^\infty \frac{1}{(z-\lambda)^2} (\tfrac{\lambda}{2\pi}(1-e^{\frac{\beta}{4}\nu}) - N(\lambda)) d\lambda\mp i(1-e^{\nu\beta/4})/2.
$$
By \eqref{eq:counting},  $| N(\lambda)- \tfrac1{2\pi}\alpha_{\lambda,\nu}(0)|\le 1$. With $f(\lambda)=|\alpha_{\lambda,\nu}(0)-\lambda(1-e^{\frac{\beta}{4}\nu})|$ we have the bound
\begin{align}\label{eq:logder_bnd}
\sup_{z\in K} \left|\frac{\zeta'(z)}{\zeta(z)}\right|\le \int_{-\infty}^\infty \frac{f(\lambda)}{d(\lambda,K)^2}d\lambda + \frac{\pi}{d(\mathbb R,K)}+1/2,
\end{align}
where $d$ is Euclidean distance.

Since  $d(\lambda,K)^{-2}$ decays like $\lambda^{-2}$ as $\lambda\to \pm \infty$,  there exists a positive bounded  function $g_\lambda$ so that
$q_\lambda=pg_\lambda^{-1}d(K,\lambda)^{-2}$ is a probability density, and $0<\eps<g_\lambda<\sqrt{\beta/32}$ outside an  interval $L=[-\ell,\ell]$.  Here $\eps, \ell$ only depend on $p$ and $K$.
By Jensen's inequality,
\begin{equation}\notag 
\exp\int_{-\infty}^\infty p\frac{f(\lambda)}{d(\lambda,K)^2}d\lambda \le
\int_{-\infty}^{\infty}  e^{f(\lambda)g_\lambda} q_\lambda d\lambda.
\end{equation}
By Theorem \ref{thm:E_mom}  applied with $\gamma=\sqrt{\beta/2}<1+\beta/2$ on $L^c$ we have
\begin{align}
 E  e^{f(\lambda)g_\lambda} < c\frac{(1+|\lambda|)^{1/2}}{1+\lambda^2},\qquad  E \int_{L^c}  e^{f(\lambda)g_\lambda} q_\lambda d\lambda\le  \frac{c}{1+\ell^{1/2}}.   \label{eq:expmomlogder_1}
\end{align}
Let $c'$ be the supremum of $\gamma_\lambda$. Then by \eqref{eq:Emom_bnd} we get
\begin{align}
E \int_L  e^{f(\lambda)g_\lambda} q_\lambda d\lambda\le e^{c_1 (1+\ell +\log(1+\ell))} \label{eq:expmomlogder_2}
\end{align}
with  $c_1$ depending only on $c', \beta$, but not $\nu$.
The claim  follows from the bounds \eqref{eq:logder_bnd}, \eqref{eq:expmomlogder_1}, and  \eqref{eq:expmomlogder_2}.
\end{proof}


\subsection{Moments of ratios of $\zeta_\beta$ are finite}

The exponential moment bounds for the log derivative of $\zeta_\beta$ in  Proposition \ref{prop:expmom} can be used to show that  ratios have finite moments.

\begin{proposition}\label{prop:moment_ratio}
Let $w_j \in \CC\setminus \R$,  let $K\subset \mathbb C^k$ be compact and  $\zeta=\zeta_{\beta,\nu}$. Then 
\begin{align*}
\sup_{z \in K,\nu}  E\prod_{j=1}^k\left| \frac{\zeta(z_j)}{\zeta(w_j)}\right|<\infty.
\end{align*}
\end{proposition}

\begin{proof}
By H\"older's inequality and the reflection symmetry  $\zeta(\bar z)=\overline{\zeta(z)}$, it suffices to show that
$\sup_{z\in K,\nu}E(\left| \zeta(z)/\zeta(w)\right|^k)<\infty$ 
for $\Im w >0$ and $K$ in the closed upper half plane.  Since for real $x$, $|\zeta(x+iy)|$ is increasing in $y\ge 0$ by \eqref{eq:triangle}, we may further assume that $K$ is in the open upper half plane. We write
\[
\frac{\zeta(z)^k}{\zeta(w)^k}= \exp\Big(k \int_{\gamma} \frac{\zeta'(v)}{\zeta(v)} dv\Big),
\]
where $\gamma$ is the oriented line segment connecting $w$ and $z$. The claim  follows from Proposition \ref{prop:expmom} and Jensen's inequality. 
\end{proof}

\begin{proposition} 
\label{prop:ratio_analytic}
Fix $w_1,\dots, w_n\in \CC\setminus\R$, and $z_2, \dots, z_n\in \CC$. Then 
\[
r(z_1)= E\prod_{j=1}^n\frac{\zeta(z_j)}{\zeta(w_j)}
\]
is an entire function. 
\end{proposition}
\begin{proof}
Proposition \ref{prop:moment_ratio} implies that for any closed contour $\gamma$ of finite length
\[
\int_\gamma E \left|\prod_{j=1}^n\frac{\zeta(z_j)}{\zeta(w_j)}\right| |dz_1|<\infty. 
\]
Fubini's theorem gives  
\[
\int_\gamma r(z_1) dz_1= E\int_\gamma \prod_{j=1}^n\frac{\zeta(z_j)}{\zeta(w_j)} dz_1=0,
\]
since $\zeta$ is an entire function.  Morera's theorem completes the proof. 
\end{proof}

\section{Moment formulas } \label{s:moments}

Using the SDE \eqref{eq:cE} and the bounds in Section \ref{s:moments_tails} we can set up equations for the expectations of certain functionals of $\cE$ and $\zeta$.

\subsection{The function $\hat \zeta$}

The function $\zeta=\zeta_\beta=A-qB$ does not have a first moment, since  $q$ has Cauchy distribution and is independent of $A$ and $B$. The following simple variant has some moments. 
Let 
\begin{equation}\label{e:hatzeta}
\hat \zeta=\frac{\zeta}{\sqrt{1+q^2}}=\frac{U\cE + \bar U \cE^*}2, \qquad U=\frac{1-iq}{ \sqrt{1+q^2}}.
\end{equation}
The first formula follows from \eqref{eq:zeta_cE}.   The random variable $U$ is uniformly distributed on the unit complex semicircle around $0$ with positive real part, $\bar U=1/U$, and $U,\cE$ are independent. When the $k+\epsilon$-th absolute moment of $\cE, \cE^*$ exist, we have
$$
E\hat\zeta^{k}=\frac{1}{2^k}\sum_{j=0}^{k}\binom{k}{j} E[U^{2j-k}]E[\cE^j(\cE^*)^{k-j}],  \qquad E[U^{m}]=\frac{\sin(\pi m/2)}{\pi m/2}.
$$
For real $z$ and $k<1+\beta/2$ these moments exist by Theorem \ref{thm:E_mom}. Mixed moments can be expressed similarly. 
Separating  odd and even cases, we get
\begin{equation}\label{e:hatzeta-moments}
\begin{aligned}
E\hat\zeta^{2\ell}&=\binom{2\ell}{\ell}\frac{1}{4^\ell} E[\cE^\ell(\cE^*)^\ell],\\ E\hat\zeta^{2\ell+1}&=\frac{1}{2^{2\ell+1}\pi }\sum_{j=0}^{2\ell+1}\binom{2\ell+1}{j} \frac{(-1)^{1+j-\ell}}{j-\ell-1/2}E[\cE^j(\cE^*)^{2\ell+1-j}].
\end{aligned}
\end{equation}
These formulas motivate the study of the moments of the structure function $\cE$, which are also interesting on their own right. 

\subsection{Joint moments of the structure function}

Our methods give differential equations for the moments of   products of the structure function $\mathcal E$. Fix $z_1\ldots,z_n\in \mathbb C$, let $\eta\in \{-1,1\}^k$ be an index set, and let 
$$
\mathcal E_\eta=\prod_{j=1\atop \eta_j=1}^k \mathcal E(z_j) \prod_{j=1\atop \eta_j=-1}^k \mathcal E^*(z_j), \qquad z \cdot \eta=\sum_{j=1}^kz_j\eta_j ,
$$
with $\cE_{\eta,\nu}$ defined similarly using $\cE_\nu, \cE_\nu^*$. In this section we will omit the time parameter in the notation. 

For convenience of notation let $\sigma_j$ be the operator that multiplies the $j$th coordinate of $\eta$ by $-1$, and let $\sigma_j \mathcal E_\eta =\mathcal E _{\sigma_j \eta}$. For a fixed time $u$, $(\mathcal E_{\eta}, \eta\in \{-1,1\}^k)$ is a vector in $\mathbb C^{2^k}$ indexed by $\eta$. Then $\sigma_j$ is a permutation matrix  acting on the vector space  $\mathbb C^{2^k}$. 

Consider the Brownian motion $b_1,b_2$ driving the SDE \eqref{eq:cE}. Let
$b^{\eta_j}=\eta_j i b_1-b_2$.
For $k=1$ the equations \eqref{eq:cE} can be written as
$$
d\cE_\eta= \frac{i z \eta_1}{2} f_\beta  \cE_\eta  du+ \frac{\sigma_1-1}{2} \cE_\eta db^{\eta_1}.
$$
The general $k\ge 1$ case is described by the following proposition.
\begin{proposition}\label{prop:cE_eta}
Fix $k\ge 1$ and consider the operator  
$$
\Xi= \sum_{j=1\atop \eta_j=1}^k(\sigma_j-1)\sum_{\ell=1\atop \eta_\ell=-1}^k(\sigma_\ell-1),
$$
This is a $2^k\times 2^k$ matrix acting on $\mathbb C^{2^k}$. 
Then we have 
\begin{align}
 d\mathcal E_{\eta,\nu}=\frac12 
\left(i z\cdot \eta  f_\beta\, du + \Xi dt+\sum_{j=1}^k  (\sigma_j-1)db^{\eta_j}\right)\mathcal E_{\eta,\nu}.  \label{eq:cE_eta}
\end{align}
For  $\nu\in (-\infty,0)$ the processes $\cE_{\eta,\nu}, \eta\in\{-1,1\}^k$ provide the unique strong solution of the system \eqref{eq:cE_eta} on $[\nu,\infty)$ with initial condition $1$ at time $\nu$. The processes $\cE_\eta$ can be obtained as the a.s.~uniform on compact limits of $\cE_{\eta,\nu}$ as $\nu\to -\infty$.
\end{proposition}
\begin{proof}
The proof follows from It\^o's formula applied to the product $\cE_\eta$. For a fixed $\nu\in (-\infty,0)$ the system \eqref{eq:cE_eta} is a linear system, so it has a unique strong solution with a given initial condition. Corollary \ref{cor:cE_SDE} implies the last statement. 
\end{proof}


For any $\nu\neq\infty$,  the coefficients of the vector valued SDE \eqref{eq:cE_eta} satisfy the Lipschitz conditions required by the standard existence and uniqueness theorem for SDEs, Theorem 5.2.1 in \cite{OksSDE}. This implies that $u\to \int_{\nu}^u \sum_{j=1}^k  (\sigma_j-1)\mathcal E_{\eta,\nu} db^{\eta_j}$ is an $L^2$ martingale. Hence for each $u\in [\nu,\infty)$  the expected value
\[
r_{\eta,\nu}(u):=E \cE_{\eta,\nu}(u)
\]
is finite, and it satisfies the differential equation system
\begin{align}
    r_{\eta,\nu}'=\tfrac{1}{2}\left(iz \cdot \eta f_\beta+\Xi\right)r_{\eta,\nu} \label{eq:r_eta}
\end{align}
with initial condition $r_{\eta,\nu}(\nu)=1$. Here we extend the definition  
$\Xi r_\eta=r_{\Xi \eta}$. Since the time dependent coefficients in the linear system \eqref{eq:r_eta} are bounded on compacts, there is a unique solution  $r_{\eta,\nu}(z)$, which is analytic in $z$.


For $\eta=\pm(1,\ldots,1)$, one of the sums in the expression for $\Xi$ vanishes, so $\Xi=0$. We get 
$
r'=\frac12 i\,z\cdot \eta f_\beta r
$,
and so 
\begin{align}\label{eq:r_all_one}
r_{\eta,\nu}=\exp\Big(\frac{i z\cdot \eta}{2} \left(e^{\frac{\beta  u}{4}}-e^{\frac{\beta  \nu }{4}}\right)\Big),
\qquad \eta=\pm(1,\ldots,1).
\end{align}
In particular, in this case we have $r_{\eta,\nu}(0)=\exp\Big(\frac{i z\cdot \eta}{2} \left(1-e^{\frac{\beta  \nu }{4}}\right)\Big)$.

\begin{proposition}\label{prop:E-moments}
Assume that $1+\beta/2>k$, and let $\lambda_1, \dots, \lambda_k\in \R $. Then at time $0$ we have 
\[
E\prod_{j=1}^k \cE(\lambda_j)=\prod_{j=1}^k e^{i\lambda_j/2}.
\]
\end{proposition}
\begin{proof}
With $\eta=(1,\dots,1)$ we have $r_{\eta,\nu}(0)\to \prod_{j=1}^n e^{i\lambda_j/2}$ as $\nu\to -\infty$.

By Theorem \ref{thm:E_mom} there is a constant $c,\eps>0$  so that at time $0$ we have \[
E[|\cE_{\nu}(\lambda)|^{k+\eps}]<c (1+|\lambda|)^{2(k+\eps)^2/\beta}, \qquad \lambda\in \R.
\]
H\"older's inequality implies that $\prod_{j=1}^k \cE_{\nu}(\lambda_j)$ are uniformly integrable as $\nu\to -\infty$. Since $\prod_{j=1}^k \cE_{\nu}(\lambda_j)$ converges to $\prod_{j=1}^k \cE(\lambda_j)$, the statement follows. 
\end{proof}
With \eqref{e:hatzeta-moments} we get 
\begin{equation}\label{e:hatzeta-first}
  E \hat\zeta(\lambda )=\tfrac{2}{\pi}\cos(\lambda/2), \qquad \lambda \in \mathbb R. 
\end{equation}
Now consider the case $n=2$, $z_1=z_2=\lambda\in \R$. 
From \eqref{eq:r_all_one} we get
\[
r_{(1,1),\nu}=r_{-(1,1),\nu}^{-1}=\exp\Big(i \lambda \left(e^{\frac{\beta  u}{4}}-e^{\frac{\beta  \nu }{4}}\right)\Big).
\]
Introduce 
\[
r_*=r_{(1,-1),\nu}=r_{(-1,1),\nu}=E |\cE_{\nu,u}(\lambda)|^2,
\]
by \eqref{eq:r_eta} we have
\begin{align*}
r_*'&=\frac12 (\sigma_1-1)(\sigma_2-1) r_{\theta,\nu}=-\cos\left(\lambda \left(e^{\frac{\beta  u}{4}}-e^{\frac{\beta  \nu }{4}}\right)\right)+r_*.
\end{align*}
This ODE can be solved to give
\begin{align}\label{eq:2nd_mom}
    r_*(u)=1+e^u \int_{\nu}^u e^{-s} \left(1-\cos\left(\lambda \left(e^{\frac{\beta  s}{4}}-e^{\frac{\beta  \nu }{4}}\right)\right)\right)ds.
\end{align}
If $\beta>2$ then as $\nu\to -\infty$ this function converges to
\[
1+e^u \int_{-\infty}^u e^{-s} \left(1-\cos\left(\lambda  e^{\frac{\beta  s}{4}}\right)\right)ds.
\]
By Theorem 58 if $\beta>2$ then at time 0, $E|\cE_{\nu}(\lambda)|^{2+\eps}<c (1+|\lambda|^2)^{2(2+\eps)^2/\beta}$ with an absolute constant $c$ and $\eps>0$. By uniform integrability, \[
E[|\cE(\lambda)|^{2}]=1+ \int_{-\infty}^0 e^{-s} \left(1-\cos\left(\lambda  e^{\frac{\beta  s}{4}}\right)\right)ds, \qquad \lambda\in \R, \beta>2.
\]
With a time-change and  \eqref{e:hatzeta-moments} we get
\begin{align}\label{e:hatzeta-second}
2E[\hat \zeta(\lambda)^2]= E[|\cE(\lambda)|^2]&=1+\frac{4}{\beta}\int_0^1 t^{-4/\beta-1}(1-\cos(\lambda t))dt.
\end{align}
This can be written as a  generalized hypergeometric function
\begin{align}\label{eq:2ndmom}
 _1F_2\left(-\frac{2}{\beta };\frac{1}{2},1-\frac{2}{\beta };-\lambda ^2/4\right)
    &=\sum_{k=0}^\infty \frac{(-2/\beta)_k}{(1/2)_k (1-2/\beta)_k} \frac{(-\lambda^2/4)^k } {k! }=\sum_{k=0}^\infty \frac{ (-1)^k \lambda^{2k} }{(2 k)! (1- \tfrac{\beta}{2}  k)},
\end{align}
with $(x)_k=x(x+1)\cdots (x+k-1)$.

In the $\beta=2$ case the integral \eqref{eq:2nd_mom} with $u=0$ explodes as $-\lambda^2\nu/2$ as $\nu\to -\infty$. More precisely, we have
\begin{align*}
    &\lim_{\nu\to-\infty}  \big(E[|\cE_{\nu}(\lambda)|^2]+\lambda^2 \nu/2\big)=1-\frac32 \lambda^2+\sum_{k=2}^\infty \frac{(-1)^{k+1}}{(k-1) (2 k)!} \lambda^{2k}.
\end{align*}
With the cosine integral  $\operatorname{Ci}x=-\int_{x}^\infty (\cos t)/t\, ds$ and the Euler constant $\gamma$ we can write this as
\begin{align*}
\lambda ^2 \operatorname{Ci}|\lambda| -\gamma  \lambda ^2-\lambda ^2 \log |\lambda| -\lambda  \sin \lambda +\cos \lambda .
\end{align*}

Now consider the more general case $n=2$, $z_1=\lambda_1$, $z_2=\lambda_2$ with $\lambda_j\in \R$. Using the previous notation we have
\[
r_*=E\left[\cE_{\nu}(\lambda_1) \cE_{\nu}^*(\lambda_2)\right]
\]
By \eqref{eq:r_eta} and \eqref{eq:r_all_one} we get that  $r_*$ satisfies the ODE
\[
r_*'=\frac12 i(\lambda_1-\lambda_2)f_\beta r_*+\frac{1}{2}(r_*+\bar r_*)-\cos\left(\tfrac{\lambda_1+\lambda_2}{2}  \left(e^{\frac{\beta  u}{4}}-e^{\frac{\beta  \nu }{4}}\right)  \right).
\]
Writing $r_*=1+e^t a+i b$ with $a,b\in \R$ this leads to
\begin{align*}
  a'&=-\tfrac{\lambda_1-\lambda_2}{2} e^{-t} f_\beta \, b+e^{-t}\left(1-\cos\left(\tfrac{\lambda_1+\lambda_2}{2}  \left(e^{\frac{\beta  u}{4}}-e^{\frac{\beta  \nu }{4}}\right)  \right)\right),\\
    b'&=e^t \tfrac{\lambda_1-\lambda_2}{2} f_\beta\,  a+\tfrac{\lambda_1-\lambda_2}{2} f_\beta   ,
\end{align*}
with $a(\nu)=b(\nu)=0$. This ODE can be directly solved using the integrating factor method to get a somewhat complicated integral expression in terms of Bessel functions. For $\beta>2$ letting $\nu\to -\infty$ we get the two point function $E\left[\cE(\lambda_1) \cE^*(\lambda_2)\right]$.



\subsection{Moments of ratios of the structure function}

We consider expectations of products of functions $\cE, \cE^{-1}$. Let
\begin{align}
\cG_t=\frac{\cE^*_t}{\cE_t},
\end{align}
and for $z_1,\ldots, z_k$, $\eta\in \{-1,1\}^k$ fixed let 
\begin{align}
d\cE_t^\eta=\prod_{j=1}^k \cE_t^{\eta_j}(z_j).
\end{align}
For $\nu\in(-\infty,0)$ we define $\cG_{\nu,t}$, $\cE_{\nu,t}^\eta$ similarly.

It\^o's formula together with \eqref{eq:cE} gives
\begin{align}\label{eq:E_prod}
d\mathcal E^\eta=\frac{i}{2}z \cdot \eta\, f_\beta \cE^\eta \, dt + \frac12 \cE^\eta \sum_{j=1}^k\eta_j(\cG(z_j)-1) (idb_1-db_2).
\end{align}
For $\nu\in(-\infty,0)$ the process $\cE^\eta_\nu$ satisfies \eqref{eq:E_prod} with initial condition $\cE^\eta_{\nu,\nu}=1$.

\begin{proposition}\label{prop:ratios}
Let $z\in \CC^k$, with $\Im z_j<0$ for all $j$, and let $\eta\in \{-1,1\}^k$.
Then for $-\infty<\nu<u$ we have
\begin{align}\label{eq:E_eta_mom}
    E\,\cE_{\nu,u}^\eta=\exp\Big(\tfrac{i}{2} z\cdot \eta\, (e^{\beta u/4}-e^{\beta \nu/4})\Big).
\end{align}
\end{proposition}
\begin{proof}
By \eqref{eq:cE_Polya} of Corollary \ref{cor:cE_SDE} we have $|\cG_{\nu,t}(z)|\le 1$ for $\Im z<0$. Hence we can use Lemma \ref{l:linear-sde} of the Appendix for the SDE \eqref{eq:E_prod} for $\cE^\eta_\nu$.  With $r(u)=E\cE^\eta_{\nu,u}$, $\nu<u$ we get the equation
\[
r(u)=1+\frac{i}{2} z \eta \int_\nu^u f_\beta(s) r(s) ds. 
\]
Solving the corresponding ODE gives \eqref{eq:E_eta_mom}.
\end{proof}



\subsection{Borodin-Strahov moment formulas for $\zeta$}\label{subs:BS_zeta}

Borodin and Strahov  \cite{BorodinStrahov}  compute  
\begin{align}\label{eq:BS_moment}
\lim_{n\to \infty}E\prod_{j=1}^k \frac{p_n(z_j/\sqrt{n})}{p_n(w_j/\sqrt{n})}=\begin{cases}
e^{i \sum_{j=1}^n \frac{z_j-w_j}{2}}&\qquad \text{ if } \Im w_j<0, 1\le j\le n,\\
e^{-i \sum_{j=1}^n \frac{z_j-w_j}{2}}&\qquad \text{ if } \Im w_j>0, 1\le j\le n,
\end{cases}
\end{align}
for the characteristic polynomial $p_n$ of the Gaussian beta ensemble for $\beta=1,2,4$. The answer does not depend on the value of $\beta$ if $\beta=1,2,4$. Borodin and Strahov \cite{BorodinStrahov}   pose the question whether this is true for  all $\beta>0$.

We will show that the analogous expectation does not depend on $\beta$  for the stochastic zeta function.  In the next section we also show this for the circular beta ensemble.

Chhaibi, Najnudel  and Nikhegbali \cite{CNN} show the formula analogous to \eqref{eq:BS_moment} for the characteristic polynomial of Haar unitary matrices and for $\zeta_2$. Chhaibi, Hovhannisyan, Najnudel, Nikeghbali,   and Rodgers \cite{chhaibi2019limiting} show that the normalized characteristic polynomial of the Gaussian unitary ensemble converges to $\zeta_2$.

We will need the following simple lemma.
\begin{lemma}\label{lem:Cauchy_ave}
Suppose that $a_j, b_j, c_j, d_j\in \CC, 1\le j\le k$. Let $q$ be a Cauchy distributed random variable. Then
\begin{align*}
E \prod_{j=1}^k \frac{a_j+q b_j}{c_j+q d_j}=\begin{cases}
\prod_{j=1}^k \frac{a_j+i b_j}{c_j+i d_j},& \qquad \text{if } \Im \frac{c_j}{d_j}>0 \text{ for all } 1\le j\le k,\\
\prod_{j=1}^k \frac{a_j-i b_j}{c_j-i d_j},& \qquad \text{if } \Im \frac{c_j}{d_j}<0 \text{ for all } 1\le j\le k.
\end{cases}
\end{align*}
\end{lemma}
\begin{proof}
For the first case, set
\[
r(x)=\prod_{j=1}^k \frac{a_j+x b_j}{c_j+x d_j}.
\]
We compute 
$$\int_{\infty}^\infty \frac{r(x)}{\pi(1+x^2)}dx=\lim_{r\to\infty}\int_{\gamma_r} \frac{r(z)}{\pi(1+z^2)}dz,$$ where $\gamma_r$ is the counterclockwise oriented curve constructed from the line segment $[-r,r]$ and the corresponding semicircle in the upper half plane. The function $\frac{r(z)}{\pi(1+z^2)}$  has  singularities at $\pm i$ and $-\frac{c_j}{d_j}$, and out of these only $i$ is in the upper  half plane. The residue of $\frac{r(z)}{\pi(1+z^2)}$ at $z=i$ is exactly $\frac{r(i)}{2\pi i}$.  This proves the first case, the second follows by conjugation. 
\end{proof}


Next, we show that  Borodin-Strahov conjecture holds for $\zeta_\beta$, Theorem \ref{thm:BS_conj-intro} of the Introduction. 


\begin{proof}[Proof of Theorem \ref{thm:BS_conj-intro}]
We  first consider the first case,   when $\Im w_j<0$.
We first prove the appropriate statement for the approximate versions of $\zeta$. Let $\nu\not=-\infty$.  

Recall that $\zeta_{\nu}(z)=\cA_{\nu,0}(z)-q \cB_{\nu,0}(z)$  where $q$ is a Cauchy distributed random variable independent of $\cA, \cB$. By \eqref{eq:cE_Polya} of Corollary \ref{cor:cE_SDE} we have 
\[
|\cA_{\nu,u}(z)-i \cB_{\nu,u}(z)|\ge |\cA_{\nu,u}(z)+i \cB_{\nu,u}(z)|, \qquad \text{for $\Im z<0$,}
\]
which implies 
$
\Im \frac{\cA_{\nu,0}(z)}{\cB_{\nu,0}(z)}<0$.
The product
$
\prod_{j=1}^k \left|\frac{\zeta_\nu(z_j)}{\zeta_\nu(w_j)}\right|
$
has finite expectation by Proposition \ref{prop:moment_ratio}.
Since $q$ is independent of $\cA, \cB$, Lemma \ref{lem:Cauchy_ave} implies
\begin{align*}
E\left[\prod_{j=1}^k \frac{\zeta_\nu(z_j)}{\zeta_\nu(w_j)}\big\vert \cA, \cB\right]=
\prod_{j=1}^k \frac{\cA_{\nu,0}(z_j)-i \cB_{\nu,0}(z_j)}{\cA_{\nu,0}(w_j)-i \cB_{\nu,0}(w_j)}=
\prod_{j=1}^k \frac{\cE_{\nu,0}(z_j)}{\cE_{\nu,0}(w_j)},& \quad \text{ if } \Im w_j<0   \text{ for all $j$,}
\end{align*}
and that the right hand side has finite expectation.


Assume first that we have $\Im z_j<0, \Im w_j<0$ for all $j$. Then by Proposition \ref{prop:ratios} we have 
\[
E\prod_{j=1}^k \frac{\zeta_\nu(z_j)}{\zeta_\nu(w_j)}=E\prod_{j=1}^k \frac{\cE_{\nu,0}(z_j)}{\cE_{\nu,0}(w_j)}=\exp\Big(\frac{i}{2}\sum_{k=1}^n(z_k-w_k) (1-e^{\beta \nu/4})\Big).
\]
Proposition \ref{prop:moment_ratio} shows that the products on the left are uniformly integrable as $\nu\to-\infty$. 
This implies the statement of the theorem in the case when $\Im z_j<0, \Im w_j<0$ for all $j$. 



The expected value of the product of the ratios is an entire function in each $z_j$ variable by Proposition \ref{prop:ratio_analytic}.  This extends the claim to $z_j\in \CC$, and proves the first case of the theorem.   The second case follows by conjugation. 
\end{proof}

\begin{remark}
If one could evaluate the function
\begin{align}
r_{\beta,n}(z_1, \dots, z_k, w_1, \dots, w_k)=E\prod_{j=1}^k \frac{\zeta(z_j)}{\zeta(w_j)} \label{eq:rho_beta}
\end{align}
for all choices of $z_j, w_j\in \CC$ then this would lead to the joint $n$-point correlation functions of the $\Sineb$ process. The $n$-point correlation function $\rho_{\beta,n}:\R^n\to \R$ would be given as 
\begin{align}\label{eq:n-corr}
\rho_{\beta,n}(\lambda_1, \dots, \lambda_n)=\left[\frac{\partial^n}{\partial {z_1}\dots \partial {z_n}}\Big|_{z=w} r_{\beta,n}(z_1, \dots, z_k, w_1, \dots, w_k)\right]_{w_1=\lambda_1, \dots, w_n=\lambda_n}
\end{align}
where $[f(\cdot)]_{x}=\lim_{\eps \to 0^+}\frac{1}{2\pi i}(f(x-i \eps)-f(x+i \eps))$, see \cite{BorodinStrahov}.

Theorem \ref{thm:BS_conj-intro} and \eqref{eq:n-corr} give $\rho_{\beta,1}(\lambda)=\frac{1}{2\pi}$, the intensity of the $\Sineb$ process. \end{remark}

\subsection{Borodin-Strahov moment formulas for the circular beta ensemble }

Let $\beta>0$ and let $p_n(z)$ be the  characteristic polynomial of a size $n$ circular beta ensemble as in \eqref{eq:char_pol}. The following theorem shows that the Borodin-Strahov moment formulas hold for this model even before taking the limit.
\begin{theorem}[Borodin-Strahov moment conjecture, circular beta case]\label{thm:finite-BStr}
\begin{align}\label{eq:finite-BStr}
E \prod_{j=1}^\ell \frac{p_n(e^{i z_j/n} ) e^{-i z_j/2}}{p_n(e^{i w_j/n}) e^{-i w_j/2}} =\begin{cases}
e^{i \sum_{j=1}^\ell \frac{z_j-w_j}{2}}&\qquad \text{ if } \Im w_j<0, 1\le j\le \ell,\\
e^{-i \sum_{j=1}^\ell \frac{z_j-w_j}{2}}&\qquad \text{ if } \Im w_j>0, 1\le j\le \ell.
\end{cases}
\end{align}
\end{theorem}
\begin{proof}
Consider the random discrete measure $\mu_{n,\beta}$ introduced in the proof of Theorem \ref{thm:char_conv}. The support of $\mu_{n,\beta}$ has circular beta distribution, and the weights are given by  an independent Dirichlet distribution with parameter $(\beta/2, \dots, \beta/2)$. 

Killip and Nenciu \cite{KillipNenciu}, see also \cite{BNR2009}, identify the joint distribution of the  modified Verblunsky coefficients $\gamma_{0}, \dots, \gamma_{n-1}$ corresponding to  $\mu_{n,\beta}$. They show that $\gamma_k$ are independent and rotational invariant, and  $|\gamma_k|^2$ has Beta$(1,\tfrac{\beta}{2}(n-k-1))$ distribution. In particular,  $\gamma_{n-1}$ is uniform on $\{|z|=1\}$. 

We will use the notation introduced in  Section \ref{s:Verblunsky}. Define $\cE_k(z), \cE_k^*(z)$ for $0\le k\le n$, $z\in \CC$ as  
\[
\bin{\cE_k(z)}{\cE_k^*(z)}=e^{-\frac{izk}{2n}} \bin{\psi_k(e^{iz/n})}{\psi_k^*(e^{iz/n})}=UX_kH_k.
\]
By \eqref{eq:mSzego} we have
\begin{align*}
\bin{\cE_{k+1}}{\cE_{k+1}^*}=A_k \mat{e^{\frac{iz}{2n}}}{0}{0}{e^{\frac{-iz }{2n}}}\bin{\cE_{k}}{\cE_{k}^*}
&=\mat{\frac{1}{1-\gamma_k }}{-\frac{\gamma_k }{1-\gamma_k }}{-\frac{\bar \gamma_k}{1-\bar \gamma_k }}{\frac{1 }{1-\bar \gamma_k }}\bin{e^{\frac{iz}{2n}} \cE_k}{e^{-\frac{iz}{2n}} \cE_k^* }, 
\end{align*}
which leads to 
\begin{align}\label{eq:cE_k_ratio}
\frac{\cE_{k+1}(z)}{\cE_{k+1}(w)}=e^{\frac{i(z-w)}{2n}}\frac{\cE_k(z)-\gamma_k e^{-\frac{i z}{n}}\cE_k^*(z)}{\cE_k(w)-\gamma_k e^{-\frac{i w}{n}}\cE_k^*(w)}.
\end{align}
If $\Im w<0$ then $|[1,-i] H_k(w)|\ge |[1,i] H_k(w)|$ by \eqref{E_AB}, which implies $|\cE_k(w)|\ge |\cE_k^*(w)|$ as well. 

Fix $w_j, z_j\in \CC$ with $1\le j\le \ell$, and assume that $\Im w_j<0$ for all $j$.
The random variable  $\gamma_k$ is independent of $\cE_j, \cE_j^*, j\le k$, and it has rotationally invariant distribution. Hence with
$
\mathcal{Q}_k:=\prod_{j=1}^\ell \frac{\cE_k(z_j)}{\cE_k(w_j)}
$
we have
\[
E\left[\mathcal Q_{k+1} \big\vert \mathcal Q_j, j\le k, |\gamma_k|=r\right]=e^{\frac{i}{2n}\sum_{j=1}^\ell (z_j-w_j)}\frac{1}{2\pi}\int_{0}^{2\pi} \prod_{j=1}^\ell\frac{\cE_k(z_j)-r e^{i t} e^{-\frac{i z_j}{n}}\cE_k^*(z_j)}{\cE_k(w_j)-r e^{i t}  e^{-\frac{i w_j}{n}}\cE_k^*(w_j)} dt.
\]
We have $|\cE_k(w_j)|\ge |\cE_k^*(w_j)|$ from $\Im w_j<0$, and thus by Lemma \ref{lem:average} below we get
\[
E\left[\mathcal Q_{k+1} \big\vert \mathcal Q_j, j\le k, |\gamma_k|=r\right]=e^{\frac{i}{2n}\sum_{j=1}^\ell (z_j-w_j)} \mathcal Q_k.
\]
Taking expectations and using $\mathcal Q_0=1$ we get
$
E\mathcal Q_n=e^{\frac{i}{2} \sum_{j=1}^\ell (z_j-w_j)}.$
Since
\[
\mathcal Q_n=\prod_{j=1}^\ell \frac{\cE_n(z_j)}{\cE_n(w_j)}=\prod_{j=1}^\ell\frac{\psi_n(e^{iz_j/n}) e^{-iz_j/2}}{\psi_n(e^{iw_j/n}) e^{-iw_j/2}},
\]
and $\psi_n=p_n$, the first case of \eqref{eq:finite-BStr} follows. The second case follows after conjugation.
\end{proof}

\begin{remark}
  The same proof works for random measures supported on $n$ points on the unit circle, with Verblunsky coefficients $\gamma_0,\dots, \gamma_{n-1}$ satisfying the following condition. Given $|\gamma_j|, 0\le j\le n-1$ the arguments of $\gamma_j$ are independent and uniformly distributed on $[0,2\pi]$. 
\end{remark}

The following lemma is related to  Lemma \ref{lem:Cauchy_ave}  through a Cayley transform. 

\begin{lemma}\label{lem:average}
Suppose that $a_j,b_j,c_j,d_j\in \CC$ with $c_j\neq 0$ and $|d_j|<|c_j|$. Then
\[
\frac{1}{2\pi}\int_0^{2\pi} \prod_{j=1}^k \frac{a_j-e^{it} b_j}{c_j-e^{it} d_j} dt=\prod_{j=1}^k \frac{a_j}{c_j}.
\]
\end{lemma}
\begin{proof}
We can rewrite the left hand side as a complex line integral on the unit circle as 
\[
\frac{1}{2\pi i}\oint\limits_{\bigodot} \frac{1}{z}\prod_{j=1}^k \frac{a_j-z b_j}{c_j-z d_j} dz.
\]
The integrand has poles at $0$ and at $\frac{c_j}{d_j}, 1\le j\le k$. Because of our conditions the only pole inside the unit circle is 0, and the residue is $\prod_j \frac{a_j}{c_j}$.
\end{proof}

\appendix

\section{Law of iterated logarithm for Brownian integrals}

\begin{theorem}\label{thm:lil} 
For every $a>2, b<1/2$ there is a constant $c$ so that  
\begin{align}
 P\Big(\sup_{t>0} B(t)^2/t-a \log(1+|\hspace{-0.2em}\log t|)>y\Big)\le  c e^{-by} \qquad \text{for all }y\ge 0.\label{eq:lil_tail}
\end{align}
\end{theorem}
 This is an effective small and large-time version of the upper bound in the law of iterated logarithm. By setting $t=1$ inside the $\sup$ we get $B(1)^2$, which shows that the rate of the  exponential decay cannot be more than $1/2$. The lower bound $2$ on the parameter $a$ is sharp by the usual law  of  iterated logarithm.  

\begin{proof} Let $f,g:(-1,\infty)\to \mathbb R$ be non-decreasing functions and $\eps\in(0,1)$. If $f(t)>g(t)$ for  $t\ge 0$ then  $f(s)\ge g(s-\eps)$ for $s\in [t, t+\eps]$. Hence
\begin{align}\label{eq:lil_ineq}
1(f(t)\ge g(t)\mbox{ for some } t\ge 0)\le \tfrac{1}{\eps} \int_0^\infty 1(f(s)\ge g(s-\eps))ds.
\end{align}
Let 
$\bar B_t=\max_{0\le s\le t} |B_s|$, and apply \eqref{eq:lil_ineq} to  $f(t)=\bar B(e^{t})^2$ and  $g(t)=e^{t}(y+a\log(1+t))$. The expectation of the resulting inequality bounds 
$P(\sup_{t\ge 1} B_t^2/t-a \log(1+\log t)>y)$ above as 
\begin{align}
P(\bar B(e^t)^2\ge g(t)\text{ for some $t\ge 0$})
  \le \tfrac{1}{\eps} \int_0^\infty P(\bar B(e^s)^2\ge e^{s-\eps}(y+a\log(1+s-\eps))ds. \label{eq:lil_bnd1}
\end{align}
Since $\max_{0\le r\le 1} B(r)$ is distributed as $|B(1)|$,  union and Gaussian tail bounds yield
 \[
P(\bar B(e^s)^2\ge e^s x)=P(\bar B(1)^2\ge x)\le 2P(B(1)^2\ge x)\le 2 e^{-x/2}.
\]
Thus the right hand side of \eqref{eq:lil_bnd1} is bounded above by  
\begin{align*}
 \tfrac{2}{\eps} \int_0^\infty  e^{- e^{-\eps}(y+a\log(1+s-\eps))/2} ds=\frac{  4(1-\eps)^{1-e^{-\eps} a/2}e^{-e^{-\eps} y/2}}{(a e^{-\eps}-2)\eps}.
\end{align*}
To make the last step valid and to get the required bound we need $a e^{-\eps}>2$ and $e^{-\eps}/2>b$, so we choose $\eps<\min(1, \log(a/2),\log(1/(2b))$.
Time inversion $B_t \to tB(1/t)$ gives the same bound for the supremum for $0< t\le 1$, from which \eqref{eq:lil_tail} follows. 
\end{proof}






We apply Theorem \ref{thm:lil} to estimate the growth of Brownian integrals. 

\begin{proposition}\label{prop:BR_mart}
Suppose that $B$ is two-sided Brownian motion and $x_u, u\le 0$ is adapted to the filtration generated by its increments. Assume further that there is a random variable $C$ and a constant $a>0$ so that 
$
|X_u|\le C e^{a u} 
$ for $u\le 0$.
Then 
$$
|\int_{-\infty}^u X_u dB |\le \frac{2C}{\sqrt{a}}  e^{a u} \left(Z+ \log(1+2a|u|)+|\hspace{-0.2em}\log a|  +\log(1+2|\hspace{-0.2em}\log C |)\right) \quad \text{for all } u\le 0,
$$
so the left hand side is well defined. With  an absolute constant $c$, the  random variable $Z$ satisfies 
\begin{align}\label{eq:Z_tail}
P(Z>y)\le ce^{- y}, \qquad y\ge 0. 
\end{align}
\end{proposition}
\begin{proof}
We have 
\begin{align}\notag
\int_{-\infty}^u X_s^2\, ds \le \int_{-\infty}^u C^2 e^{2a s} ds=\frac{C^2}{2a} e^{2a u}<\infty,
\end{align}
so the process
$
M_u=\int_{-\infty}^u X_u dB
$
is well defined. Moreover, 
\begin{equation}\label{eq:MM_bnd}
[M]_u\le \frac{C^2}{2a} e^{2a u}, \qquad \text{ for all } u\le 0.
\end{equation}
By the Dubins-Schwarz theorem there is a Brownian motion $W(x), x\ge 0$ so that $
M_u=W([M]_u)$.
Let 
$$Z=\tfrac13 \sup_{x>0} (W(x)^2/x-3 \log(1+|\hspace{-0.2em}\log x|)).$$ 
By Theorem \ref{thm:lil} this random variable satisfies \eqref{eq:Z_tail}, and for $u\le 0$ we have 
\[
M_u^2\le 3 [M]_u Z+3 [M]_u  (1+\log(1+|\hspace{-0.2em}\log [M]_u|)).
\]
The function $x(1+\log(1+|\hspace{-0.2em}\log x|))$ is increasing, so by \eqref{eq:MM_bnd} we also get
\begin{equation}\label{e:Mu2}
M_u^2\le \frac{3C^2}{2a} e^{2a u} \big(
Z+1+ \log(1+|\hspace{-0.2em}\log \frac{C^2}{2a} e^{2a u}|)
\big).
\end{equation}
For $x,y>0$ we have   
$|\hspace{-0.2em}\log (xy)|\le |\hspace{-0.2em}\log x|+|\hspace{-0.2em}\log y|$, $\log(1+x+y)\le \log (1+x)+\log(1+y)$, and $\log(1+x)\le x$. Using these bounds repeatedly we get
\begin{align*}
\log(1+|\hspace{-0.2em}\log \frac{C^2}{2a} e^{2a u}|)&\le \log(1+2 |\hspace{-0.2em}\log C|+ |\hspace{-0.2em}\log 2a|+2a |u|) \\
&\le \log(1+2a |u|)+\log 2+|\hspace{-0.2em}\log a| +\log(1+2|\hspace{-0.2em}\log C |).
\end{align*}
Take square roots in \eqref{e:Mu2} and use the inequality $\sqrt{1+y}\le 1+y$ for $y>0$. For $q$ fixed,  $Z+q$ satisfies the same tail bound as $Z$ with a different $c$. This implies the claim. 
\end{proof}

\section{Moment bounds for an almost linear SDE}

The following lemma is used when calculating moments of ratios of $\zeta$.

\begin{lemma}\label{l:linear-sde}
Consider the diffusion
\begin{align}\label{eq:linear-sde}
    dX=Y dt+Z dW,
\end{align}
with $|Y|\le a |X|$, $|Z|\le b |X|$, and  $a, b$ and $E|X_0^2|$ finite.
Then for any $t\ge 0$ 
\begin{align}\notag
    E|X_t|^2\le E|X_0|^2 e^{(2 a+2b^2)t}, \qquad 
    EX_t=EX_0+\int_0^t E Y_s ds. \label{eq:linear-sde-moment}
\end{align}
In particular, if $Y=\eta X$ for $\eta\in \CC$ then $EX_t=EX_0 e^{\eta t}$.
\end{lemma}
\begin{proof}
Let $\tau_c$ be the first time $|X_t|\ge c$. By It\^o's formula 
$$
\int_0^{\tau_c\wedge s}d|X|^2=\int_0^{\tau_c\wedge s}2 \Re(\bar X Z dW)+(2 \Re X \bar Y +2|Z|^2)dt.
$$
In the interval $[0,\tau_c\wedge s]$ the quadratic variation is bounded, so  
\begin{align*}
M_t=|X_{t\wedge \tau_c}|^2-\int_0^{t\wedge \tau_c} (2 \Re X_s \bar Y_s+2Z_s^2)\, ds
\end{align*}
is a martingale, and 
\begin{align}
 E|X_{t\wedge \tau_c}^2|=E\int_0^{t\wedge \tau_c} (2 \Re X_s \bar Y_s+2Z_s^2)\, ds  \le (2a+2b^2)\, E\int_0^{t\wedge \tau_c} |X|^2  ds.
\end{align}
Since
$$
\int_0^{t\wedge \tau_c} |X|^2  ds=\int_0^{t} |X|^2 \ind(s\le \tau_c)  ds \le \int_0^{t} |X|_{s\wedge \tau_c}^2  ds
$$
we have
$
 E|X_{t\wedge \tau_c}^2|\le E|X_0^2|e^{(2a+2b^2) t}
$
by Gronwall's inequality. Fatou's lemma gives
\[
E|X_t^2|\le E|X_0^2|e^{(2a+2b^2) t}
\]
as well. From this we see that the quadratic variation of 
$
\int_0^t Z dW
$
has finite expectation, so it is a martingale. Thus we can take expectations in \eqref{eq:linear-sde}  which gives (\ref{eq:linear-sde-moment}). 
The last claim follows from solving the  equation 
$
EX_t=EX_0+\eta \int_0^t EX_s ds.
$
\end{proof}

\noindent {\bf Acknowledgments.}
The first author was partially supported by the NSF award DMS-1712551. The second author was supported by the Canada Research Chair program, the NSERC Discovery Accelerator grant, and the MTA Momentum Random Spectra research group. We thank Lucas Ashbury-Bridgwood, Alexei Borodin, Paul Bourgade, Reda Chhaibi, L\'aszl\'o Erd\H os, Yun Li, Vadim Gorin, and Alexei Poltoratski  for useful discussions. 

\addcontentsline{toc}{section}{References}


\def\cprime{$'$}

\bigskip\noindent
Benedek Valk\'o
\\Department of Mathematics
\\University of Wisconsin - Madison
\\Madison, WI 53706, USA
\\{\tt valko@math.wisc.edu}
\\[20pt]
B\'alint Vir\'ag
\\Departments of Mathematics and Statistics
\\University of Toronto
\\Toronto ON~~M5S 2E4, Canada
\\{\tt balint@math.toronto.edu}

\end{document}